\documentclass[a4paper]{amsart}
\usepackage{amscd}
\usepackage{amsmath}
\usepackage{amssymb}
\usepackage{amsthm}
\usepackage{bbm} 
\usepackage{mathrsfs}
\usepackage{stmaryrd}
\usepackage{tikz-cd}
\usepackage[arrow, matrix, curve]{xy}

\usepackage[T1]{fontenc}
\usepackage[utf8]{inputenc}

\usepackage{lipsum}

\makeatletter
\renewcommand\subsection{\@startsection{subsection}{2}%
  \z@{-.5\linespacing\@plus-.7\linespacing}{.5\linespacing}%
  {\normalfont\scshape}}
\renewcommand\subsubsection{\@startsection{subsubsection}{3}%
  \z@{.5\linespacing\@plus.7\linespacing}{-.5em}%
  {\normalfont\scshape}}
\makeatother

\usepackage[T1]{fontenc}

\makeatletter
\@namedef{subjclassname@2020}{%
  \textup{2020} Mathematics Subject Classification}
\makeatother

\numberwithin{equation}{section} \swapnumbers

\newtheorem{satz}{Satz}[section]

\newtheorem{theorem}[satz]{Theorem}
\newtheorem{proposition}[satz]{Proposition}
\newtheorem{corollary}[satz]{Corollary}
\newtheorem{lemma}[satz]{Lemma}
\newtheorem{assumption}[satz]{Assumption}

\newtheorem{definition}[satz]{Definition}

\newtheorem{remark}[satz]{Remark}

\newtheorem{example}[satz]{Example}
\newtheorem{examples}[satz]{Examples}

\newcommand{\bbr}{\mathbb{R}}
\newcommand{\bbe}{\mathbb{E}}
\newcommand{\bbn}{\mathbb{N}}
\newcommand{\bbp}{\mathbb{P}}

\newcommand{\bbs}{\mathbb{S}}
\newcommand{\bbb}{\mathbb{B}}

\newcommand{\cala}{\mathscr{A}}
\newcommand{\calb}{\mathscr{B}}

\newcommand{\calf}{\mathscr{F}}

\newcommand{\calk}{\mathscr{K}}
\newcommand{\call}{\mathscr{L}}
\newcommand{\calm}{\mathscr{M}}
\newcommand{\caln}{\mathscr{N}}

\newcommand{\cals}{\mathscr{S}}

\newcommand{\supp}{{\rm supp}}
\newcommand{\Id}{{\rm Id}}

\newcommand{\lin}{{\rm lin}}

\newcommand{\rk}{{\rm rk}}

\newcommand{\tr}{{\rm tr}}

\newcommand{\bbI}{\mathbbm{1}}

\newcommand{\la}{\langle}
\newcommand{\ra}{\rangle}

\newcommand{\IL}{[\![}

\begin{document}

\hyphenation{rea-li-za-tion pa-ra-me-tri-za-ti-ons pa-ra-me-tri-za-ti-on Schau-der sub-ma-ni-fold}

\title[Invariant manifolds in continuously embedded Hilbert spaces]{Invariant manifolds for stochastic partial differential equations in continuously embedded Hilbert spaces}
\author{Rajeev Bhaskaran \and Stefan Tappe}
\address{Indian Statistical Institute, 8th Mile, Mysore Road, Bangalore 560 059, India}
\email{brajeev@isibang.ac.in}
\address{Albert Ludwig University of Freiburg, Department of Mathematical Stochastics, Ernst-Zermelo-Stra\ss{}e 1, D-79104 Freiburg, Germany}
\email{stefan.tappe@math.uni-freiburg.de}
\date{6 April, 2022}
\thanks{Rajeev Bhaskaran gratefully acknowledges financial support from the Science and Engineering Research Board (SERB), India -- vide grant CRG/2019/002594. Stefan Tappe gratefully acknowledges financial support from the Deutsche Forschungsgemeinschaft (DFG, German Research Foundation) -- project number 444121509, and from the Deutsche Mathematiker-Vereinigung (DMV, German Mathematical Society) -- Fachgruppe Stochastik.}
\begin{abstract}
We provide necessary and sufficient conditions for stochastic invariance of finite dimensional submanifolds  for solutions of stochastic partial differential equations (SPDEs) in continuously embedded Hilbert spaces with non-smooth coefficients. Furthermore, we establish a link between invariance of submanifolds for such SPDEs in Hermite Sobolev spaces and invariance of submanifolds for finite dimensional SDEs. This provides a new method for analyzing stochastic invariance of submanifolds for finite dimensional It\^{o} diffusions, which we will use in order to derive new invariance results for finite dimensional SDEs.
\end{abstract}
\keywords{Stochastic partial differential equation, continuously embedded Hilbert spaces, invariant manifold, finite dimensional diffusion, multi-parameter group, Hermite Sobolev space, translation invariant solution}
\subjclass[2020]{60H15, 60H10, 60G17}

\maketitle\thispagestyle{empty}

\section{Introduction}

The problem of finding invariant submanifolds of solutions of stochastic partial differential equations (SPDEs) arises, for example, in connection with stochastic models in finance wherein the submanifolds offer the possibility of finite dimensional realizations of the solutions which are otherwise infinite dimensional (see, for example \cite{Bj_Sv, Bj_La, Bjoerk, Filipovic-Teichmann, Filipovic-Teichmann-royal, Tappe-Wiener, Tappe-Levy, Tappe-affin-real, Tappe-affine}). The problem, related to the computability of ``interest rate term structure models'', is also known as the ``consistency problem'' for such models; see \cite{fillnm}. In this paper we study the mathematical problem of finding invariant submanifolds for a general class of SPDEs that includes apart from quasi-semilinear and semilinear SPDEs (see, for example \cite{Filipovic-inv, Nakayama, Tappe-affin-real}) a more recent class of SPDEs studied in \cite{Rajeev, Rajeev-2019}. We will refer to this latter class as It\^{o} type SPDEs.

In this paper, we develop a general framework, which covers the aforementioned types of SPDEs, and we present an invariance result for finite dimensional submanifolds, which generalizes existing results in this direction. In particular, the usual assumption that the volatilities must be smooth, is not required in our framework (see Theorem \ref{thm-SPDE}). Furthermore, we establish a link between invariance of submanifolds for such SPDEs in Hermite Sobolev spaces and invariance of submanifolds for finite dimensional SDEs (see Theorem \ref{thm-M-N}). Using this connection, we will also contribute new invariance results for finite dimensional SDEs (see, in particular Theorems \ref{thm-inv-b} and \ref{thm-nullstellen}). As we will see, our results are stable under the dimension of the driving noise, which may in particular be infinite dimensional.

In order to outline our findings, let $(G,H)$ be a pair of continuously embedded separable Hilbert spaces; this means that $G \subset H$ as sets, and that the embedding operator from $(G,\| \cdot \|_G)$ into $(H,\| \cdot \|_H)$ is continuous. Consider an SPDE of the form 
\begin{align}\label{SPDE}
\left\{
\begin{array}{rcl}
dY_t & = & L(Y_t) dt + A(Y_t) dW_t
\\ Y_0 & = & y_0
\end{array}
\right.
\end{align}
driven by a $\bbr^{\infty}$-Wiener process $W$ with continuous coefficients $L : G \to H$ and $A : G \to \ell^2(H)$; we refer to Section \ref{sec-SPDE} for further details. We emphasize that SPDEs of the type (\ref{SPDE}) in particular cover the following two types of SPDEs:
\begin{itemize}
\item Semilinear SPDEs of the type
\begin{align}\label{SPDE-semi-intro}
\left\{
\begin{array}{rcl}
dY_t & = & ( B Y_t + \alpha(Y_t) ) dt + \sigma(Y_t) dW_t
\\ Y_0 & = & y_0,
\end{array}
\right.
\end{align}
where $B : H \supset D(B) \to H$ is a densely defined, closed operator, and $\alpha : H \to H$ and $\sigma : H \to \ell^2(H)$ are continuous mappings. Here the Hilbert space $G$ is given by the domain $G := D(B)$, equipped with the graph norm
\begin{align}\label{graph-norm-intro}
\| y \|_G = \sqrt{ \| y \|_H^2 + \| B y \|_H^2 }, \quad y \in G,
\end{align}
and the coefficients in (\ref{SPDE}) are given by $L = B + \alpha$ and $A = \sigma$. This includes SPDEs in the framework of the semigroup approach (see, for example \cite{Da_Prato, Atma-book}), which also arise for the modeling of interest rate curves. We refer to Section \ref{sub-sec-semilinear} for more details.

\item The above mentioned It\^{o} type SPDEs (see \cite{Rajeev, Rajeev-2019}), where the pair $(G,H)$ of continuously embedded Hilbert spaces is given by Hermite Sobolev spaces $G = \cals_{p+1}(\bbr^d)$ and $H = \cals_p(\bbr^d)$ for some $p \in \bbr$, and the coefficients $L : G \to H$ and $A : G \to \ell^2(H)$ are given by second and first order differential operators of the form
\begin{align}\label{L-intro}
L(y) &:= \frac{1}{2} \sum_{i,j=1}^d ( \langle \sigma,y \rangle \langle \sigma,y \rangle^{\top} )_{ij} \partial_{ij}^2 y - \sum_{i=1}^d \langle b_i,y \rangle \partial_i y,
\\ \label{A-intro} A^j(y) &:= - \sum_{i=1}^d \langle \sigma_{i}^j,y \rangle \partial_i y, \quad j \in \bbn.
\end{align}
where $b_i \in \cals_{-(p+1)}(\bbr^d)$ for $i=1,\ldots,d$ and $\sigma_i^j \in \cals_{-(p+1)}(\bbr^d)$ for $i=1,\ldots,d$ and $j \in \bbn$, and where $\la \cdot,\cdot \ra$ denotes the dual pair on $\cals_{-(p+1)}(\bbr^d) \times \cals_{p+1}(\bbr^d)$. We refer to Section \ref{sec-examples-HS} for further details.
\end{itemize}
Let $\calm \subset H$ be a finite dimensional $C^2$-submanifold of $H$. We are interested in local invariance of $\calm$, which means that for each starting point $y_0 \in \calm$ there exists a local solution $Y$ to the SPDE (\ref{SPDE}) with $Y_0 = y_0$ such that $Y^{\tau} \in \calm$, where the positive stopping time $\tau > 0$ denotes the lifetime of $Y$. Let us first recall a known result for semilinear SPDEs of the type (\ref{SPDE-semi-intro}). If $\sigma^j \in C^1(H)$ for each $j \in \bbn$, then the following statements are equivalent:
\begin{enumerate}
\item[(i)] $\calm$ is locally invariant for the semilinear SPDE (\ref{SPDE-semi-intro}).

\item[(ii)] We have
\begin{align}\label{semi-cond-intro-1}
&\calm \subset D(B),
\\ \label{semi-cond-intro-2} &\sigma^j|_{\calm} \in \Gamma(T \calm), \quad j \in \bbn,
\\ \label{semi-cond-intro-3} &B|_{\calm} + \alpha|_{\calm} - \frac{1}{2} \sum_{j=1}^{\infty} D \sigma^j \cdot \sigma^j|_{\calm} \in \Gamma(T \calm).
\end{align}
\end{enumerate}
Here $\Gamma(T \calm)$ denotes the space of all vector fields on $\calm$; that it, the space of all mappings $A : \calm \to H$ such that $A(y) \in T_y \calm$ for each $y \in \calm$, where $T_y \calm$ denotes the tangent space to $\calm$ at $y$. Furthermore, for each $j \in \bbn$ we denote by $D \sigma^j \cdot \sigma^j|_{\calm}$ the mapping $y \mapsto D \sigma^j(y) \sigma^j(y)$, $y \in \calm$.

For this result we refer to \cite{Filipovic-inv, Nakayama}; see also \cite{FTT-manifolds}, where the more general situation with jump-diffusions and submanifolds with boundary has been treated. In \cite{Filipovic-inv}, the conditions (\ref{semi-cond-intro-2}) and (\ref{semi-cond-intro-3}) above are called ``Nagumo type consistency'' conditions. However the term $\frac{1}{2} \sum_{j=1}^{\infty} D \sigma^j \cdot \sigma^j$ in condition (\ref{semi-cond-intro-3}) can also be viewed as a ``Stratonovich'' correction term, which requires smoothness of the volatilities $\sigma^j$, $j \in \bbn$.

When dealing with the more general SPDE (\ref{SPDE}), the smoothness of the coefficients $A^j$, $j \in \bbn$ becomes problematic, since they are defined between two different Hilbert spaces $A^j : G \to H$. In particular, for It\^{o} type SPDEs with coefficients of the form (\ref{L-intro}) and (\ref{A-intro}), the volatilities $A^j$, $j \in \bbn$ are typically not of class $C^1$ (see Remark \ref{rem-non-smoothness}). Therefore, one of the principal challenges that we deal with in this paper is to find a suitable generalization of condition (\ref{semi-cond-intro-3}) for these SPDEs.

This leads to a geometric framework where we consider $(G,H)$-submanifolds. More precisely, a $C^2$-submanifold $\calm$ of $H$ is called a $(G,H)$-submanifold of class $C^2$ if $\calm \subset G$ and $\tau_H \cap \calm = \tau_G \cap \calm$, where $\tau_H$ and $\tau_G$ denote the topologies of $H$ and $G$. In our main result we will show that for such a submanifold $\calm$ the following statements are equivalent:
\begin{enumerate}
\item[(i)] $\calm$ is locally invariant for the SPDE (\ref{SPDE}).

\item[(ii)] We have
\begin{align}\label{tang-A-intro}
&A^j|_{\calm} \in \Gamma(T \calm), \quad j \in \bbn,
\\ \label{tang-L-intro} &[ L|_{\calm} ]_{\Gamma(T \calm)} - \frac{1}{2} \sum_{j=1}^{\infty} [A^j|_{\calm}, A^j|_{\calm}]_{\calm} = [0]_{\Gamma(T \calm)}.
\end{align}
\end{enumerate}
We refer to Theorem \ref{thm-SPDE} for the precise result and further details. The condition (\ref{tang-L-intro}) is an equation in the quotient space $A(\calm) / \Gamma(T \calm)$, where $A(\calm)$ denotes the space of all mappings $A : \calm \to H$. Furthermore, for each $j \in \bbn$ the element $[A^j|_{\calm}, A^j|_{\calm}]_{\calm}$ arises from the quadratic variation term in It\^{o}'s formula, when we realize the solutions $Y$ of the SPDE (\ref{SPDE}) on $\calm$ as the image $Y = \phi(X)$ of a finite dimensional process $X$ and a local parametrization $\phi : V \rightarrow U \cap \calm$ of the submanifold $\calm$; we refer to Definition \ref{def-A-Strat} for more details. The advantage in this formulation is clearly that it does not require smoothness of the vector fields $A^j$, $j \in \bbn$, which is also seen in subsequent results; see, for example Theorem \ref{thm-main-3}.

In particular, our main result applies to semilinear SPDEs of the type (\ref{SPDE-semi-intro}), where $\sigma$ is only assumed to be continuous. Recalling that $G = D(B)$ endowed with the graph norm (\ref{graph-norm-intro}), in this situation we will show that for a finite dimensional $C^2$-submanifold $\calm$ of $H$ the following statements are equivalent:
\begin{enumerate}
\item[(i)] $\calm$ is locally invariant for the semilinear SPDE (\ref{SPDE-semi-intro}).

\item[(ii)] $\calm$ is a $(G,H)$-submanifold of class $C^2$, which is locally invariant for the semilinear SPDE (\ref{SPDE-semi-intro}).

\item[(iii)] $\calm$ is a $(G,H)$-submanifold of class $C^2$, and we have
\begin{align}\label{semi-intro-4}
&\sigma^j|_{\calm} \in \Gamma(T \calm), \quad j \in \bbn,
\\ \label{semi-intro-5} &[ (B + \alpha)|_{\calm} ]_{\Gamma(T \calm)} - \frac{1}{2} \sum_{j=1}^{\infty} [\sigma^j|_{\calm},\sigma^j|_{\calm}]_{\calm} = [0]_{\Gamma(T \calm)}.
\end{align}
\end{enumerate}
Furthermore, if $\sigma^j \in C^1(H)$ for each $j \in \bbn$, then condition (\ref{semi-intro-5}) is equivalent to (\ref{semi-cond-intro-3}). We refer to Theorem \ref{thm-semi-lin} and Remark \ref{rem-semi-lin} for further details. These findings are a consequence a more general result for so-called quasi-semilinear SPDEs, which we establish in this paper; see Theorem \ref{thm-quasi-semilinear}.

Note that in the aforementioned result for semilinear SPDEs we only assume that $\calm$ is a finite dimensional $C^2$-submanifold of $H$, whereas in our main result we assume that $\calm$ is a $(G,H)$-submanifold of class $C^2$. Indeed, as the previous equivalences (i)--(iii) show, for semilinear SPDEs the submanifold $\calm$ is automatically a $(G,H)$-submanifold in case of local invariance, which is due to the fact that $G = D(B)$ endowed with the graph norm (\ref{graph-norm-intro}).

Our main result also applies to It\^{o} type SPDEs (\ref{SPDE}), where the coefficients are of the form (\ref{L-intro}) and (\ref{A-intro}), and where we recall that $G = \cals_{p+1}(\bbr^d)$ and $H = \cals_p(\bbr^d)$ for some $p \in \bbr$. Then, for any $\Phi \in G$ the submanifold
\begin{align*}
\calm = \{ \tau_x \Phi : x \in \bbr^d \}
\end{align*}
is locally invariant for the SPDE (\ref{SPDE}), where $(\tau_x)_{x \in \bbr^d}$ denotes the group of translation operators on $H$. This shows that the solutions to the It\^{o} type SPDE (\ref{SPDE}) are translation invariant; that is, we have $Y = \tau_X \Phi$ for some $\bbr^d$-valued diffusion $X$; see also \cite{Rajeev}.

We will generalize this result to SPDEs (\ref{SPDE}) with a general pair of continuously embedded Hilbert spaces $(G,H)$ be as follows. Let $T = (T(t))_{t \in \bbr^d}$ be a multi-parameter $C_0$-group on $H$, let $\caln$ be an $m$-dimensional $C^2$-submanifold of $\bbr^d$ for some $m \leq d$, and consider the submanifold
\begin{align}\label{structure-intro}
\calm = \{ T(t) y_0 : t \in \caln \}
\end{align}
for some $y_0 \in G$. Denoting by $\psi : \bbr^d \to H$ the orbit map $\psi(t) := T(t) y_0$ for $t \in \bbr^d$, we will show that the following statements are equivalent:
\begin{enumerate}
\item[(i)] $\calm$ is locally invariant for the SPDE (\ref{SPDE}).

\item[(ii)] $\caln$ is locally invariant for the $\bbr^d$-valued SDE
\begin{align*}
\left\{
\begin{array}{rcl}
dX_t & = & \bar{b}(X_t) dt + \bar{\sigma}(X_t) dW_t
\\ X_0 & = & x_0,
\end{array}
\right.
\end{align*}
where $\bar{\sigma} : \caln \to \ell^2(\bbr^d)$ and $\bar{b} : \caln \to \bbr^d$ are the unique solutions of the equations
\begin{align}\label{L-Rajeev-intro}
L|_{\calm} &= \frac{1}{2} \sum_{i,j=1}^d (\bar{\sigma} \bar{\sigma}^{\top})_{ij} \circ \psi^{-1}|_{\calm} \, B_{ij}|_{\calm} + \sum_{i=1}^d \bar{b}_i \circ \psi^{-1}|_{\calm} \, B_i|_{\calm},
\\ \label{A-Rajeev-intro} A^j|_{\calm} &= \sum_{i=1}^d \bar{\sigma}_{i}^j \circ \psi^{-1}|_{\calm} \, B_i|_{\calm}, \quad j \in \bbn.
\end{align}
\end{enumerate}

We refer to Theorem \ref{thm-SPDE-group} for the precise statement. Note that the structures of the coefficients in (\ref{L-intro}) and (\ref{A-intro}) are particular cases of (\ref{L-Rajeev-intro}) and (\ref{A-Rajeev-intro}). This result is a consequence of a more general result for arbitrary $(G,H)$-submanifolds, which we establish in this paper; see Theorem \ref{thm-inv-embedded}. Moreover, we will show that the structure (\ref{structure-intro}) of the submanifold $\calm$ appears naturally with coefficients of the kind (\ref{L-Rajeev-intro}) and (\ref{A-Rajeev-intro}) in case of local invariance; see Theorem \ref{thm-structure-manifold} for the precise result.

Diffusions on manifolds in $\mathbb R^d$ is a well studied topic (see for a partial list \cite{Ito-manifold, Elworthy, Emery, Hsu, Ikeda, Stroock}). In this paper, we will also establish new results concerning the invariance of finite dimensional submanifolds for $\bbr^d$-valued diffusions of the type
\begin{align}\label{SDE-b-sigma-intro}
\left\{
\begin{array}{rcl}
dX_t & = & b(X_t)dt + \sigma(X_t) dW_t
\\ X_0 & = & x_0
\end{array}
\right.
\end{align}
with coefficients $b : \bbr^d \to \bbr^d$ and $\sigma : \bbr^d \to \ell^2(\bbr^d)$. Our essential assumption is that these coefficients belong to a Hermite Sobolev space with sufficient regularity. More precisely, we assume that for some $q > \frac{d}{4}$ we have $b_i \in \cals_q(\bbr^d)$ for $i=1,\ldots,d$ and $\sigma_i^j \in \cals_q(\bbr^d)$ for $i=1,\ldots,d$ and $j \in \bbn$. Let $\caln$ be an $m$-dimensional $C^2$-submanifold of $\bbr^d$ for some $m \leq d$. We set $G := \cals_{-q}(\bbr^d)$, $H := \cals_{-(q+1)}(\bbr^d)$, define the coefficients of the SPDE (\ref{SPDE}) as (\ref{L-intro}), (\ref{A-intro}) with $p := -(q+1)$, and consider the submanifold
\begin{align*}
\calm := \{ \delta_x : x \in \caln \},
\end{align*}
where $\delta_x$ denotes the Dirac distribution at point $x$. Then the following statements are equivalent:
\begin{enumerate}
\item[(i)] $\calm$ is locally invariant for the SPDE (\ref{SPDE}).

\item[(ii)] $\caln$ is locally invariant for the SDE (\ref{SDE-b-sigma-intro}).
\end{enumerate}
We refer to Theorem \ref{thm-M-N}, which establishes the announced link between the invariance of submanifolds for SPDEs in Hermite Sobolev spaces and the invariance of submanifolds for finite dimensional SDEs. In particular, in some situations it turns out that locally invariance of $\calm$ for the SPDE (\ref{SPDE}) is easier to prove, which is the key for providing new invariance results for finite dimensional SDEs.

One application of this connection appears in the situation, where we consider the conditions
\begin{align}\label{conditions-1-intro}
b|_{\caln} &\in \Gamma(T \caln), 
\\ \label{conditions-2-intro} \sigma^j|_{\caln} &\in \Gamma(T \caln), \quad j \in \bbn,
\end{align}
and where we are interested in finding an additional condition ensuring that $\caln$ is locally invariant for the SDE (\ref{SDE-b-sigma-intro}). In this regard, we will show that under conditions (\ref{conditions-1-intro}) and (\ref{conditions-2-intro}) the following conditions are equivalent:
\begin{enumerate}
\item[(i)] $\caln$ is locally invariant for the SDE (\ref{SDE-b-sigma-intro}).

\item[(ii)] We have
\begin{align*}
\sum_{j=1}^{\infty} \big( [A^j|_{\calm},A^j|_{\calm}]_{\calm} - [ \bar{A}^j(A^j(\cdot),\cdot)|_{\calm} ]_{\Gamma(T \calm)} \big) = [0]_{\Gamma(T \calm)},
\end{align*}
where, in accordance with (\ref{A-intro}), we have set
\begin{align*}
\bar{A}^j(y,z) := - \sum_{i=1}^d \langle \sigma_{i}^j,z \rangle \partial_i y, \quad j \in \bbn.
\end{align*}
\end{enumerate}
We refer to Theorem \ref{thm-inv-b} for further details. A consequence of this result is that the conditions
\begin{align*}
b|_{\caln} &\in \Gamma(T \caln),
\\ \sigma^j|_{\caln} &\in \Gamma^*(T \caln), \quad j \in \bbn,
\end{align*}
where $\Gamma^*(T \caln)$ denotes the space of all locally simultaneous vector fields on $\caln$, are sufficient for local invariance of $\caln$ for the SDE (\ref{SDE-b-sigma-intro}); see Proposition \ref{prop-inv-fin-dim}. This is a generalization of the result that an affine submanifold $\caln$ is locally invariant if and only if we have (\ref{conditions-1-intro}) and (\ref{conditions-2-intro}). We also establish such a result in the general framework for SPDEs of the type (\ref{SPDE}); see Corollary \ref{cor-affine}.

Another application of the connection between invariance of submanifolds for SDEs and SPDEs occurs in the situation, where the submanifold $\caln$ is given by the zeros of smooth functions. More precisely, we assume that the dimension of $\caln$ is given by $m = d-n$, where $n < d$, and that there exist an open subset $O \subset \bbr^d$ and a mapping $f : \bbr^d \to \bbr^n$ such that
\begin{align*}
\caln = \{ x \in O : f(x) = 0 \}.
\end{align*}
Concerning the components of $f$ we assume that $f_k \in \cals_{q+1}(\bbr^d)$ for all $k=1,\ldots,n$. As we will show, then the following statements are equivalent:
\begin{enumerate}
\item[(i)] The submanifold $\caln$ is locally invariant for the SDE (\ref{SDE-b-sigma-intro}).

\item[(ii)] For all $k = 1,\ldots,n$ we have
\begin{align*}
\bigg( \sum_{i=1}^d b_i \partial_i f_k + \frac{1}{2} \sum_{i,j=1}^d ( \sigma \sigma^{\top} )_{ij} \partial_{ij}^2 f_k \bigg) \bigg|_{\caln} &= 0,
\\ \sum_{i=1}^d \sigma_{i}^j \partial_i f_k \bigg|_{\caln} &= 0, \quad j \in \bbn.
\end{align*}
\end{enumerate}
For this result, we refer to Theorem \ref{thm-nullstellen}. We illustrate the latter result with the example of the unit sphere $\bbs^{d-1}$ (Corollary \ref{cor-unit-sphere} and Example \ref{example-Stroock}) and recover an earlier result of Stroock.

The paper is organized as follows: Section \ref{sec-SPDE} introduces SPDEs in the framework of continuously embedded Hilbert spaces. In Section \ref{sec-manifolds} we introduce the notion of a submanifold in embedded Hilbert spaces. Section \ref{sec-manifolds-general} is devoted to calculus on such submanifolds. In Section \ref{sec-manifolds-HS} we consider submanifolds generated by the infinitesimal generator of a multi-parameter strongly continuous group. In Section \ref{sec-SPDE-general} we present our main result concerning invariant manifolds. Afterwards, in Section \ref{sec-quasi-semilinear} we present consequences for quasi-semilinear SPDEs, which includes the particular case of semilinear SPDEs. In Section \ref{sec-quasi-linear} we study the invariance of manifolds which are generated by orbit maps; this includes It\^{o} type SPDEs. In Section \ref{sec-interplay} we provide the link between the invariance of submanifolds for finite dimensional SDEs and the invariance of submanifolds for SPDEs in Hermite Sobolev spaces, and provide new invariance results for finite dimensional SDEs. Appendix \ref{app-groups} is devoted to results on multi-parameter groups and Appendix \ref{app-HS-spaces} to results on Hermite Sobolev spaces. This includes a proof of the Sobolev embedding theorem for Hermite Sobolev spaces.

\section{Stochastic partial differential equations in continuously embedded Hilbert spaces}\label{sec-SPDE}

In this section we provide the required prerequisites about SPDEs in continuously embedded Hilbert spaces.

\begin{definition}
We call $W = (W^j)_{j \in \bbn}$ a \emph{standard $\bbr^{\infty}$-Wiener process} if $(W^j)_{j \in \bbn}$ is a sequence of independent real-valued standard Wiener processes on some stochastic basis.
\end{definition}

For a Hilbert space $H$ we denote by $\ell^2(H)$ the Hilbert space of all $H$-valued sequences $y = (y^j)_{j \in \bbn}$ such that
\begin{align*}
\| y \|_{\ell^2(H)} := \bigg( \sum_{j=1}^{\infty} \| y^j \|_H^2 \bigg)^{1/2} < \infty.
\end{align*}

\begin{proposition}
Let $W = (W^j)_{j \in \bbn}$ be a standard $\bbr^{\infty}$-Wiener process on a stochastic basis $(\Omega,\calf,(\calf_t)_{t \in \bbr_+},\bbp)$, let $H$ be a separable Hilbert space, and let $A$ be a predictable $\ell^2(H)$-valued process such that we have $\bbp$-almost surely
\begin{align}\label{A-integrability}
\int_0^t \| A_s \|_{\ell^2(H)}^2 ds < \infty, \quad t \in \bbr_+.
\end{align}
Then the process $(\int_0^t A_s dW_s)_{t \in \bbr_+}$ given by
\begin{align}\label{integral-series}
\int_0^t A_s dW_s := \sum_{j=1}^{\infty} \int_0^t A_s^j dW_s^j, \quad t \in \bbr_+
\end{align}
is a well-defined $H$-valued continuous local martingale, and the convergence is in probability, uniformly on compact intervals.
\end{proposition}

\begin{proof}
Let $T > 0$ be arbitrary. We denote by $M_T^2(H)$ the space of all $H$-valued square-integrable martingales $M = (M_t)_{t \in [0,T]}$, which, endowed with the norm
\begin{align*}
\| M \|_{\infty} = \bbe \bigg[ \sup_{t \in [0,T]} \| M_t \|_H^2 \bigg]^{1/2}, \quad M \in M_T^2(H)
\end{align*}
is a Hilbert space. Furthermore, by Doob's martingale inequality, an equivalent norm is given by
\begin{align*}
\| M \|_T = \bbe \big[ \| M_T \|_H^2 \big]^{1/2}, \quad M \in M_T^2(H).
\end{align*}
Concerning the predictable process $A$, we first suppose that
\begin{align*}
\bbe \bigg[ \int_0^T \| A_s \|_{\ell^2(H)}^2 ds \bigg] < \infty.
\end{align*}
Then by the It\^{o} isometry and the monotone convergence theorem we have
\begin{align*}
\sum_{j=1}^{\infty} \bbe \Bigg[ \bigg\| \int_0^T A_s^j d W_s^j \bigg\|_H^2 \Bigg] = \sum_{j=1}^{\infty} \bbe \bigg[ \int_0^T \| A_s^j \|_H^2 d s \bigg] = \bbe \bigg[ \int_0^T \| A_s \|_{\ell^2(H)}^2 ds \bigg] < \infty,
\end{align*}
and hence the series $\sum_{j=1}^{\infty} \int_0^T A_s^j d W_s^j$ converges in $M_T^2(H)$. The situation with a general predictable process $A$ satisfying (\ref{A-integrability}) follows by localization, and, by the definition of the norm $\| \cdot \|_{\infty}$, the convergence is in probability, uniformly on compact intervals.
\end{proof}

\begin{definition}
Let $G$ and $H$ be two normed spaces. Then we call $(G,H)$ \emph{continuously embedded normed spaces} (or \emph{normed spaces with continuous embedding}) if the following conditions are fulfilled:
\begin{enumerate}
\item We have $G \subset H$ as sets.

\item The embedding operator $\Id : (G, \| \cdot \|_G) \to (H, \| \cdot \|_H)$ is continuous; that is, there is a constant $K > 0$ such that
\begin{align*}
\| x \|_H \leq K \| x \|_G \quad \text{for all $x \in G$.}
\end{align*}
\end{enumerate}
\end{definition}

\begin{definition}
Let $H_1,\ldots,H_n$ be normed spaces. Then we call $(H_1,\ldots,H_n)$ \emph{continuously embedded normed spaces} if for each $k=1,\ldots,n-1$ the pair $(H_k,H_{k+1})$ is a pair of continuously embedded normed spaces.
\end{definition}

Now, let $(G,H)$ be separable Hilbert spaces with continuous embedding. Furthermore, let $L : G \to H$ and $A : G \to \ell^2(H)$ be continuous\footnote{More precisely, here and in the sequel, we call a mapping $L : G \to H$ \emph{continuous} if $L : (G,\| \cdot \|_G) \to (H,\| \cdot \|_H)$ is continuous. The continuity of $A$ is understood analogously.} mappings. Then for each $j \in \bbn$ the component $A^j : G \to H$ is continuous.

\begin{definition}\label{def-martingale-solution}
Let $y_0 \in G$ be arbitrary. A triplet $(\bbb,W,Y)$ is called a \emph{local martingale solution} to the SPDE (\ref{SPDE}) with $Y_0 = y_0$ if the following conditions are fulfilled:
\begin{enumerate}
\item $\bbb = (\Omega,\calf,(\calf_t)_{t \in \bbr_+},\bbp)$ is a stochastic basis; that is, a filtered probability space satisfying the usual conditions.

\item $W$ is a standard $\bbr^{\infty}$-Wiener process on the stochastic basis $\bbb$.

\item $Y$ is a $G$-valued adapted\footnote{See Remark \ref{rem-adapted} for details about this notion.} process such that for some strictly positive stopping time $\tau > 0$ we have $\bbp$-almost surely
\begin{align}\label{SPDE-int-cond}
\int_0^{t \wedge \tau} \big( \| L(Y_s) \|_H + \| A(Y_s) \|_{\ell^2(H)}^2 \big) ds < \infty, \quad t \in \bbr_+
\end{align}
and $\bbp$-almost surely
\begin{align}\label{SPDE-integral-form}
Y_{t \wedge \tau} = y_0 + \int_0^{t \wedge \tau} L(Y_s) ds + \int_0^{t \wedge \tau} A(Y_s) dW_s, \quad t \in \bbr_+,
\end{align}
where the stochastic integral is defined according to (\ref{integral-series}). The stopping time $\tau$ is also called the \emph{lifetime} of $Y$.
\end{enumerate}
If we can choose $\tau = \infty$, then $(\bbb,W,Y)$ is also called a \emph{global martingale solution} (or simply a \emph{martingale solution}) to the SPDE (\ref{SPDE}) with $Y_0 = y_0$.
\end{definition}

\begin{remark}
As it is apparent from the integrability condition (\ref{SPDE-int-cond}), the stochastic integrals appearing in (\ref{SPDE-integral-form}) are understood as stochastic integrals in the Hilbert space $(H,\| \cdot \|_H)$. Therefore, the right-hand side of (\ref{SPDE-integral-form}) is generally $H$-valued, whereas the left-hand side is $G$-valued. This indicates that the existence of martingale solutions to the SPDE (\ref{SPDE}) can generally not be warranted. If there exists a martingale solution $Y$, then its sample paths are continuous with respect to the norm $\| \cdot \|_H$, but they do not need to be continuous with respect to the norm $\| \cdot \|_G$.
\end{remark}

\begin{remark}\label{rem-adapted}
Let $\bbb$ be a stochastic basis. In our situation, there are two reasonable ways to define what it means that a $G$-valued process $Y$ is \emph{adapted}; namely:
\begin{enumerate}
\item We regard $Y$ as a process taking its values in the subspace $G$ of the Hilbert space $(H,\| \cdot \|_H)$ and call it \emph{adapted} if for each $t \in \bbr_+$ the mapping $Y_t : \Omega \to G$ is $\calf_t$-$\calb(H)_G$-measurable, where $\calb(H)_G$ denotes the trace $\sigma$-algebra
\begin{align*}
\calb(H)_G = \{ B \cap G : B \in \calb(H) \}.
\end{align*}

\item We regard $Y$ as a process taking its values in the Hilbert space $(G,\| \cdot \|_G)$ and call it \emph{adapted} if for each $t \in \bbr_+$ the mapping $Y_t : \Omega \to G$ is $\calf_t$-$\calb(G)$-measurable.
\end{enumerate}
However, by Kuratowski's theorem (see, for example \cite[Thm. I.3.9]{Parthasarathy}) we have $\calb(G) = \calb(H)_G$, showing that these two concepts of adaptedness are equivalent.
\end{remark}

\begin{remark}\label{rem-SPDEs}
The SPDE (\ref{SPDE}) can also be realized as an SPDE driven by a trace class Wiener process, as considered, for example in \cite{Da_Prato, Atma-book}. Indeed, let $U$ be a separable Hilbert space, and let $\bar{W}$ be an $U$-valued $Q$-Wiener process for some nuclear, self-adjoint, positive definite linear operator $Q \in L_1^{++}(U)$; see, for example \cite[Def. 4.2]{Da_Prato}. There exist an orthonormal basis $\{ e_j \}_{j \in \bbn}$ of $U$ and a sequence $(\lambda_j)_{j \in \bbn} \subset (0,\infty)$ with $\sum_{j \in \bbn} \lambda_j < \infty$ such that
\begin{align*}
Q e_j = \lambda_j e_j \quad \text{for all $j \in \bbn$.}
\end{align*}
The space $U_0 := Q^{1/2}(U)$, equipped with the inner product
\begin{align*}
\langle u,v \rangle_{U_0} := \langle Q^{-1/2}u, Q^{-1/2}v \rangle_U, \quad u,v \in U_0
\end{align*}
is another separable Hilbert space. We fix the orthonormal basis $\{ g_j \}_{j \in \bbn}$ of $U_0$ given by $g_j := \sqrt{\lambda_j} e_j$ for each $j \in \bbn$, and we denote by $L_2^0(H) := L_2(U_0,H)$ the space of all Hilbert-Schmidt operators from $U_0$ into $H$. Note that $L_2^0(H) \cong \ell^2(H)$, because $T \mapsto (T g_j)_{j \in \bbn}$ is an isometric isomorphism. By \cite[Prop. 4.3]{Da_Prato} the sequence $(\bar{W}^j)_{j \in \bbn}$ defined as
\begin{align*}
\bar{W}^j := \frac{1}{\sqrt{\lambda_j}} \langle \bar{W},e_j \rangle_U, \quad j \in \bbn
\end{align*}
is a sequence of independent real-valued standard Wiener processes. Hence $W = (\bar{W}^j)_{j \in \bbn}$ is a standard $\bbr^{\infty}$-Wiener process. As a consequence of the series representation of the stochastic integral with respect to the trace class Wiener process $\bar{W}$ (see, for example \cite[Prop. 2.4.5]{Liu-Roeckner}), the SPDE (\ref{SPDE}) can be expressed as
\begin{align}\label{SPDE-trace-class}
\left\{
\begin{array}{rcl}
dY_t & = & L(Y_t) dt + \bar{A}(Y_t) d\bar{W}_t
\\ Y_0 & = & y_0
\end{array}
\right.
\end{align}
where the continuous mapping $\bar{A} : G \to L_2^0(H)$ is given by
\begin{align*}
\bar{A}(y) := \sum_{j=1}^{\infty} \la \bullet, g_j \ra_{U_0} \, A^j(y), \quad y \in G,
\end{align*}
and, vice versa, the SPDE (\ref{SPDE-trace-class}) can be expressed by the SPDE (\ref{SPDE}), where the continuous mapping $A : G \to \ell^2(H)$ is given by
\begin{align*}
A(y) := (\bar{A}(y) g_j)_{j \in \bbn}, \quad y \in G.
\end{align*}
\end{remark}

\begin{remark}\label{rem-global-weak-solution}
In the particular case $G = H = \bbr^d$ the SPDE (\ref{SPDE}) is rather an SDE, and a martingale solution $(\bbb,W,Y)$ is a weak solution. If, in this case, the continuous mappings $L : \bbr^d \to \bbr^d$ and $A : \bbr^d \to \ell^2(\bbr^d)$ satisfy the linear growth condition, then for each $y_0 \in \bbr^d$ there exists a global weak solution $(\bbb,W,Y)$ to the SDE (\ref{SPDE}) with $Y_0 = y_0$. Indeed, taking into account Remark \ref{rem-SPDEs}, this follows from \cite[Thm. 2]{GMR} (or \cite[Thm. 3.12]{Atma-book}), applied with $H = H_{-1} = \bbr^m$ and $J = \Id_{\bbr^m}$.
\end{remark}

\begin{remark}
The situation where the Wiener process $W$ is $\bbr^r$-valued is covered by choosing $A^j \equiv 0$ for all $j > r$. If we are additionally in the situation of Remark \ref{rem-global-weak-solution}, then the existence of global weak solutions also follows from \cite[Thms. IV.2.3 and IV.2.4]{Ikeda}.
\end{remark}

\begin{remark}
If there is no ambiguity, we will simply call $Y$ a local martingale solution or a global martingale solution to the SPDE (\ref{SPDE}) with $Y_0 = y_0$.
\end{remark}

Now, let $\calm \subset G$ be a subset. In this paper, the subset $\calm$ will typically be a finite dimensional submanifold.

\begin{definition}
The subset $\calm$ is called \emph{locally invariant} for the SPDE (\ref{SPDE}) if for each $y_0 \in \calm$ there exists a local martingale solution $Y$ to the SPDE (\ref{SPDE}) with $Y_0 = y_0$ and lifetime $\tau > 0$ such that $Y^{\tau} \in \calm$ up to an evanescent set\footnote{A random set $A \subset \Omega \times \mathbb{R}_+$ is called \emph{evanescent} if the set $\{ \omega \in \Omega : (\omega,t) \in A \text{ for some } t \in \mathbb{R}_+ \}$ is a $\mathbb{P}$-nullset, cf. \cite[1.1.10]{Jacod-Shiryaev}.}.
\end{definition}

\begin{definition}
The subset $\calm$ is called \emph{globally invariant} (or simply \emph{invariant}) for the SPDE (\ref{SPDE}) if for each $y_0 \in \calm$ there exists a global martingale solution $Y$ to the SPDE (\ref{SPDE}) with $Y_0 = y_0$ such that $Y \in \calm$ up to an evanescent set.
\end{definition}

\section{Finite dimensional submanifolds in embedded Hilbert spaces}\label{sec-manifolds}

In this section we provide the required background about finite dimensional submanifolds in embedded Hilbert spaces. It is divided into two parts. In Section \ref{sec-manifolds-general} we provide the preliminaries about submanifolds in Hilbert spaces, and later on we introduce submanifolds in embedded Hilbert spaces. In Section \ref{sec-manifolds-HS} we deal with submanifolds given by orbit maps of group actions, in particular in Hermite Sobolev spaces.

\subsection{Finite dimensional submanifolds in Hilbert spaces}\label{sec-manifolds-general}

In this section we deal with finite dimensional submanifolds in Hilbert spaces. Let $H$ be a Hilbert space. Furthermore, let $m \in \bbn$ and $k \in \overline{\bbn}$ be positive integers. Here we use the notation $\overline{\bbn} = \bbn \cup \{ \infty \}$, which means that $k = \infty$ is allowed.

\begin{definition}\label{def-immersion}
Let $V \subset \mathbb{R}^m$ be an open subset, and let $\phi \in C^k(V;H)$ be a mapping.
\begin{enumerate}
\item Let $x_0 \in V$ be arbitrary. The mapping $\phi$ is called a \emph{$C^k$-immersion at $x_0$} if $D \phi(x_0) \in L(\bbr^m,H)$ is one-to-one.

\item The mapping $\phi$ is called a \emph{$C^k$-immersion} if it is a $C^k$-immersion at $x_0$ for each $x_0 \in V$. 
\end{enumerate}
\end{definition}

\begin{definition}
A subset $\calm \subset H$ is called an $m$-dimensional \emph{$C^k$-submanifold} of $H$ if for every $y \in \calm$ there exist an open neighborhood $U \subset H$ of $y$, an open set $V \subset \mathbb{R}^m$ and a mapping $\phi \in C^k(V;H)$ such that:
\begin{enumerate}
\item The mapping $\phi : V \rightarrow U \cap \calm$ is a homeomorphism.

\item $\phi$ is a $C^k$-immersion.
\end{enumerate}
The mapping $\phi$ is called a \emph{local parametrization} of $\calm$ around $h$.
\end{definition}

For what follows, let $\calm$ be an $m$-dimensional $C^k$-submanifold of $H$.

\begin{lemma}\cite[Lemma~6.1.1]{fillnm}\label{lemma-change-para}
Let $\phi_i : V_i \rightarrow U_i \cap \calm$, $i=1,2$ be two local parametrizations of $\calm$ with $W := U_1 \cap U_2 \cap \calm \neq \emptyset$. Then the mapping
\begin{align*}
\varphi := \phi_1^{-1} \circ \phi_2 : \phi_2^{-1}(W) \rightarrow \phi_1^{-1}(W)
\end{align*}
is a $C^k$-diffeomorphism.
\end{lemma}

\begin{definition}\label{def-tang-raum}
Let $y \in \calm$ be arbitrary. The \emph{tangent space} of $\calm$ to $y$ is the subspace
\begin{align*}
T_y \calm := D \phi(x)\mathbb{R}^m,
\end{align*}
where $x := \phi^{-1}(y)$, and $\phi : V \rightarrow U \cap \calm$ denotes a local parametrization of $\calm$ around $y$.
\end{definition}

\begin{remark}
By Lemma \ref{lemma-change-para} the Definition \ref{def-tang-raum} of the tangent space does not depend on the choice of the parametrization.
\end{remark}

\begin{remark}
Let $y \in \calm$ be arbitrary, and let $\phi : V \rightarrow U \cap \calm$ be a local parametrization of $\calm$ around $y$. Then $D \phi(x) \in L(\bbr^m,T_y \calm)$ is a linear isomorphism, where $x := \phi^{-1}(y) \in V$.
\end{remark}

\begin{remark}
Let $U \subset H$ be an open subset such that $U \cap \calm \neq \emptyset$. Then $\calm_U := U \cap \calm$ is also an $m$-dimensional $C^k$-submanifold of $H$, and we have
\begin{align*}
T_y \calm_U = T_y \calm \quad \text{for all $y \in \calm_U$.}
\end{align*}
\end{remark}

\begin{definition}\label{def-tangent-bundle}
The \emph{tangent bundle} of $\calm$ is defined as
\begin{align*}
T \calm := \bigsqcup_{y \in \calm} T_y \calm := \{ (y,z) : y \in \calm \text{ and } z \in T_y \calm \}.
\end{align*}
\end{definition}

\begin{definition}\label{def-vector-field}
A mapping $A : \calm \to H$ is called a \emph{vector field} on $\calm$ if
\begin{align*}
A(y) \in T_y \calm \quad \text{for each $y \in \calm$,}
\end{align*}
that is
\begin{align*}
\{ (y,A(y)) : y \in \calm \} \subset T \calm.
\end{align*}
We denote by $\Gamma(T \calm)$ the space of all vector fields on $\calm$.
\end{definition}

\begin{definition}\label{def-local-vector-field}
Let $z \in \calm$ be arbitrary. A mapping $A : \calm \to H$ is called a \emph{local vector field} on $\calm$ around $z$ if there is an open neighborhood $U \subset H$ of $z$ such that
\begin{align*}
A(y) \in T_y \calm \quad \text{for each $y \in U \cap \calm$,}
\end{align*}
that is
\begin{align*}
\{ (y,A(y)) : y \in U \cap \calm \} \subset T \calm.
\end{align*}
We denote by $\Gamma_z(T \calm)$ the space of all local vector fields on $\calm$ around $z$.
\end{definition}

\begin{definition}\label{def-simultaneous}
A mapping $A : \calm \to H$ is called a \emph{locally simultaneous vector field} on $\calm$ if for each $y \in \calm$ there exists an open neighborhood $U \subset H$ of $y$ such that
\begin{align*}
A(y) \in T_z \calm \quad \text{for each $z \in U \cap \calm$,}
\end{align*}
that is
\begin{align*}
\{ (z,A(y)) : z \in U \cap \calm \} \subset T \calm.
\end{align*}
We denote by $\Gamma^*(T \calm)$ the space of all locally simultaneous vector fields on $\calm$.
\end{definition}

\begin{definition}\label{def-push}
Let $\phi : V \to U \cap \calm$ be a local parametrization of $\calm$.
\begin{enumerate}
\item For a mapping $a : V \to \bbr^m$ we define $\phi_* a : U \cap \calm \to H$ as
\begin{align*}
(\phi_* a)(y) := D \phi(x) a(x), \quad y \in U \cap \calm,
\end{align*}
where $x := \phi^{-1}(y) \in V$. 

\item Similarly, for two mappings $a,b : V \to \bbr^m$ we define $\phi_{**}(a,b) : U \cap \calm \to H$ as
\begin{align*}
(\phi_{**}(a,b))(y) := D^2 \phi(x) ( a(x), b(x) ), \quad y \in U \cap \calm,
\end{align*}
where $x := \phi^{-1}(y) \in V$. 

\item Setting $\calm_U := U \cap \calm$, for a vector field $A \in \Gamma(T \calm_U)$ we define $\phi_*^{-1} A : V \to \bbr^m$ as
\begin{align*}
(\phi_*^{-1} A)(x) := D \phi(x)^{-1} A(y), \quad x \in V,
\end{align*}
where $y := \phi(x) \in U \cap \calm$.
\end{enumerate}
\end{definition}

\begin{remark}
Note that for every vector field $A \in \Gamma(T \calm)$ and every local parametrization $\phi : V \to U \cap \calm$ there exists a unique mapping $a : V \to \bbr^m$ such that $A|_{U \cap \calm} = \phi_* a$.
\end{remark}

\begin{proposition}\label{prop-linear-inverse}
Let $D \subset H$ be a dense subset. Furthermore, let $y_0 \in \calm$ be arbitrary. There exist a local parametrization $\phi : V \to U \cap \calm$ around $y_0$ and a bounded linear operator $\psi \in L(H,\bbr^m)$ of the form
\begin{align*}
\psi = \la \zeta,\cdot \ra_H := \big( \la \zeta_1,\cdot \ra_H, \ldots, \la \zeta_m,\cdot \ra_H \big)
\end{align*}
with $\zeta_1,\ldots,\zeta_m \in D$ such that $\phi^{-1} = \psi|_{U \cap \calm}$ and we have
\begin{align}\label{inverse-phi}
D \psi(y)|_{T_y \calm} = D \phi(x)^{-1} \quad \text{for all $y \in U \cap \calm$,}
\end{align}
where $x := \psi(y) \in V$.
\end{proposition}

\begin{proof}
This follows from \cite[Prop. 6.1.2 and Lemma 6.1.3]{fillnm}.
\end{proof}

As we will see now, tangent spaces can also be characterized by means of curves.

\begin{definition}
Every mapping $\gamma \in C^1((-\epsilon,\epsilon);H)$ for some $\epsilon > 0$ with $\gamma(t) \in \calm$ for all $t \in (-\epsilon,\epsilon)$ is called a \emph{curve}.
\end{definition}

\begin{proposition}\label{prop-tang-curve}
For each $y \in \calm$ we have
\begin{align*}
T_y \calm = \{ w \in H : \gamma(0) = y \text{ and } \gamma'(0) = w \text{ for some curve } \gamma : (-\epsilon,\epsilon) \to \calm \}.
\end{align*}
\end{proposition}

\begin{proof}
Let $w \in T_y \calm$ be arbitrary. Furthermore, let $\phi : V \to U \cap \calm$ be a local parametrization around $y$, and set $x := \phi^{-1}(y) \in V$. There exists $v \in \bbr^m$ such that $w = D \phi(x) v$. Moreover, there exists $\epsilon > 0$ such that $x + t v \in V$ for each $t \in (-\epsilon,\epsilon)$. Hence, the curve $\gamma : (-\epsilon,\epsilon) \to U \cap \calm$ given by
\begin{align*}
\gamma(t) := \phi(x+tv), \quad t \in (-\epsilon,\epsilon)
\end{align*}
is well-defined, and we have $\gamma(0) = \phi(x) = y$ as well as $\gamma'(0) = D \phi(x) v = w$.

Now, let $w \in H$ be such that $\gamma(0) = y$ and $\gamma'(0) = w$ for some curve $\gamma : (-\epsilon,\epsilon) \to \calm$. By Proposition \ref{prop-linear-inverse} there exist a local parametrization $\phi : V \to U \cap \calm$ around $y$ and a bounded linear operator $\psi \in L(H,\bbr^m)$ such that $\phi^{-1} = \psi|_{U \cap \calm}$. We may assume that $\epsilon > 0$ is small enough such that $\gamma(t) \in U \cap \calm$ for all $t \in (-\epsilon,\epsilon)$. Now, we define $c : (-\epsilon,\epsilon) \to V$ as $c := \psi \circ \gamma$. Then we have $c \in C^1((-\epsilon,\epsilon);\bbr^m)$ with $c(0) = \psi(y) = x$ as well as $\gamma = \phi \circ c$. Therefore, we obtain
\begin{align*}
w = \gamma'(0) = D \phi(x) c'(0) \in T_y \calm, 
\end{align*}
completing the proof.
\end{proof}

In the next result we consider the particular situation of submanifolds in Euclidean space which are the zeros of smooth functions.

\begin{lemma}\label{lemma-tang-nullstellen}
Let $\calm$ be a $(d-n)$-dimensional $C^k$-submanifold of $\bbr^d$, where $d,n \in \bbn$ are such that $n < d$. Suppose there exist an open subset $O \subset \bbr^d$ and a mapping $f \in C^1(O;\bbr^n)$ such that
\begin{align*}
\calm = \{ y \in O : f(y) = 0 \}.
\end{align*}
Let $y \in \calm$ be such that $Df(y)\bbr^d = \bbr^n$. Then we have
\begin{align*}
T_y \calm = \ker Df(y).
\end{align*}
\end{lemma}

\begin{proof}
Let $w \in T_y \calm$ be arbitrary. By Proposition \ref{prop-tang-curve} there exists a curve $\gamma : (-\epsilon,\epsilon) \to \calm$ such that $\gamma(0) = y$ and $\gamma'(0) = w$. We have $f(\gamma(t)) = 0$ for all $t \in (-\epsilon,\epsilon)$. Therefore, we have $\frac{d}{dt} f(\gamma(t)) = 0$ for all $t \in (-\epsilon,\epsilon)$, and hence, in particular
\begin{align*}
0 = (f \circ \gamma)'(0) = Df(\gamma(0)) \gamma'(0) = D f(y) w,
\end{align*}
proving the inclusion
\begin{align*}
T_y \calm \subset \ker Df(y).
\end{align*}
Moreover, by the rank-nullity theorem we have $\dim \ker Df(y) = d-n$, completing the proof.
\end{proof}

For two normed spaces $X$ and $Y$ and an integer $n \in \bbn$ we denote by $L^n(X,Y)$ the space of all continuous $n$-multilinear maps $T : X^n \to Y$.

\begin{lemma}\label{lemma-second-order-chain}
Let $X,Y,Z$ be normed spaces, and let $U \subset X$ and $V \subset Y$ be open subsets. Let $f \in C^2(U;Y)$ with $f(U) \subset V$ and $g \in C^2(V;Z)$ be mappings. Then we have $g \circ f \in C^2(U;Z)$, and for each $x \in U$ we have
\begin{align*}
\underbrace{D^2 (g \circ f)(x)}_{\in L^{2}(X,Z)} = \underbrace{D^2 g(f(x))}_{\in L^{2}(Y,Z)} \circ \underbrace{(Df(x),Df(x))}_{\in L(X,Y) \times L(X,Y)} + \underbrace{Dg(f(x))}_{\in L(Y,Z)} \circ \underbrace{D^2 f(x)}_{\in L^{2}(X,Y)}.
\end{align*}
\end{lemma}

\begin{proof}
This follows from the higher order chain rule; see \cite[pages 87, 88]{Abraham}.
\end{proof}

Now, we turn back to the situation where $\calm$ be an $m$-dimensional $C^k$-submanifold of the Hilbert space $H$.

\begin{lemma}\label{lemma-second-tang}
Suppose $k \geq 2$. Let $\phi_i : V_i \rightarrow U_i \cap \calm$, $i=1, 2$ be two local parametrizations of $\calm$ with $U \cap \calm \neq \emptyset$, where $U := U_1 \cap U_2$. Then for all $y \in U \cap \calm$ and $w_1,w_2 \in T_y \calm$ we have
\begin{align*}
D^2 \phi_1(x_1) ( v_1, v_2 ) - D^2 \phi_2(x_2) ( u_1, u_2 ) \in T_y \calm,
\end{align*}
where $x_i := \phi_i^{-1}(y) \in V_i$ and $v_i := D \phi_1(x_1)^{-1} w_i, u_i := D \phi_2(x_2)^{-1} w_i \in \bbr^m$ for $i=1, 2$.
\end{lemma}

\begin{proof}
By assumption we have $W := U \cap \calm \neq \emptyset$. Thus, by Lemma~\ref{lemma-change-para} the mapping
\begin{align*}
\varphi := \phi_1^{-1} \circ \phi_2 : \phi_2^{-1}(W) \rightarrow \phi_1^{-1}(W)
\end{align*}
is a $C^k$-diffeomorphism. By the usual chain rule, we have
\begin{align*}
D \phi_2(x_2) = D(\phi_1 \circ \varphi)(x_2) = D \phi_1(x_1) D \varphi(x_2),
\end{align*}
and hence
\begin{align*}
D \varphi(x_2) = D \phi_1(x_1)^{-1} D \phi_2 (x_2).
\end{align*}
Therefore, by the second order chain rule (Lemma \ref{lemma-second-order-chain}) we obtain
\begin{align*}
&D^2 \phi_2(x_2)(u_1,u_2) = D^2 (\phi_1 \circ \varphi)(x_2)(u_1,u_2)
\\ &= D^2 \phi_1(x_1) ( D \varphi(x_2) u_1, D \varphi(x_2) u_2 ) + D \phi_1(x_1) D^2 \varphi(x_2)(u_1,u_2)
\\ &= D^2 \phi_1(x_1) ( v_1, v_2 ) + D \phi_1(x_1) D^2 \varphi(x_2)(u_1,u_2).
\end{align*}
Since $D \phi_1(x_1) D^2 \varphi(x_2)(u_1,u_2) \in T_y \calm$, this completes the proof.
\end{proof}

Now, let $G$ be another Hilbert space such that $(G,H)$ is a pair of continuously embedded Hilbert spaces. Denoting by $\tau_G$ and $\tau_H$ the respective topologies, we have $\tau_H \cap G \subset \tau_G$. Recall that $\calm$ denotes an $m$-dimensional $C^k$-submanifold of $H$. If $\calm \subset G$, then we have $\tau_H \cap \calm \subset \tau_G \cap \calm$.

\begin{definition}\label{def-manifold-G}
We call $\calm$ an $m$-dimensional $(G,H)$-submanifold of class $C^k$ if $\calm \subset G$ and $\tau_H \cap \calm = \tau_G \cap \calm$.
\end{definition}

\begin{proposition}\label{prop-manifold-scaled}
Let $\calm$ be an $m$-dimensional $C^k$-submanifold of $H$. Then the following statements are equivalent:
\begin{enumerate}
\item[(i)] $\calm$ is a $(G,H)$-submanifold of class $C^k$.

\item[(ii)] $\calm \subset G$ and the identity $\Id : (\calm, \| \cdot \|_H) \to (\calm, \| \cdot \|_G)$ is continuous.

\item[(iii)] $\calm \subset G$ and the identity $\Id : (\calm, \| \cdot \|_H) \to (\calm, \| \cdot \|_G)$ is a homeomorphism.

\item[(iv)] $\calm \subset G$ and each local parametrization $\phi : V \to U \cap \calm$\footnote{More precisely, here in the following statements we mean a local parametrization of the $C^k$-submanifold $\calm$ of $H$.} is also a homeomorphism $\phi : V \to (U \cap \calm,\| \cdot \|_G)$.

\item[(v)] Each local parametrization $\phi : V \to U \cap \calm$ satisfies $\phi \in C(V;G) \cap C^k(V;H)$.

\item[(vi)] For each $y \in \calm$ there exists a local parametrization $\phi : V \to U \cap \calm$ around $y$, which satisfies $\phi \in C(V;G) \cap C^k(V;H)$.

\item[(vii)] $\calm \subset G$ and for each $y \in \calm$ there exists a local parametrization $\phi : V \to U \cap \calm$ around $y$, which is also a homeomorphism $\phi : V \to (U \cap \calm,\| \cdot \|_G)$.
\end{enumerate}
If any of the previous conditions is fulfilled, then we have $T \calm \subset G \times H$.
\end{proposition}

\begin{proof}
The equivalences (i) $\Leftrightarrow$ (ii) $\Leftrightarrow$ (iii) are obvious.

\noindent(iii) $\Rightarrow$ (iv): By hypothesis, $\phi : V \to (U \cap \calm, \| \cdot \|_H) \to (U \cap \calm, \| \cdot \|_G)$ is a homeomorphism.

\noindent(iv) $\Rightarrow$ (v) $\Rightarrow$ (vi): These implications are obvious.

\noindent(vi) $\Rightarrow$ (vii): $\phi : V \to (U \cap \calm, \| \cdot \|_H) \to (U \cap \calm, \| \cdot \|_G)$ is a homeomorphism, because $\phi \in C(V;G)$.

\noindent(vii) $\Rightarrow$ (iii): Let $y \in \calm$ be arbitrary, and let $\phi : V \to U \cap \calm$ be a local parametrization  around $y$, which is also a homeomorphism $\phi : V \to (U \cap \calm,\| \cdot \|_G)$. Then the restricted identity 
\begin{align*}
\Id|_{U \cap \calm} : (U \cap \calm, \| \cdot \|_H) \overset{\phi^{-1}}{\longrightarrow} V \overset{\phi}{\longrightarrow} (U \cap \calm, \| \cdot \|_G)
\end{align*}
is homeomorphism.

\noindent The additional statement $T \calm \subset G \times H$ is a direct consequence of (vi).
\end{proof}

\begin{remark}
If $\calm$ is an $m$-dimensional $(G,H)$-submanifold of class $C^k$, then, according to Proposition \ref{prop-manifold-scaled}, it is also an $m$-dimensional topological submanifold of $G$. 
\end{remark}

Recall that $(G,H)$ denotes a pair of continuously embedded Hilbert spaces, and that $\calm$ is an $m$-dimensional $C^k$-submanifold of $H$. Now, let $(H_0,H_1,\ldots,H_{k-1},H_k)$ be continuously embedded Hilbert spaces such that $G = H_0$ and $H = H_k$.

\begin{definition}
We call $\calm$ an $m$-dimensional $(H_0,\ldots,H_k)$-submanifold of class $C^k$ if the following conditions are fulfilled:
\begin{enumerate}
\item $\calm$ is an $m$-dimensional $(G,H)$-submanifold of class $C^k$.

\item $\calm$ is an $m$-dimensional $C^j$-submanifold of $H_j$ for each $j=1,\ldots,k$.
\end{enumerate}
\end{definition}

\begin{proposition}\label{prop-degrees}
The following statements are equivalent:
\begin{enumerate}
\item[(i)] $\calm$ is an $m$-dimensional $(H_0,\ldots,H_k)$-submanifold of class $C^k$.

\item[(ii)] $\calm$ is an $m$-dimensional $(G,H_j)$-submanifold of class $C^j$ for each $j=1,\ldots,k$.

\item[(iii)] For each $y \in \calm$ there exists a local parametrization $\phi : V \to U \cap \calm$\footnote{Also here we mean a local parametrization of the $C^k$-submanifold $\calm$ of $H$.} around $y$, which satisfies $\phi \in \bigcap_{j=0}^k C^j(V;H_j)$.
\end{enumerate}
If any of the previous conditions is fulfilled, then we have $T \calm \subset G \times H_1$.
\end{proposition}

\begin{proof}
(i) $\Rightarrow$ (ii): This implication follows because $\tau_H \cap \calm = \tau_G \cap \calm$ implies $\tau_{H_j} \cap \calm = \tau_G \cap \calm$ for all $j=1,\ldots,k$.

\noindent (ii) $\Rightarrow$ (iii): Let $y \in \calm$ be arbitrary, and let $\phi_k : V_k \to U_k \cap \calm$ be a local parametrization of the $C^k$-submanifold $\calm$ around $y$. By Proposition \ref{prop-manifold-scaled} we have $\phi_k \in C(V_k;G) \cap C^k(V_k;H)$. Let $j \in \{ 1,\ldots,k-1 \}$ be arbitrary, and let $\phi_j : V_j \to U_j \cap \calm$ be a local parametrization of the $C^j$-submanifold $\calm$ of $H_j$ around $y$. Of course, $\calm$ is also a $C^j$-submanifold of $H$, and $\phi_k : V_k \to U_k \cap \calm$ is a local parametrization around $y$. The mapping $\phi_j : V_j \to U_j \cap \calm$ is also such a local parametrization around $y$, because
\begin{align*}
\phi_j : V_j \to ( U_j \cap \calm, \| \cdot \|_{H_j} ) \to ( U_j \cap \calm, \| \cdot \|_H )
\end{align*}
is a homeomorphism, because $\tau_H \cap \calm = \tau_{H_j} \cap \calm$. Setting $W_j := U_k \cap U_j \cap \calm$, by Lemma \ref{lemma-change-para} the mapping
\begin{align*}
\varphi_j := \phi_j^{-1} \circ \phi_k : \phi_k^{-1}(W_j) \to \phi_j^{-1}(W_j)
\end{align*}
is a $C^j$-diffeomorphism, and hence we have $\phi_k|_{\phi_k^{-1}(W_j)} = \phi_j \circ \varphi_j \in C^j(\phi^{-1}(W_j);H_j)$. Therefore, setting $V := \bigcap_{j=1}^{k-1} \phi^{-1}(W_j)$ and $\phi := \phi_k|_V$, we obtain
\begin{align*}
\phi \in \bigcap_{j=0}^k C^j(V;H_j).
\end{align*}
(iii) $\Rightarrow$ (i): By Proposition \ref{prop-manifold-scaled} we have $\calm \subset G$ and $\tau_H \cap \calm = \tau_G \cap \calm$. Let $y \in \calm$ be arbitrary, and let $\phi : V \to U \cap \calm$ be a local parametrization around $y$ such that $\phi \in \bigcap_{j=0}^k C^j(V;H_j)$.  Let $j \in \{ 1,\ldots,k-1 \}$ be arbitrary. Then
\begin{align*}
\phi : V \to ( U \cap \calm, \| \cdot \|_H ) \to ( U \cap \calm, \| \cdot \|_{H_j} )
\end{align*}
is a homeomorphism, because $\tau_H \cap \calm = \tau_{H_j} \cap \calm$. Furthermore $\phi \in C^j(V;H_j)$ is a $C^j$-immersion, because $\phi \in C^k(V;H)$ is a $C^k$-immersion. This proves that $\calm$ is a $C^j$-submanifold of $H_j$.

\noindent The additional statement $T \calm \subset G \times H_1$ is a direct consequence of (iii).
\end{proof}

For what follows, let $H_0$ be another Hilbert space such that $(G,H_0,H)$ are continuously embedded Hilbert spaces. We assume that $\calm$ is an $m$-dimensional $(G,H_0,H)$-submanifold of class $C^2$. By Proposition \ref{prop-degrees} we have $T \calm \subset G \times H_0$. For the following result, recall the notation from Definition \ref{def-push}.

\begin{proposition}\label{prop-decomp-Damir-scale}
Let $A : H_0 \to H$ and $B : G \to H_0$ be continuous mappings, and let $\phi : V \to U \cap \calm$ be a local parametrization. Setting $\calm_U := U \cap \calm$, we assume that $A|_{\calm_U}, B|_{\calm_U} \in \Gamma(T \calm_U)$. We define $a,b : V \to \bbr^m$ as $a := \phi_*^{-1} A|_{\calm_U}$ and $b := \phi_*^{-1} B|_{\calm_U}$. Then the following statements are true:
\begin{enumerate}
\item If $A \in C^1(H_0;H)$, then we have $a \in C^1(V;\bbr^m)$, $b \in C(V;\bbr^m)$ and the decomposition
\begin{align}\label{zerl-general}
D A \cdot B|_{\calm_U} = \phi_* (D a \cdot b) + \phi_{**}(a,b).
\end{align}
\item Suppose there is a mapping $\bar{A} \in C^{1,0}(H_0 \times G;H)$ such that $A(y) = \bar{A}(y,y)$ for all $y \in G$ and $\bar{A}(\cdot,z)|_{\calm_U} \in \Gamma(T \calm_U)$ for each $z \in \calm_U$. We define $\bar{a} : V \times V \to \bbr^m$ as $\bar{a} := \phi_*^{-1} \bar{A}$. Then we have $a,b \in C(V;\bbr^m)$, $\bar{a} \in C^{1,0}(V \times V;\bbr^m)$, and the decomposition
\begin{align}\label{zerl-1}
D_1 \bar{A} \cdot B|_{\calm_U} = \phi_*(D_1 \bar{a} \cdot b) + \phi_{**}(a,b).
\end{align}
In particular, if $\bar{A}(\cdot,z) \in L(H_0,H)$ for each $z \in G$, then we have the decomposition
\begin{align}\label{zerl-1-linear}
\bar{A}(B(\cdot),\cdot)|_{\calm_U} = \phi_* ( D_1 \bar{a} \cdot b ) + \phi_{**}(a,b).
\end{align}
\item Suppose there is a mapping $\bar{A} \in C^1(H_0 \times H_0;H)$ such that $A(y) = \bar{A}(y,y)$ for all $y \in G$ and $\bar{A}(\cdot,z)|_{\calm_U} \in \Gamma(T \calm_U)$ for each $z \in \calm_U$. We define $\bar{a} : V \times V \to \bbr^m$ as $\bar{a} := \phi_*^{-1} \bar{A}$. Then we have $a \in C^1(V;\bbr^m)$, $b \in C(V;\bbr^m)$, $\bar{a} \in C^1(V \times V;\bbr^m)$, and the decomposition
\begin{align}\label{zerl-2}
D A \cdot B|_{\calm_U} = \phi_* ( D_2 \bar{a} \cdot b ) + D_1 \bar{A} \cdot B|_{\calm_U}.
\end{align}
In particular, if $\bar{A}(\cdot,z) \in L(H_0,H)$ for each $z \in G$, then we have the decomposition
\begin{align}\label{zerl-2-linear}
D A \cdot B|_{\calm_U} = \phi_* ( D_2 \bar{a} \cdot b ) + \bar{A}(B(\cdot),\cdot)|_{\calm_U}.
\end{align}
\end{enumerate}
\end{proposition}

\begin{remark}
Before we proceed with the proof, let us clarify some notation. In general, the symbols $D_1$ and $D_2$ denote the partial derivatives with respect to the first and the second coordinate. We use the notation $D A \cdot B|_{\calm_U}$ for the mapping $\calm_U \to H$, $y \mapsto D A(y) B(y)$, and the mapping $D a \cdot b$ is defined analogously. Furthermore, we use the notation $D_1 \bar{A} \cdot B|_{\calm_U}$ for the mapping $\calm_U \to H$, $y \mapsto D_1 \bar{A}(y,y) B(y)$, and the mapping $D_1 \bar{a} \cdot b$ is defined analogously. The mapping $\phi_*^{-1} \bar{A} : V \times V \to \bbr^m$ is defined as
\begin{align*}
(\phi_*^{-1} \bar{A})(x,\xi) := D \phi(x)^{-1} \bar{A}(y,z), \quad x \in V,
\end{align*}
where $y := \phi(x) \in U \cap \calm$ and $z := \phi(\xi) \in U \cap \calm$.
\end{remark}

\begin{proof}[Proof of Proposition \ref{prop-decomp-Damir-scale}]
By Proposition \ref{prop-degrees} we have
\begin{align*}
\phi \in C(V;G) \cap C^1(V;H_0) \cap C^2(V;H).
\end{align*}
Therefore, we have $A \circ \phi \in C^1(V;H)$ and $B \circ \phi \in C(V;H_0)$, and hence, by \cite[Prop. 6.1.1]{fillnm} we deduce that $a \in C^1(V;\bbr^m)$ and $b \in C(V;\bbr^m)$. Let $y \in U \cap \calm$ be arbitrary and set $x := \phi^{-1}(y) \in V$. There exists $\epsilon > 0$ such that
\begin{align*}
x + t b(x) \in V \quad \text{for all $t \in (-\epsilon,\epsilon)$.}
\end{align*}
Consequently, the curve
\begin{align*}
\gamma : (-\epsilon,\epsilon) \rightarrow U \cap \calm, \quad \gamma(t) := \phi(x + t b(x))
\end{align*}
is well-defined and satisfies $\gamma(0) = y$. Since $\phi \in C^1(V;H_0)$, we have
\begin{align*}
\gamma \in C^1((-\epsilon,\epsilon);H_0) \quad \text{and} \quad \frac{d}{dt}\gamma(t) |_{t=0} = D \phi(x) b(x) = B(y),
\end{align*}
because $b = \phi_*^{-1} B|_{\calm_U}$. Therefore, since $A \in C^1(H_0;H)$, by the chain rule we have $A \circ \gamma \in C^1((-\epsilon,\epsilon);H)$ and
\begin{align*}
\frac{d}{dt} A(\gamma(t)) |_{t=0} = DA(y)B(y).
\end{align*}
On the other hand, since $A|_{\calm_U} = \phi_* a$, we have
\begin{align*}
A(\gamma(t)) = D \phi(x + t b(x)) a(x + tb(x)), \quad t \in (-\epsilon,\epsilon).
\end{align*}
Thus, noting that $D \phi \in C^1(V;L(\bbr^m,H))$, by the Leibniz Rule we have
\begin{align*}
\frac{d}{dt} A(\gamma(t)) |_{t=0} &= \frac{d}{dt} D\phi(x + t b(x)) a(x + t b(x)) |_{t=0}
\\ &= D\phi(x)(D a(x) b(x)) + D^2 \phi(x)(a(x),b(x)).
\end{align*}
Combining the latter two identities we obtain the decomposition (\ref{zerl-general}).

Now, suppose that the additional assumptions from the second statement are fulfilled. Similar as above, by \cite[Prop. 6.1.1]{fillnm} we deduce that $a,b \in C(V;\bbr^m)$ and $\bar{a} \in C^{1,0}(V \times V;\bbr^m)$. Let $y,z \in U \cap \calm$ be arbitrary, and set $x := \phi^{-1}(y) \in V$ and $\xi := \phi^{-1}(z) \in V$. By the decomposition (\ref{zerl-general}) we have
\begin{align*}
D_1 \bar{A}(y,z) B(y) = D \phi(x) ( D_1 \bar{a}(x,\xi) b(x) ) + D^2 \phi(x) ( \bar{a}(x,\xi), b(x) ).
\end{align*}
With $y=z$ this in particular proves the decomposition (\ref{zerl-1}). If $\bar{A}(\cdot,z) \in L(H_0,H)$ for each $z \in G$, then the decomposition (\ref{zerl-1-linear}) is a direct consequence.

Now, suppose that the additional assumptions from the third statement are fulfilled. Similar as above, by \cite[Prop. 6.1.1]{fillnm} we deduce that $a \in C^1(V;\bbr^m)$, $b \in C(V;\bbr^m)$ and $\bar{a} \in C^1(V \times V;\bbr^m)$. Let $T : \bbr^m \to \bbr^m \times \bbr^m = \bbr^{2m}$ be the linear operator $T(x) = (x,x)$. By the chain rule and \cite[Prop. 2.4.12.ii]{Abraham}, for all $x \in V$ and $v \in \bbr^m$ we have
\begin{align*}
D a(x) v &= D (\bar{a} \circ T)(x)v = D \bar{a}(T(x)) DT(x) v = D \bar{a}(T(x)) Tv
\\ &= D \bar{a}(Tx)(v,v) = D_1 \bar{a}(Tx)v + D_2 \bar{a}(Tx)v = D_1 \bar{a}(x,x)v + D_2 \bar{a}(x,x)v.
\end{align*}
Therefore, for all $x \in V$ we have
\begin{align*}
Da(x)b(x) = D_1 \bar{a}(x,x)b(x) + D_2 \bar{a}(x,x)b(x).
\end{align*}
Hence, by the decompositions (\ref{zerl-general}) and (\ref{zerl-1}) for all $y \in U \cap \calm$ we obtain
\begin{align*}
D A(y) B(y) &= D \phi(x) ( D_1 \bar{a}(x,x)b(x) ) + D \phi(x) ( D_2 \bar{a}(x,x)b(x) )
\\ &\quad + D^2 \phi(x) ( a(x), b(x) )
\\ &= D \phi(x) ( D_2 \bar{a}(x,x) b(x) ) + D_1 \bar{A}(y,y) B(y),
\end{align*}
where $x := \phi^{-1}(y) \in V$, proving the decomposition (\ref{zerl-2}). If $\bar{A}(\cdot,z) \in L(H_0,H)$ for each $z \in G$, then the decomposition (\ref{zerl-2-linear}) is a direct consequence.
\end{proof}

\begin{definition}
We say that an $m$-dimensional $C^k$-submanifold $\calm$ of $H$ has \emph{one chart} if there exists a parametrization $\phi \in C^k(V;H)$ such that $\phi(V) = \calm$. In this case, we call the mapping $\phi : V \to \calm$ a \emph{global parametrization} of $\calm$.
\end{definition}

For what follows, let $d \in \bbn$ be a positive integer such that $m \leq d$, and let $\caln$ be an $m$-dimensional $C^k$-submanifold of $\bbr^d$. The following definition generalizes the concept of an immersion from Definition \ref{def-immersion}.

\begin{definition}
Let $X \subset \bbr^d$ be an open subset such that $X \cap \caln \neq \emptyset$, and let $\psi \in C^k(X;H)$ be a mapping.
\begin{enumerate}
\item Let $x_0 \in X \cap \caln$ be arbitrary. The mapping $\psi$ is called a \emph{$C^k$-immersion on $\caln$ at $x_0$} if $D \psi(x_0)|_{T_{x_0} \caln} \in L(T_{x_0} \caln,H)$ is one-to-one.

\item The mapping $\psi$ is called a \emph{$C^k$-immersion on $\caln$} if it is a $C^k$-immersion on $\caln$ at $x_0$ for each point $x_0 \in X \cap \caln$.
\end{enumerate}
\end{definition}

For what follows, we fix a mapping $\psi \in C^k(\bbr^d;H)$. Thus, we consider the situation $X = \bbr^d$.

\begin{lemma}\label{lemma-matrix-m-m}
Let $x_0 \in \caln$ be arbitrary, let $\{ v_1,\ldots,v_m \}$ be a basis of $T_{x_0} \caln$, and let $h_1,\ldots,h_m \in H$ be such that the matrix
\begin{align}\label{matrix-m-m}
\big( \la D \psi(x_0) v_i, h_j \ra_H \big) \in \bbr^{m \times m}
\end{align}
is invertible. Then $\psi$ is a $C^k$-immersion on $\caln$ at $x_0$.
\end{lemma}

\begin{proof}
It suffices to show that the vectors $D \psi(x_0) v_i$, $i=1,\ldots,m$ are linearly independent. For this purpose, let $c_1,\ldots,c_m \in \bbr$ be such that
\begin{align*}
\sum_{i=1}^m c_i D \psi(x_0) v_i = 0.
\end{align*}
Then for each $j = 1,\ldots,m$ we have
\begin{align*}
\sum_{i=1}^m c_i \la D \psi(x_0)v_i, h_j \ra_H = \bigg\la \sum_{i=1}^m c_i D \psi(x_0)v_i, h_j \bigg\ra_H = 0,
\end{align*}
and by invertibility of the matrix (\ref{matrix-m-m}) we deduce that $c_1 = \ldots = c_m = 0$.
\end{proof}

\begin{lemma}\label{lemma-induced-0}
Let $x_0 \in \caln$ be such that $\psi$ is a $C^k$-immersion on $\caln$ at $x_0$. Then there exists an open neighborhood $W_0 \subset \bbr^d$ of $x_0$ such that:
\begin{enumerate}
\item The submanifold $W_0 \cap \caln$ has one chart.

\item $\psi|_{W_0 \cap \caln} : W_0 \cap \caln \to \psi(W_0 \cap \caln)$ is a homeomorphism.

\item $\psi$ is a $C^k$-immersion on $W_0 \cap \caln$.
\end{enumerate}
\end{lemma}

\begin{proof}
Let $\varphi : V \to W \cap \caln$ be a local parametrization around $x_0$. We set $\xi_0 := \varphi^{-1}(x_0) \in V$ and $\phi := \psi \circ \varphi$. Then by the chain rule we have $\phi \in C^k(V;H)$ and $D \phi(\xi_0) = D \psi(x_0) D \varphi(\xi_0)$, showing that $\phi$ is a $C^k$-immersion at $\xi_0$. By \cite[Prop. 6.1.1]{fillnm} and the Local Injectivity Theorem (see \cite[Thm. 2.5.10]{Abraham}), there exists an open neighborhood $V_0 \subset V$ of $\xi_0$ such that $\phi|_{V_0}$ is an injective $C^k$-immersion and $\phi|_{V_0} : V_0 \to \phi(V_0)$ is a homeomorphism. Since $\varphi$ is a homeomorphism, there is an open neighborhood $W_0 \subset W$ of $x_0$ such that $\varphi(V_0) = W_0 \cap \caln$. Hence, the submanifold $W_0 \cap \caln$ has one chart with global parametrization $\varphi|_{V_0} : V_0 \to W_0 \cap \caln$, and $\psi = \phi \circ \varphi^{-1} : W_0 \cap \caln \to \psi(W_0 \cap \caln)$ is a homeomorphism. Furthermore, by the chain rule, for each $x \in W_0 \cap \caln$ we have
\begin{align*}
D \psi(x)|_{T_x \caln} = D \phi(\xi) D \varphi(\xi)^{-1} \in L(T_x \caln, H),
\end{align*}
where $\xi := \varphi^{-1}(x) \in V$, showing that $\psi$ is a $C^k$-immersion on $W_0 \cap \caln$.
\end{proof}

\begin{lemma}\label{lemma-induced}
Suppose that $\psi|_{\caln} : \caln \to \psi(\caln)$ is a homeomorphism, and that $\psi$ is a $C^k$-immersion on $\caln$. Then the following statements are true:
\begin{enumerate}
\item $\calm := \psi(\caln)$ is an $m$-dimensional $C^k$-submanifold $\calm$ of $H$.

\item For each local parametrization $\varphi : V \to W \cap \caln$ of $\caln$ there exists an open subset $U \subset H$ such that the mapping $\phi := \psi \circ \varphi : V \to U \cap \calm$ is a local parametrization of $\calm$.

\item If $(H_0,\ldots,H_k)$ are continuously embedded Hilbert spaces for such that $H_k = H$ and $\psi \in \bigcap_{j=0}^k C^j(\bbr^d;H_j)$, then $\calm$ is a $(H_0,\ldots,H_k)$-submanifold of class $C^k$.
\end{enumerate}
\end{lemma}

\begin{proof}
Let $\varphi : V \to W \cap \caln$ be a local parametrization of $\caln$, and set $\phi := \psi \circ \varphi$. Since $\psi|_{\caln}$ is a homeomorphism, there exists an open subset $U \subset H$ such that $\psi(W \cap \caln) = U \cap \calm$. Hence, the mapping $\phi : V \to U \cap \calm$ is a homeomorphism. Furthermore, by the chain rule, for each $\xi \in V$ we have $D \phi(\xi) = D \psi(x) D \varphi(\xi)$, where $x := \varphi(\xi) \in W \cap \caln$, showing that $\phi$ is a $C^k$-immersion. Hence, the first two statements follow, and the third statement is a consequence of Proposition \ref{prop-degrees}.
\end{proof}

\begin{definition}\label{def-manifold-embedded}
We say that a submanifold $\calm$ as in Lemma \ref{lemma-induced} is \emph{induced by $(\psi,\caln)$}.
\end{definition}

From now on, we assume that $\psi|_{\caln} : \caln \to \psi(\caln)$ is a homeomorphism, and that $\psi$ is a $C^k$-immersion on $\caln$. According to Lemma \ref{lemma-induced}, let $\calm$ be the $m$-dimensional $C^k$-submanifold of $H$, which is induced by $(\psi,\caln)$. The structure of local parametrizations is illustrated in the following diagram:

\begin{center}
\begin{tikzcd}
& U \cap \calm \\
V \arrow[r, "\varphi"] \arrow[ru, "\phi"] & W \cap \caln \arrow[u, "\psi"']
\end{tikzcd}
\end{center}

\begin{lemma}\label{lemma-M-N-one-chart}
If the submanifold $\caln$ has one chart, then the submanifold $\calm$ has one chart, and if $\varphi : V \to \caln$ is a global parametrization of $\caln$, then $\phi := \psi \circ \varphi : V \to \calm$ is a global parametrization of $\calm$.
\end{lemma}

\begin{proof}
This is a consequence of Lemma \ref{lemma-induced}.
\end{proof}

\begin{lemma}\label{lemma-tang-embedded}
Let $y \in \calm$ be arbitrary, and set $x := \psi^{-1}(y) \in \caln$. Then we have $T_y \calm = D \psi(x) T_x \caln$.
\end{lemma}

\begin{proof}
Let $\varphi : V \to W \cap \caln$ be a local parametrization around $x$. By Lemma \ref{lemma-induced} there exists an open subset $U \subset H$ such that the mapping $\phi := \psi \circ \varphi : V \to U \cap \calm$ is a local parametrization around $y$. Setting $\xi := \varphi^{-1}(x) \in V$, by the chain rule we obtain
\begin{align*}
T_y \calm = D \phi(\xi) \bbr^m = D \psi(x) D \varphi(\xi) \bbr^m = D \psi(x) T_x \caln,
\end{align*}
completing the proof.
\end{proof}

For the upcoming result, let $\varphi : V \to W \cap \caln$ be a local parametrization of $\caln$, and let $\phi := \psi \circ \varphi : V \to U \cap \calm$ be the corresponding local parametrization of $\calm$; see Lemma \ref{lemma-induced}. For a mapping $a : V \to \bbr^m$ we define $\phi_* a : U \cap \calm \to H$ and $\varphi_* a : W \cap \caln \to \bbr^d$ according to Definition \ref{def-push}, and for a mapping $b : W \cap \caln \to \bbr^d$ we define $\psi_* b : U \cap \calm \to H$ analogously.

\begin{lemma}\label{lemma-push-chain-rule}
The following statements are true:
\begin{enumerate}
\item For a mapping $a : V \to \bbr^m$ we have $\phi_* a = \psi_* \varphi_* a$.

\item If $k \geq 2$, then for two mappings $a,b : V \to \bbr^m$ we have
\begin{align*}
\phi_{**} (a,b) = \psi_{**}(\varphi_* a, \varphi_* b) + \psi_* \varphi_{**}(a,b).
\end{align*}
\end{enumerate}
\end{lemma}

\begin{proof}
Let $y \in U \cap \calm$ be arbitrary. We set $\xi := \phi^{-1}(y) \in V$ and $x := \varphi(\xi) \in W \cap \caln$. Then we also have $x = \psi^{-1}(y)$. By the usual chain rule we obtain
\begin{align*}
(\phi_* a)(y) &= D \phi(\xi) a(\xi) = D(\psi \circ \varphi)(\xi) a(\xi) = D \psi(x) D \varphi(\xi) a(\xi)
\\ &= D \psi(x) (\varphi_* a)(x) = (\psi_* \varphi_* a)(y),
\end{align*}
and, if $k \geq 2$, then by the second order chain rule (see Lemma \ref{lemma-second-order-chain}) we obtain
\begin{align*}
(\phi_{**} (a,b))(y) &= D^2 \phi(\xi) ( a(\xi), b(\xi) ) = D^2(\psi \circ \varphi)(\xi) ( a(\xi), b(\xi) )
\\ &= D^2 \psi(x) (D \varphi(\xi) a(\xi), D \varphi(\xi) b(\xi)) + D \psi(x) D^2 \varphi(\xi) (a(\xi), b(\xi))
\\ &= D^2 \psi(x) ((\varphi_* a)(x), (\varphi_* b)(x)) + D \psi(x) \varphi_{**}(a,b)(x)
\\ &= (\psi_{**}(\varphi_* a, \varphi_* b))(y) + (\psi_* \varphi_{**}(a,b))(y),
\end{align*}
completing the proof.
\end{proof}

For the following auxiliary result, recall the Definition \ref{def-local-vector-field} of a local vector field, and the Definition \ref{def-simultaneous} of a locally simultaneous vector field.

\begin{lemma}\label{lemma-eq-loc-fields}
For every mapping $a : \caln \to \bbr^d$ the following statements are true:
\begin{enumerate}
\item Let $A : \calm \to H$ be the mapping $A := \psi_* a$. If $a \in \Gamma(T \caln)$, then we have $A \in \Gamma(T \calm)$.

\item Let $\bar{A} : \calm \times \calm \to H$ be a mapping such that for each $z \in \calm$ the mapping $\bar{A}_z := \bar{A}(\cdot,z)$ is of the form
\begin{align*}
\bar{A}_z(y) := D \psi(x) a(\xi), \quad y \in \calm,
\end{align*}
where $x := \psi^{-1}(y) \in \caln$ and $\xi := \psi^{-1}(z) \in \caln$. If $a \in \Gamma^*(T \caln)$, then we have $\bar{A}_z \in \Gamma_z(T \calm)$ for each $z \in \calm$.
\end{enumerate}
\end{lemma}

\begin{proof}
For the proof of the first statement, let $y \in \calm$ be arbitrary, and set $x := \psi^{-1}(y) \in \caln$. By Lemma \ref{lemma-tang-embedded} we obtain
\begin{align*}
A(y) = D \psi(x) a(x) \in D \psi(x) T_x \caln = T_y \calm.
\end{align*}
We proceed with the proof of the second statement. Let $z \in \calm$ be arbitrary, and set $\xi := \psi^{-1}(z) \in \caln$. Since $a \in \Gamma^*(T \caln)$, there exists an open neighborhood $W \subset \bbr^d$ of $x$ such that
\begin{align*}
a(\xi) \in T_x \caln \quad \text{for each $x \in W \cap \caln$.} 
\end{align*}
Therefore, by Lemma \ref{lemma-tang-embedded} for each $y \in U \cap \calm$ we obtain
\begin{align*}
\bar{A}_z(y) = D \psi(x) a(\xi) \in D \psi(x) T_x \caln = T_y \calm,
\end{align*}
where $x := \psi^{-1}(y) \in W \cap \caln$.
\end{proof}

\subsection{Finite dimensional submanifolds generated by orbit maps of group actions}\label{sec-manifolds-HS}

In this section we deal with finite dimensional submanifolds given by orbit maps of group actions, in particular in Hermite Sobolev spaces. Let $H$ be a separable Hilbert space. We also fix positive integers $k \in \overline{\bbn}$ and $m,d \in \bbn$ such that $m \leq d$. If $G$ is the higher-order domain of closed operators, then there is another criterion for a $(G,H)$-submanifold, which adds to Proposition \ref{prop-manifold-scaled}.

\begin{proposition}\label{prop-group-top-manifold}
Let $A_i : H \supset D(A_i) \to H$, $i=1,\ldots,d$ be closed operators, and set $G := D(A^n)$ for some $n \in \bbn$. Let $\calm$ be an $m$-dimensional $C^k$-submanifold of $H$. Then the following statements are equivalent:
\begin{enumerate}
\item[(i)] $\calm$ is an $m$-dimensional $(G,H)$-submanifold of class $C^k$.

\item[(ii)] $\calm \subset G$ and for all $j = 1,\ldots,n$ and $\alpha \in \{ 1,\ldots,d \}^j$ the restricted operator $A^{\alpha}|_{\calm} : (\calm,\| \cdot \|_H) \to (H,\| \cdot \|_H)$ is continuous.
\end{enumerate}
\end{proposition}

\begin{proof}
By Proposition \ref{prop-manifold-scaled} the submanifold $\calm$ is a $(G,H)$-submanifold of class $C^k$ if and only if $\calm \subset G$ and the identity $\Id : (\calm,\| \cdot \|_H) \to (\calm,\| \cdot \|_{D(A^n)})$ is continuous. Therefore, by the definition (\ref{norm-D-An}) of the norm $\| \cdot \|_{D(A^n)}$, the claimed equivalence follows.
\end{proof}

Now, let $T = (T(t))_{t \in \bbr^d}$ be a multi-parameter $C_0$-group on $H$ with generator $A$; see Appendix \ref{app-groups} for details. As a consequence of Lemmas \ref{lemma-induced-0}, \ref{lemma-induced} and Proposition \ref{prop-degrees-smoothness}, we obtain the following examples of submanifolds generated by the orbit maps of the group $T$. 

\begin{examples}\label{ex-manifolds-domain}
Let $\Phi \in D(A^k)$ be arbitrary, let $\psi := \xi_{\Phi} : \bbr^d \to H$ be the orbit map given by $\psi(t) = T(t) \Phi$ for each $t \in \bbr^d$, and let $\caln$ be an $m$-dimensional $C^k$-submanifold of $\bbr^d$. Then the following statements are true:
\begin{enumerate}
\item If $\psi$ is a $C^k$-immersion on $\caln$ at $x_0$ for some $x_0 \in \caln$, then there exists an open neighborhood $W_0 \subset \bbr^d$ of $x_0$ such that $\calm := \psi(W_0 \cap \caln)$ is an $m$-dimensional $(D(A^k),\ldots,H)$-submanifold of class $C^k$ with one chart, which is induced by $(\psi,\caln)$.

\item If $\psi$ is a $C^k$-immersion on $\caln$ such that $\psi|_{\caln} : \caln \to \psi(\caln)$ is a homeomorphism, then $\calm := \psi(\caln)$ is an $m$-dimensional $(D(A^k),\ldots,H)$-submanifold of class $C^k$, which is induced by $(\psi,\caln)$. 
\end{enumerate}
\end{examples}

Now, we turn to Hermite Sobolev spaces; see Appendix \ref{app-HS-spaces} for further details. For submanifolds in Hermite Sobolev spaces, there is another criterion for a $(G,H)$-submanifold, which adds to Proposition \ref{prop-manifold-scaled}. Recall that $\mathbf{H}$ denotes the Hermite operator.

\begin{proposition}
Let $p \in \bbr$ and $l \in \bbn$ be arbitrary. We set $G := \cals_{p+l}(\bbr^d)$ and $H := \cals_p(\bbr^d)$. Let $\calm$ be an $m$-dimensional $C^k$-submanifold of $H$. Then the following statements are equivalent:
\begin{enumerate}
\item[(i)] $\calm$ is an $m$-dimensional $(G,H)$-submanifold of class $C^k$.

\item[(ii)] $\calm \subset G$ and the restriction $\mathbf{H}^l|_{\calm} : (\calm,\| \cdot \|_H) \to (\mathbf{H}^l(\calm),\| \cdot \|_H)$ is a homeomorphism.
\end{enumerate}
\end{proposition}

\begin{proof}
By Proposition \ref{prop-manifold-scaled} the submanifold $\calm$ is a $(G,H)$-submanifold of class $C^k$ if and only if $\calm \subset G$ and the identity $\Id|_{\calm} : (\calm,\| \cdot \|_H) \to (\calm,\| \cdot \|_G)$ is a homeomorphism. 

\noindent(i) $\Rightarrow$ (ii): If $\Id : (\calm,\| \cdot \|_H) \to (\calm,\| \cdot \|_G)$ is a homeomorphism, then by Lemma \ref{lemma-Hermite} the restriction
\begin{align*}
\mathbf{H}^l|_{\calm} : (\calm,\| \cdot \|_H) \overset{\Id}{\longrightarrow} (\calm,\| \cdot \|_G) \overset{\mathbf{H}^l}{\longrightarrow} (\mathbf{H}^l(\calm),\| \cdot \|_H)
\end{align*}
is a homeomorphism as well.

\noindent(ii) $\Rightarrow$ (i): If $\mathbf{H}^l|_{\calm} : (\calm,\| \cdot \|_H) \to (\mathbf{H}^l(\calm),\| \cdot \|_H)$ is a homeomorphism, then by Lemma \ref{lemma-Hermite} the identity
\begin{align*}
\Id : (\calm,\| \cdot \|_H) \overset{\mathbf{H}^l}{\longrightarrow} (\mathbf{H}^l(\calm),\| \cdot \|_H) \overset{\mathbf{H}^{-l}}{\longrightarrow} (\calm,\| \cdot \|_G)
\end{align*}
is a homeomorphism as well.
\end{proof}

For the rest of this section, we will present examples of submanifolds in Hermite Sobolev spaces which are generated by the orbit maps of the translation group. For this purpose, recall the translation group $\tau = (\tau_x)_{x \in \bbr^d}$ from Appendix \ref{app-HS-spaces}; see in particular Lemma \ref{lemma-tau-group}. For each $i=1,\ldots,d$ we define the family $\tau^i = (\tau_x^i)_{x \in \bbr}$ as
\begin{align*}
\tau_x^i := \tau_{x e_i}, \quad x \in \bbr.
\end{align*}
Let $p \in \bbr$ be arbitrary. Then $\tau^1,\ldots,\tau^d$ are commutative $C_0$-groups on $\cals_p(\bbr^d)$, and we have
\begin{align*}
\tau_x = \tau_{x_1}^1 \circ \ldots \circ \tau_{x_d}^d, \quad x \in \bbr^d.
\end{align*}
For each $i=1,\ldots,d$ we denote by $A_{p,i} : \cals_p(\bbr^d) \supset D(A_{p,i}) \to \cals_p(\bbr^d)$ the generator of the $C_0$-group $\tau^i$ on $\cals_p(\bbr^d)$. Then $A_p = (A_{p,1},\ldots,A_{p,d})$ is the generator of the multi-parameter $C_0$-group $\tau$. The following result shows that for each $i=1,\ldots,d$ the subspace $\cals_{p+\frac{1}{2}}(\bbr^d)$ is contained in the domain $D(A_{p,i})$, and that it is even a large subspace of $D(A_{p,i})$ in the sense that it is a core for $A_{p,i}$.

\begin{theorem}\label{thm-Sp-D-1}
For each $p \in \bbr$ and each $i=1,\ldots,d$ the following statements are true:
\begin{enumerate}
\item We have $\cals_{p+\frac{1}{2}}(\bbr^d) \subset D(A_{p,i})$.

\item $\cals_{p+\frac{1}{2}}(\bbr^d)$ is a core for $A_{p,i}$.

\item We have $A_{p,i} \Phi = -\partial_i \Phi$ for each $\Phi \in \cals_{p+\frac{1}{2}}(\bbr^d)$.
\end{enumerate}
\end{theorem}

\begin{proof}
Let $\Phi \in \cals_{p+\frac{1}{2}}(\bbr^d)$ be arbitrary. By Lemma \ref{lemma-Taylor} there exists a continuous mapping $R : \bbr \times \bbr \to \cals_p(\bbr^d)$ with $R(x,0) = 0$ for all $x \in \bbr$ such that
\begin{align*}
\tau_{x+h}^i \Phi = \tau_x^i \Phi - h \partial_i \tau_x^i \Phi - h R(x,h), \quad x,h \in \bbr
\end{align*}
in the space $\cals_p(\bbr^d)$. Thus, denoting by $\xi_{\Phi}^i : \bbr \to \cals_{p+\frac{1}{2}}(\bbr^d)$ the orbit map given by $\xi_{\Phi}^i(x) = \tau_x^i \Phi$ for each $x \in \bbr$, we have
\begin{align*}
\xi_{\Phi}^i(x+h) \Phi = \xi_{\Phi}^i(x) - h \partial_i \xi_{\Phi}^i(x) - h R(x,h), \quad x,h \in \bbr
\end{align*}
in the space $\cals_p(\bbr^d)$. By Taylor's theorem (see \cite[Thm. 2.4.15]{Abraham}) we obtain $\xi_{\Phi}^i \in C^1(\bbr;\cals_p(\bbr^d))$ with $\dot{\xi}_{\Phi}^i = -\partial_i \xi_{\Phi}^i$. We deduce that $\Phi \in D(A_{p,i})$ and $A_{p,i} \Phi = -\partial_i \Phi$. Furthermore, the subspace $\cals_{p+\frac{1}{2}}(\bbr^d)$ is $\| \cdot \|_p$-dense in $\cals_p(\bbr^d)$ and invariant under the group $\tau^i$; see Lemma \ref{lemma-tau-group}. Therefore, by virtue of \cite[Prop. II.1.7]{Engel-Nagel} the space $\cals_{p+\frac{1}{2}}(\bbr^d)$ is a core for $A_{p,i}$.
\end{proof}

\begin{lemma}\label{lemma-D-pq}
Let $p,q \in \bbr$ with $p \leq q$ and $n \in \bbn_0$ be arbitrary. Then the following statements are true:
\begin{enumerate}
\item We have $D(A_q^n) \subset D(A_p^n)$.

\item We have $A_q^{\alpha} \Phi = A_p^{\alpha} \Phi$ for all $m \in \bbn_0$ with $m \leq n$ and $\alpha \in \{ 1,\ldots,d \}^m$.
\end{enumerate}
\end{lemma}

\begin{proof}
Let $\Phi \in D(A_q^n)$ be arbitrary. Then, by Proposition \ref{prop-D-An-multi} we have $\xi_{\Phi} \in C^n(\bbr^d;\cals_q(\bbr^d))$, where $\xi_{\Phi} : \bbr^d \to \cals_q(\bbr^d)$ denotes the orbit map given by $\xi_{\Phi}(x) = \tau_x \Phi$ for each $x \in \bbr^d$. Taking into account Lemma \ref{lemma-scale-pq}, we deduce that $\xi_{\Phi} \in C^n(\bbr^d;\cals_p(\bbr^p))$, and hence, by Proposition \ref{prop-D-An-multi} we have $\Phi \in D(A_p^n)$. Furthermore, taking into account Lemma \ref{lemma-scale-pq} again, we obtain $A_q^{\alpha} \Phi = A_p^{\alpha} \Phi$ for all $m \in \bbn_0$ with $m \leq n$ and $\alpha \in \{ 1,\ldots,d \}^m$.
\end{proof}

\begin{proposition}\label{prop-Sp-in-domain}
Let $p \in \bbr$ and $n \in \bbn$ be arbitrary. Then the following statements are true:
\begin{enumerate}
\item The pair
\begin{align}\label{scale-Sp-domain}
\big( \cals_{p+\frac{n}{2}}(\bbr^d), D(A_p^n) \big)
\end{align}
consists of separable Hilbert spaces with continuous embedding.

\item For all $m \in \bbn_0$ with $m \leq n$ and $\alpha \in \{ 1,\ldots,d \}^m$ we have
\begin{align}\label{generator-on-Sp}
A_p^{\alpha} \Phi = (-1)^m \partial^{\alpha} \Phi \quad \text{for each $\Phi \in \cals_{p+\frac{n}{2}}(\bbr^d)$.}
\end{align}
\end{enumerate}
\end{proposition}

\begin{proof}
By induction we prove $\cals_{p+\frac{n}{2}}(\bbr^d) \subset D(A_p^n)$ and the identity (\ref{generator-on-Sp}) for each $n \in \bbn$. For $n=1$ this is a consequence of Theorem \ref{thm-Sp-D-1}. We proceed with the induction step $n-1 \to n$: By induction hypothesis and Lemma \ref{lemma-D-pq} we have
\begin{align*}
\cals_{p+\frac{n}{2}}(\bbr^d) = \cals_{(p+\frac{1}{2}) + \frac{n-1}{2}}(\bbr^d) \subset D(A_{p+\frac{1}{2}}^{n-1}) \subset D(A_p^{n-1}).
\end{align*}
Now, let $\Phi \in \cals_{p+\frac{n}{2}}(\bbr^d)$ be arbitrary. By Lemma \ref{lemma-D-pq} and induction hypothesis, for all $\alpha \in \{ 1,\ldots,d \}^{n-1}$ we have
\begin{align*}
A_p^{\alpha} \Phi = A_{p+\frac{1}{2}}^{\alpha} \Phi = (-1)^{n-1} \partial^{\alpha} \Phi \in \cals_{p+\frac{1}{2}}(\bbr^d) \subset D(A_p),
\end{align*}
and hence $\Phi \in D(A_p^n)$. Furthermore, using Theorem \ref{thm-Sp-D-1} we obtain (\ref{generator-on-Sp}). Finally, by Lemma \ref{lemma-diff-continuous} the pair (\ref{scale-Sp-domain}) consists of separable Hilbert spaces with continuous embedding for each $n \in \bbn$.
\end{proof}

The following result generalizes \cite[Prop. 1.4]{Rajeev-T}.

\begin{proposition}\label{prop-orbit-in-Sp}
Let $p \in \bbr$, $n \in \bbn_0$ and $\Phi \in \cals_{p + \frac{n}{2}}(\bbr^d)$ be arbitrary. Then the following statements are true:
\begin{enumerate}
\item We have
\begin{align*}
\xi_{\Phi} \in \bigcap_{k=0}^n C^k(\bbr^d;\cals_{p + \frac{n-k}{2}}(\bbr^d)),
\end{align*}
where $\xi_{\Phi} : \bbr^d \to \cals_{p+\frac{n}{2}}(\bbr^d)$ denotes the orbit map given by $\xi_{\Phi}(x) = \tau_x \Phi$ for each $x \in \bbr^d$.

\item In particular, we have $\xi_{\Phi} \in C^n(\bbr^d;\cals_p(\bbr^d))$, and for each $m \in \bbn_0$ with $m \leq n$ we have
\begin{align*}
D^m \xi_{\Phi}(x)v = (-1)^m \sum_{\alpha \in \{ 1,\ldots,d \}^m} \partial^{\alpha} \xi_{\Phi}(x) v_{\alpha}, \quad x \in \bbr^d \text{ and } v \in (\bbr^d)^m,
\end{align*}
where we use the notation $v_{\alpha} := v_{\alpha_1} \cdot \ldots \cdot v_{\alpha_m}$.
\end{enumerate}
\end{proposition}

\begin{proof}
This is a consequence of Propositions \ref{prop-degrees-smoothness}, \ref{prop-Sp-in-domain} and Lemma \ref{lemma-diff-continuous}.
\end{proof}

For the next result, recall that every finite signed measure $\mu$ on $(\bbr^d,\calb(\bbr^d))$ may be regarded as a distribution $\mu \in \cals_p(\bbr^d)$ for each $p < - \frac{d}{4}$; see Lemma \ref{lemma-distr-int-ex}.

\begin{proposition}\label{prop-psi-measure}
Let $p \in \bbr$ and $k \in \bbn$ be such that $p+\frac{k}{2} < -\frac{d}{4}$, and let $\mu \in \cals_{p+\frac{k}{2}}(\bbr^d)$ be a finite signed measure on $(\bbr^d,\calb(\bbr^d))$ with compact support and $\mu(\bbr^d) \neq 0$. Furthermore, let $\psi := \xi_{\mu} : \bbr^d \to \cals_p(\bbr^d)$ be the orbit map given by $\psi(x) = \tau_x \mu$ for $x \in \bbr^d$. Then $\psi : \bbr^d \to \psi(\bbr^d)$ is a homeomorphism and $\psi$ is a $C^k$-immersion.
\end{proposition}

\begin{proof}
First, we show that $\psi$ is injective. Let $x,y \in \bbr^d$ be such that $\psi(x) = \psi(y)$. Then we have $\tau_x \mu = \tau_y \mu$. Since $\supp(\mu)$ is compact, by Lemma \ref{lemma-fct-comp-supp} there exists a Schwartz function $\varphi \in \cals(\bbr^d)$ such that $\varphi(x+z) = x+z$ and $\varphi(y+z) = y+z$ for all $z \in \supp(\mu)$. By Lemma \ref{lemma-distr-int-ex} we obtain
\begin{align*}
0 &= \la \psi(x) - \psi(y), \varphi \ra = \int_{\bbr^d} \big( \varphi(z+x) - \varphi(z+y) \big) \mu(dz)
\\ &= \int_{\bbr^d} \big( (z+x) - (z+y) \big) \mu(dz) = (x-y) \mu(\bbr^d).
\end{align*}
Since $\mu(\bbr^d) \neq 0$, we deduce that $x=y$.

Next, we show that $\psi : \bbr^d \to \psi(\bbr^d)$ is a homeomorphism. For this purpose, let $(x_n)_{n \in \bbn} \subset \bbr^d$ and $x \in \bbr^d$ be such that $\psi(x_n) \to \psi(x)$. We will show that $x_n \to x$. First, note that the sequence $(x_n)_{n \in \bbn}$ is bounded. Indeed, suppose, on the contrary, that $(x_n)_{n \in \bbn}$ is unbounded. Then there is a subsequence $(x_{n_k})_{k \in \bbn}$ such that $|x_{n_k}| \to \infty$ for $k \to \infty$. Since $\supp(\mu)$ is compact, by Lemma \ref{lemma-fct-comp-supp} there exists a Schwartz function $\varphi \in \cals(\bbr^d)$ with compact support such that $\varphi(z+x) = 1$ for all $z \in \supp(\mu)$. Therefore, by Lemma \ref{lemma-distr-int-ex} we obtain
\begin{align*}
\la \psi(x),\varphi \ra = \int_{\bbr^d} \varphi(z+x) \mu(dz) = \mu(\bbr^d) \neq 0. 
\end{align*}
Since $\varphi$ has compact support, by Lebesgue's dominated convergence theorem we deduce that
\begin{align*}
\la \psi(x_{n_k}), \varphi \ra = \int_{\bbr^d} \varphi(z+x_{n_k}) \mu(dz) \to 0 \quad \text{for $k \to \infty$.}
\end{align*}
On the other hand, we have $\psi(x_{n_k}) \to \psi(x)$, and hence the contradiction 
\begin{align*}
\la \psi(x_{n_k}), \varphi \ra \to \la \psi(x), \varphi \ra \neq 0 \quad \text{for $k \to \infty$.}
\end{align*}
Hence, the sequence $(x_n)_{n \in \bbn}$ is bounded. Since $\psi(x_n) \to \psi(x)$, by Lemma \ref{lemma-distr-int-ex} for each $\varphi \in \cals(\bbr^d)$ we have
\begin{align*}
\int_{\bbr^d} \big( \varphi(z + x_n) - \varphi(z+x) \big) \mu(dz) = \la \psi(x_n) - \psi(x), \varphi \ra \to 0.
\end{align*}
Since the sequence $(x_n)_{n \in \bbn}$ is bounded and $\supp(\mu)$ is compact, by Lemma \ref{lemma-fct-comp-supp} there exists a Schwartz function $\varphi \in \cals(\bbr^d)$ such that $\varphi(z+x) = z+x$ for all $z \in \supp(\mu)$ as well as $\varphi(z+x_n) = z+x_n$ for all $z \in \supp(\mu)$ and all $n \in \bbn$. This gives us
\begin{align*}
\int_{\bbr^d} \big( \varphi(z + x_n) - \varphi(z+x) \big) \mu(dz) &= \int_{\bbr^d} \big( (z + x_n) - (z+x) \big) \mu(dz)
\\ &= (x_n - x) \mu(\bbr^d)
\end{align*}
for each $n \in \bbn$. Since $\mu(\bbr^d) \neq 0$, we deduce that $x_n \to x$, showing that $\psi : \bbr^d \to \psi(\bbr^d)$ is a homeomorphism.

Now, we prove that $\psi$ is a $C^k$-immersion. By Proposition \ref{prop-orbit-in-Sp} we have $\psi \in C^k(\bbr^d;\cals_p(\bbr^d))$. Let $x_0 \in \bbr^d$ be arbitrary. Since $\supp(\mu)$ is compact, by Lemma \ref{lemma-fct-comp-supp} for each $j=1,\ldots,d$ there exists a Schwartz function $\varphi_j \in \cals(\bbr^d)$ such that $\varphi_j(z+x_0) = z_j$ for all $z \in \supp(\mu)$. Therefore, by Proposition \ref{prop-orbit-in-Sp} and Lemma \ref{lemma-distr-int-ex} for all $i,j = 1,\ldots,d$ we have
\begin{align*}
\la D \psi(x_0)e_i, \varphi_j \ra &= - \la \partial_i \tau_{x_0} \mu, \varphi_j \ra = \la \tau_{x_0} \mu, \partial_i \varphi_j \ra
\\ &= \int_{\bbr^d} \partial_i \varphi_j(z + x_0) \mu(dz) = \delta_{ij} \mu(\bbr^d).
\end{align*}
Since $\mu(\bbr^d) \neq 0$, by Lemma \ref{lemma-matrix-m-m} it follows that $\psi$ is a $C^k$-immersion at $x_0$.
\end{proof}

For the next results, recall that every polynomial $f : \bbr^d \to \bbr$ in several variables with $\deg(f) = n$ for some $n \in \bbn_0$ may be regarded as a distribution $f \in \cals_p(\bbr^d)$ for each $p < - \frac{d}{4} - \frac{n}{2}$; see Lemma \ref{lemma-polynomial}.

\begin{lemma}\label{lemma-functionals-poly}
Let $p \in \bbr$ and $n \in \bbn$ with $n \leq d$ be such that $p < - \frac{d}{4} - \frac{n}{2}$. Let $f \in \cals_p(\bbr^d)$ be the polynomial
\begin{align*}
f : \bbr^d \to \bbr, \quad f(x) = x_1 \cdot \ldots \cdot x_n.
\end{align*}
Then for each $z \in \bbr^d$ there exists a Schwartz function $\varphi_z \in \cals(\bbr^d)$ such that for each $x \in \bbr^d$ we have
\begin{align*}
\la \tau_x f, \varphi_z \ra &= (z_1 - x_1) \cdot \ldots \cdot (z_n - x_n),
\\ \la \partial_i \tau_x f, \varphi_z \ra &= (z_1 - x_1) \cdot \ldots \cdot (z_{i-1} - x_{i-1}) \cdot (z_{i+1} - x_{i+1}) \cdot \ldots \cdot (z_n - x_n)
\end{align*}
for all $i = 1,\ldots,d$.
\end{lemma}

\begin{proof}
Let $\varphi \in \cals(\bbr^d)$ be the density of a $d$-dimensional standard normal distribution; that is
\begin{align*}
\varphi(y) = \frac{1}{(2 \pi)^{d/2}} \exp \bigg( -\frac{\| y \|^2}{2} \bigg), \quad y \in \bbr^d.
\end{align*}
Now, let $z \in \bbr^d$ be arbitrary. We set $\varphi_z := \tau_z \varphi$. Then $\varphi_z$ is the density of the $d$-dimensional normal distribution ${\rm N}(z,\Id)$. Now, let $Y \sim {\rm N}(z,\Id)$ be a normally distributed random vector. Then, for all $x \in \bbr^d$ we have
\begin{align*}
\la \tau_x f, \varphi_z \ra &= \int_{\bbr^d} f(y-x) \varphi_z(y) dy = \bbe[(Y_1-x_1) \cdot \ldots \cdot (Y_n - x_n)]
\\ &= \bbe[Y_1-x_1] \cdot \ldots \cdot \bbe[Y_n - x_n] = (z_1 - x_1) \cdot \ldots \cdot (z_n - x_n).
\end{align*}
Now, let $i=1,\ldots,d$ be arbitrary. Note that
\begin{align*}
&\partial_i \varphi_z(y) = -(y_i - z_i) \varphi_z(y), \quad y \in \bbr^d,
\\ &\bbe[(Y_i - x_i)(Y_i - z_i)] = \bbe[(Y_i - z_i)^2] + (z_i - x_i) \bbe[Y_i - z_i] = 1, \quad x \in \bbr^d.
\end{align*}
Therefore, for all $x \in \bbr^d$ we obtain
\begin{align*}
\la \partial_i \tau_x f, \varphi_z \ra &= - \la \tau_x f, \partial_i \varphi_z \ra = \int_{\bbr^d} f(y-x) (y_i - z_i) \varphi_z(y) dy
\\ &= \bbe[ (Y_1 - x_1) \cdot \ldots \cdot (Y_n - x_n) \cdot (Y_i - z_i) ]
\\ &= (z_1 - x_1) \cdot \ldots \cdot (z_{i-1} - x_{i-1}) \cdot (z_{i+1} - x_{i+1}) \cdot \ldots \cdot (z_n - x_n),
\end{align*}
completing the proof.
\end{proof}

For $m \in \bbn$ with $m \leq d$ we denote by $\bbr^m \times \{ 0 \} \subset \bbr^d$ be the subspace $\bbr^m \times \{ 0 \} := \lin \{ e_1,\ldots,e_m \}$, where $e_1,\ldots,e_m \in \bbr^d$ denote the first $m$ unit vectors.

\begin{proposition}\label{prop-psi-poly}
Let $p \in \bbr$ and $k,m,n \in \bbn$ with $m \leq n \leq d$ be such that $p + \frac{k}{2} < - \frac{d}{4} - \frac{n}{2}$. Let $f \in \cals_{p+\frac{k}{2}}(\bbr^d)$ be the polynomial
\begin{align*}
f : \bbr^d \to \bbr, \quad f(x) = x_1 \cdot \ldots \cdot x_n,
\end{align*}
and let $\psi := \xi_f : \bbr^d \to \cals_p(\bbr^d)$ be the orbit map given by $\psi(x) = \tau_x f$ for $x \in \bbr^d$. Set $\caln := \bbr^m \times \{ 0 \} $. Then $\psi|_{\caln} : \caln \to \psi(\caln)$ is a homeomorphism and $\psi$ is a $C^k$-immersion on $\caln$.
\end{proposition}

\begin{proof}
First, we show that $\psi|_{\caln}$ is injective. Let $x,y \in \caln$ be such that $\psi(x) = \psi(y)$, that is $\tau_x f = \tau_y f$. By Lemma \ref{lemma-functionals-poly} we have
\begin{align*}
(z_1 - x_1) \cdot \ldots \cdot (z_n - x_n) = (z_1 - y_1) \cdot \ldots \cdot (z_n - y_n), \quad z \in \bbr^d.
\end{align*}
Taking partial derivatives with respect to $z$, inductively we deduce that $x=y$.

Next, we show that $\psi|_{\caln} : \caln \to \psi(\caln)$ is a homeomorphism. Let $(x_m)_{m \in \bbn} \subset \bbr^n$ and $x \in \bbr^n$ be such that $\psi(x_m) \to \psi(x)$. By Lemma \ref{lemma-functionals-poly} we have
\begin{align*}
(z_1 - x_{m,1}) \cdot \ldots \cdot (z_n - x_{m,n}) \to (z_1 - x_1) \cdot \ldots \cdot (z_n - x_n), \quad z \in \bbr^d.
\end{align*}
Taking partial derivatives with respect to $z$, inductively we deduce that $x_m \to x$.

Now, we prove that $\psi$ is a $C^k$-immersion on $\caln$. By Proposition \ref{prop-orbit-in-Sp} we have $\psi \in C^k(\bbr^d;\cals_p(\bbr^d))$. Let $x_0 \in \caln$ be arbitrary. We set $z_j := \bbI + x_0 - e_j \in \caln$ for $j=1,\ldots,m$, where $\bbI := \sum_{i=1}^m e_i = (1,\ldots,1,0,\ldots,0) \in \caln$. Then for all $i,j = 1,\ldots,m$ we have $z_{j,i} - x_{0,i} = 1 - \delta_{ij}$, and by Lemma \ref{lemma-functionals-poly} we obtain
\begin{align*}
\la D \psi(x_0) e_i, \varphi_{z_j} \ra = -\la \partial_i \tau_{x_0} f, \varphi_{z_j} \ra = -\delta_{ij}.
\end{align*}
Therefore, by Lemma \ref{lemma-matrix-m-m} it follows that $\psi$ is a $C^k$-immersion on $\caln$ at $x_0$.
\end{proof}

For the next result, recall that $\cals_p(\bbr^d) \subset C_0^1(\bbr^d)$ for each $p > \frac{d}{4} + \frac{1}{2}$; see the Sobolev embedding theorem for Hermite Sobolev spaces (Theorem \ref{thm-Ck-versions}).

\begin{proposition}\label{prop-psi-function}
Let $p \in \bbr$ and $k \in \bbn$ be such that $p + \frac{k}{2} > \frac{d}{4} + \frac{1}{2}$, and let $\varphi \in \cals_{p+\frac{k}{2}}(\bbr^d)$ be arbitrary. Let $n \in \bbn$ with $m \leq n \leq d$ be arbitrary, let $\caln$ be an $m$-dimensional $C^k$-submanifold of $\bbr^d$, and let $E \subset \bbr^d$ be an $n$-dimensional subspace such that $T \caln \subset \caln \times E$. Suppose there are $v_1,\ldots,v_n \in E$ and $z_1,\ldots,z_n \in \bbr^d$ such that the matrix
\begin{align}\label{matrix}
\big( D_{v_i} \varphi(z_j) \big)_{i,j=1,\ldots,n} \in \bbr^{n \times n}
\end{align}
is invertible. Let $\psi := \xi_{\varphi} : \bbr^d \to \cals_p(\bbr^d)$ be the orbit map given by $\psi(x) = \tau_x \varphi$ for $x \in \bbr^d$. Then $\psi|_{\caln} : \caln \to \psi(\caln)$ is a homeomorphism and $\psi$ is a $C^k$-immersion on $\caln$.
\end{proposition}

\begin{proof}
First, we show that $\psi$ is injective. Let $x,y \in \bbr^d$ be such that $\psi(x) = \psi(y)$. Then we have $\tau_x \varphi = \tau_y \varphi$. Suppose that $x \neq y$. Then for all $z \in \bbr^d$ we have
\begin{align*}
\varphi(z-x) = \la \delta_z, \tau_x \varphi \ra = \la \delta_z, \tau_y \varphi \ra = \varphi(z-y).
\end{align*}
We set $\Delta := y-x \neq 0$. Inductively, for all $z \in \bbr^d$ and $n \in \bbn_0$ we obtain
\begin{align*}
\varphi(z - x) = \varphi(z - x - \Delta) = \ldots = \varphi(z - x - n \Delta).
\end{align*}
Since the matrix (\ref{matrix}) is invertible, we have $\varphi \neq 0$. Hence, there exists $z \in \bbr^d$ such that $\varphi(z-x) \neq 0$. However, by Theorem \ref{thm-Ck-versions} we have $\varphi \in C_0^1(\bbr^d)$, and hence, we obtain the contradiction
\begin{align*}
\lim_{n \to \infty} \varphi(z-x - n \Delta) = 0,
\end{align*}
showing that $\psi$ is injective.

Next, we show that $\psi : \bbr^d \to \psi(\bbr^d)$ is a homeomorphism. Let $(x_n)_{n \in \bbn} \subset \bbr^d$ and $x \in \bbr^d$ be such that $\psi(x_n) \to \psi(x)$. We will show that $x_n \to x$. First, note that the sequence $(x_n)_{n \in \bbn}$ is bounded. Indeed, suppose, on the contrary, that $(x_n)_{n \in \bbn}$ is unbounded. Then there is a subsequence $(x_{n_k})_{k \in \bbn}$ such that $|x_{n_k}| \to \infty$ for $k \to \infty$. Since $\varphi \neq 0$, there exists $z \in \bbr^d$ such that $\varphi(z-x) \neq 0$. Since $\tau_{x_{n_k}} \varphi \to \tau_x \varphi$, we have
\begin{align*}
\lim_{k \to \infty} \varphi(z - x_{n_k}) = \lim_{k \to \infty} \la \delta_{z}, \tau_{x_{n_k}} \varphi \ra = \la \delta_{z}, \tau_x \varphi \ra = \varphi(z - x) \neq 0.
\end{align*}
However, by Theorem \ref{thm-Ck-versions} we obtain the contradiction $\varphi \in C_0^1(\bbr^d)$, showing that the sequence $(x_n)_{n \in \bbn}$ is bounded. Now, let $(n_k)_{k \in \bbn}$ be an arbitrary subsequence. Since $(x_{n_k})_{k \in \bbn}$ is bounded, there exists another subsequence $(n_{k_l})_{l \in \bbn}$ such that $\lim_{l \to \infty} x_{n_{k_l}} = y$ for some $y \in \bbr^d$. This gives us $\psi(x_{n_{k_l}}) \to \psi(y)$, and hence $\psi(x) = \psi(y)$. By the injectivity of $\psi$ we deduce that $x = y$. Therefore, we have $\lim_{l \to \infty} x_{n_{k_l}} = x$. Since the subsequence $(n_k)_{k \in \bbn}$ was arbitrary, we deduce that $x_n \to x$, showing that $\psi : \bbr^d \to \psi(\bbr^d)$ is a homeomorphism.

Now, we show that $\psi$ is a $C^k$-immersion on $\caln$. By Proposition \ref{prop-orbit-in-Sp} we have $\psi \in C^k(\bbr^d;\cals_p(\bbr^d))$. Let $x_0 \in \caln$ be arbitrary. We set $\Phi_j := \delta_{x_0 + z_j}$ for $j=1,\ldots,n$. By Proposition \ref{prop-orbit-in-Sp} and Lemmas \ref{lemma-tau-properties}, \ref{lemma-delta-distr}, for all $i,j = 1,\ldots,n$ we have
\begin{align*}
\la D \psi(x_0) v_i, \Phi_j \ra = -\sum_{l=1}^d v_{il} \la \partial_l \tau_{x_0} \varphi, \delta_{x_0 + z_i} \ra = -\sum_{l=1}^d v_{il} \partial_l \varphi(z_i) = -D_{v_i} \varphi(z_j).
\end{align*}
Since the matrix (\ref{matrix}) is invertible, by Lemma \ref{lemma-matrix-m-m} we deduce that $\psi$ is an immersion on $\caln$ at $x_0$.
\end{proof}

As an immediate application of Lemma \ref{lemma-induced}, Proposition \ref{prop-orbit-in-Sp} and our previous findings (Propositions \ref{prop-psi-measure}, \ref{prop-psi-poly} and \ref{prop-psi-function}), we obtain the following examples of submanifolds generated by the orbit maps of the translation group.

\begin{examples}\label{ex-manifolds-HS-spaces}
Let $k \in \bbn$ be arbitrary, and let $\caln$ be an $m$-dimensional $C^k$-submanifold of $\bbr^d$. We assume that $\Phi \in \cals_{p+\frac{k}{2}}(\bbr^d)$ with a suitable $p \in \bbr$ belongs to one of the following three types:
\begin{itemize}
\item We choose $p \in \bbr$ such that $p + \frac{k}{2} < -\frac{d}{4}$, and let $\Phi = \mu$, where $\mu$ is a finite signed measure on $(\bbr^d,\calb(\bbr^d))$ with compact support such that $\mu(\bbr^d) \neq 0$.

\item We choose $p \in \bbr$ such that $p + \frac{k}{2} < -\frac{d}{4} - \frac{n}{2}$ for some $n \in \bbn$ with $m \leq n \leq d$, and let $\Phi = f$ be the polynomial $f : \bbr^d \to \bbr$ given by $f(x) = x_1 \cdot \ldots \cdot x_n$. Furthermore, we assume that $\caln \subset \bbr^m \times \{ 0 \}$.

\item We choose $p \in \bbr$ such that $p + \frac{k}{2} > \frac{d}{4} + \frac{1}{2}$, and let $\Phi = \varphi \in \cals_{p+\frac{k}{2}}(\bbr^d)$ be arbitrary. We assume there are $n \in \bbn$ with $m \leq n \leq d$, an $n$-dimensional subspace $E \subset \bbr^d$ such that $T \caln \subset \caln \times E$, and elements $v_1,\ldots,v_n \in E$ and $z_1,\ldots,z_n \in \bbr^d$ such that the matrix $( D_{v_i} \varphi(z_j) )_{i,j=1,\ldots,n} \in \bbr^{n \times n}$ is invertible.
\end{itemize}
Let $\psi := \xi_{\Phi} : \bbr^d \to \cals_{p+\frac{k}{2}}(\bbr^d)$ be the orbit map given by $\psi(x) = \tau_x \Phi$ for $x \in \bbr^d$. Then $\calm := \psi(\caln)$ is an $m$-dimensional $(\cals_{p+\frac{k}{2}}(\bbr^d), \ldots, \cals_p(\bbr^d))$-submanifold of class $C^k$, which is induced by $(\psi,\caln)$.
\end{examples}

\section{The general invariance result}\label{sec-SPDE-general}

In this section we provide the general invariance result. Let $(G,H)$ be separable Hilbert spaces with continuous embedding, and consider the SPDE (\ref{SPDE}) with continuous mappings $L : G \to H$ and $A : G \to \ell^2(H)$. Let $\calm$ be a $(G,H)$-submanifold of class $C^2$. We denote by $A(\calm)$ be the linear space of all mappings $A : \calm \to H$. Recall that $\Gamma(T \calm)$ denotes the subspace of all vector fields on $\calm$; see Definition \ref{def-vector-field}. In the following definition we consider the quotient space $A(\calm) / \Gamma(T \calm)$, and for each $A \in A(\calm)$ we denote by $[A]_{\Gamma(T \calm)}$ the corresponding equivalence class.

\begin{definition}\label{def-A-Strat}
Let $A,B \in \Gamma(T \calm)$ be two vector fields on $\calm$. We define the mapping
\begin{align*}
[A,B]_{\calm} \in A(\calm) / \Gamma(T \calm)
\end{align*}
as follows. For each local parametrization $\phi : V \to U \cap \calm$ a local representative of $[A,B]_{\calm}$ on $U \cap \calm$ is given by
\begin{align*}
\phi_{**} ( \phi_*^{-1} A|_{U \cap \calm}, \phi_*^{-1} B|_{U \cap \calm} ),
\end{align*}
where we recall the notation from Definition \ref{def-push}.
\end{definition}

\begin{remark}
Note that, according to Lemma \ref{lemma-second-tang}, the Definition \ref{def-A-Strat} of $[A,B]_{\calm}$ does not depend on the choice of the parametrization.
\end{remark}

\begin{theorem}\label{thm-SPDE}
The following statements are equivalent:
\begin{enumerate}
\item[(i)] The submanifold $\calm$ is locally invariant for the SPDE (\ref{SPDE}).

\item[(ii)] We have
\begin{align}\label{tang-A}
&A^j|_{\calm} \in \Gamma(T \calm), \quad j \in \bbn,
\\ \label{tang-L} &[ L|_{\calm} ]_{\Gamma(T \calm)} - \frac{1}{2} \sum_{j=1}^{\infty} [A^j|_{\calm}, A^j|_{\calm}]_{\calm} = [0]_{\Gamma(T \calm)}.
\end{align}

\item[(iii)] The mappings
\begin{align}\label{map-1-main}
&A|_{\calm} : (\calm,\| \cdot \|_H) \to (\ell^2(H),\| \cdot \|_{\ell^2(H)}),
\\ \label{map-2-main} &L|_{\calm} : (\calm,\| \cdot \|_H) \to (H,\| \cdot \|_H)
\end{align}
are continuous, and for each $y_0 \in \calm$ there exists a local martingale solution $Y$ to the SPDE (\ref{SPDE}) with $Y_0 = y_0$ and lifetime $\tau$ such that $Y^{\tau} \in \calm$ up to an evanescent set and the sample paths of $Y^{\tau}$ are continuous with respect to $\| \cdot \|_G$.
\end{enumerate}
\end{theorem}

\begin{proposition}\label{prop-SPDE}
Suppose that the submanifold $\calm$ is locally invariant for the SPDE (\ref{SPDE}). If the submanifold $\calm$ has one chart with a global parametrization $\phi : V \to \calm$, and the open set $V$ is globally invariant for the $\bbr^m$-valued SDE
\begin{align*}
\left\{
\begin{array}{rcl}
d X_t & = & \ell(X_t) dt + a(X_t) dW_t
\\ X_0 & = & x_0,
\end{array}
\right.
\end{align*}
where the continuous mappings $\ell : V \to \bbr^m$ and $a : V \to \ell^2(\bbr^m)$ are the unique solutions of the equations
\begin{align}\label{A-local}
A^j|_{\calm} &= \phi_* a^j, \quad j \in \bbn,
\\ \label{L-local} L|_{\calm} &= \phi_* \ell + \frac{1}{2} \sum_{j=1}^{\infty} \phi_{**}(a^j,a^j),
\end{align}
then the submanifold $\calm$ is globally invariant for the SPDE (\ref{SPDE}).
\end{proposition}

\begin{remark}
Choosing $G = H = \bbr^d$, we see that Theorem \ref{thm-SPDE} and Proposition \ref{prop-SPDE} cover the well-known situation of finite dimensional SDEs.
\end{remark}

Before we provide the proofs of Theorem \ref{thm-SPDE} and Proposition \ref{prop-SPDE}, let us state some consequences of these results. Consider the conditions
\begin{align}\label{condition-SPDE-1}
L|_{\calm} &\in \Gamma(T \calm), 
\\ \label{condition-SPDE-2} A^j|_{\calm} &\in \Gamma(T \calm), \quad j \in \bbn.
\end{align}
We are interested in finding an additional condition which ensures such that $\calm$ is locally invariant for the SPDE (\ref{SPDE}).

\begin{proposition}\label{prop-L-tang}
Suppose that conditions (\ref{condition-SPDE-1}) and (\ref{condition-SPDE-2}) is fulfilled. Then the following statements are equivalent:
\begin{enumerate}
\item[(i)] $\calm$ is locally invariant for the SPDE (\ref{SPDE}).

\item[(ii)] We have
\begin{align*}
\sum_{j=1}^{\infty} [A^j|_{\calm}, A^j|_{\calm}]_{\calm} = [0]_{\Gamma(T \calm)}.
\end{align*}
\end{enumerate}
\end{proposition}

\begin{proof}
This is a consequence of Theorem \ref{thm-SPDE}.
\end{proof}

We say that the submanifold $\calm$ is \emph{affine} if for any local parametrization $\phi : V \to U \cap \calm$ we have $D^2 \phi = 0$.

\begin{corollary}\label{cor-affine}
Suppose the submanifold $\calm$ is affine. Then the following statements are equivalent:
\begin{enumerate}
\item[(i)] $\calm$ is locally invariant for the SPDE (\ref{SPDE}).

\item[(ii)] We have (\ref{condition-SPDE-1}) and (\ref{condition-SPDE-2}).
\end{enumerate}
\end{corollary}

\begin{proof}
This is a consequence of Theorem \ref{thm-SPDE} and Proposition \ref{prop-L-tang}.
\end{proof}

\begin{remark}\label{rem-Strat-2}
Consider the situation $G=H$ and $A^j \in C^1(H)$ for all $j \in \bbn$. If $\sum_{j=1}^{\infty} D A^j(y) A^j(y)$ converges for each $y \in H$, and the mapping $\sum_{j=1}^{\infty} D A^j \cdot A^j$ is continuous, then we can rewrite the SPDE (\ref{SPDE}) in Stratonovich form as
\begin{align*}
\left\{
\begin{array}{rcl}
dY_t & = & K(Y_t)dt + A(Y_t) \circ dW_t
\\ Y_0 & = & y_0,
\end{array}
\right.
\end{align*}
where $K : H \to H$ is given by
\begin{align*}
K = L - \frac{1}{2} \sum_{j=1}^{\infty} D A^j \cdot A^j. 
\end{align*}
If we have (\ref{tang-A}), then by the decomposition (\ref{zerl-general}) from Proposition \ref{prop-decomp-Damir-scale} we have
\begin{align*}
[K|_{\calm}]_{\Gamma(T \calm)} = [ L|_{\calm} ]_{\Gamma(T \calm)} - \frac{1}{2} \sum_{j=1}^{\infty} [A^j|_{\calm}, A^j|_{\calm}]_{\calm},
\end{align*}
and hence condition (\ref{tang-L}) is equivalent to
\begin{align*}
K|_{\calm} \in \Gamma(T \calm).
\end{align*}
We will present a corresponding result for continuously embedded Hilbert spaces with an additional intermediate space later on; see Theorem \ref{thm-main-2} below.
\end{remark}

We can express the statement of Theorem \ref{thm-SPDE} in local coordinates as follows.

\begin{proposition}\label{prop-inv-para-2}
The following statements are equivalent:
\begin{enumerate}
\item[(i)] The submanifold $\calm$ is locally invariant for the SPDE (\ref{SPDE}).

\item[(ii)] For each local parametrization $\phi : V \to U \cap \calm$ there are continuous mappings $\ell : V \to \bbr^m$ and $a : V \to \ell^2(\bbr^m)$ which are the unique solutions of the equations
\begin{align}\label{A-local-para}
A^j|_{U \cap \calm} &= \phi_* a^j, \quad j \in \bbn,
\\ \label{L-local-para} L|_{U \cap \calm} &= \phi_* \ell + \frac{1}{2} \sum_{j=1}^{\infty} \phi_{**}(a^j,a^j).
\end{align}

\item[(iii)] For each $y \in \calm$ there exist a local parametrization $\phi : V \to U \cap \calm$ around $y$ and continuous mappings $\ell : V \to \bbr^m$ and $a : V \to \ell^2(\bbr^m)$ which are the unique solutions of the equations (\ref{A-local-para}) and (\ref{L-local-para}).
\end{enumerate}
\end{proposition}

\begin{proof}
This is an immediate consequence of Theorem \ref{thm-SPDE}.
\end{proof}

In the following two results we assume that the submanifold $\calm$ is induced $(\psi,\caln)$, where $\caln$ is an $m$-dimensional $C^2$-submanifold of $\bbr^d$, and $\psi \in C^2(\bbr^d;H)$ is a $C^2$-immersion on $\caln$ such that $\psi|_{\caln} : \caln \to \psi(\caln)$ is a homeomorphism; see Definition \ref{def-manifold-embedded}.

\begin{theorem}\label{thm-inv-embedded}
The following statements are equivalent:
\begin{enumerate}
\item[(i)] The submanifold $\calm$ is locally invariant for the SPDE (\ref{SPDE}).

\item[(ii)] The submanifold $\caln$ is locally invariant for the SDE
\begin{align}\label{SDE-embedded}
\left\{
\begin{array}{rcl}
dX_t & = & b(X_t) dt + \sigma(X_t) dW_t
\\ X_0 & = & x_0,
\end{array}
\right.
\end{align}
where the continuous mappings\footnote{If the SDE (\ref{SDE-embedded}) is locally invariant, then it suffices to specify the coefficients $b$ and $\sigma$ on the submanifold $\caln$.} $b : \caln \to \bbr^d$ and $\sigma : \caln \to \ell^2(\bbr^d)$ are the unique solutions of the equations
\begin{align}\label{eqn-push-1}
A^j|_{\calm} &= \psi_* \sigma^j, \quad j \in \bbn,
\\ \label{eqn-push-2} L|_{\calm} &= \psi_* b + \frac{1}{2} \sum_{j=1}^{\infty} \psi_{**}(\sigma^j,\sigma^j).
\end{align}
\end{enumerate}
\end{theorem}

\begin{proof}
(i) $\Rightarrow$ (ii): Let $y \in \calm$ be arbitrary, and let $\varphi : V \to W \cap \caln$ be a local parametrization around $x := \psi^{-1}(y) \in \caln$. By Lemma \ref{lemma-induced} there exists an open neighborhood $U \subset H$ of $y$ such that $\phi := \psi \circ \varphi : V \to U \cap \calm$ is a local parametrization around $y$. Furthermore, by Proposition \ref{prop-inv-para-2} there are continuous mappings $\ell : V \to \bbr^m$ and $a : V \to \ell^2(\bbr^m)$ which are the unique solutions of the equations (\ref{A-local-para}) and (\ref{L-local-para}). We define the continuous mappings $b : W \cap \caln \to \bbr^d$ and $\sigma : W \cap \caln \to \ell^2(\bbr^d)$ as
\begin{align*}
\sigma^j &:= \varphi_{*} a^j, \quad j \in \bbn,
\\ b &:= \varphi_{*} \ell + \frac{1}{2} \sum_{j=1}^{\infty} \varphi_{**}(a^j,a^j).
\end{align*}
Since $y \in \calm$ was arbitrary, by Proposition \ref{prop-inv-para-2} we deduce that the submanifold $\caln$ is locally invariant for the SDE (\ref{SDE-embedded}). Furthermore, by Lemma \ref{lemma-push-chain-rule} we obtain
\begin{align*}
A^j|_{U \cap \calm} = \phi_* a^j = \psi_* \varphi_{*} a^j = \psi_* \sigma^j, \quad j \in \bbn
\end{align*}
as well as
\begin{align*}
L|_{U \cap \calm} &= \phi_* \ell + \frac{1}{2} \sum_{j=1}^{\infty} \phi_{**}(a^j,a^j)
\\ &= \psi_* \varphi_* \ell + \frac{1}{2} \sum_{j=1}^{\infty} \big( \psi_{**} ( \varphi_* a^j, \varphi_* a^j ) + \psi_* \varphi_{**} (a^j,a^j) \big)
\\ &= \psi_* b + \frac{1}{2} \sum_{j=1}^{\infty} \psi_{**}(\sigma^j,\sigma^j).
\end{align*}
Since the element $y \in \calm$ was arbitrary, this procedure provides us with continuous mappings $b : \caln \to \bbr^d$ and $\sigma : \caln \to \ell^2(\bbr^d)$ which are the unique solutions of the equations (\ref{eqn-push-1}) and (\ref{eqn-push-2}).

\noindent (ii) $\Rightarrow$ (i): Let $y \in \calm$ be arbitrary, and let $\varphi : V \to W \cap \caln$ be a local parametrization around $x := \psi^{-1}(y) \in \caln$. By Lemma \ref{lemma-induced} there exists an open neighborhood $U \subset H$ of $y$ such that $\phi := \psi \circ \varphi : V \to U \cap \calm$ is a local parametrization around $y$. Since $\caln$ is locally invariant for the SDE (\ref{SDE-embedded}), by Proposition \ref{prop-inv-para-2} there are continuous mappings $\ell : V \to \bbr^m$ and $a : V \to \ell^2(\bbr^m)$ which are the unique solutions of the equations
\begin{align*}
\sigma^j|_{W \cap \caln} &= \varphi_* a^j, \quad j \in \bbn,
\\ b|_{W \cap \caln} &= \varphi_* \ell + \frac{1}{2} \sum_{j=1}^{\infty} \varphi_{**}(a^j,a^j).
\end{align*}
By Lemma \ref{lemma-push-chain-rule} we obtain
\begin{align*}
A^j|_{U \cap \calm} = \psi_* \sigma^j|_{W \cap \caln} = \psi_* \varphi_{*} a^j = \phi_{*} a^j, \quad j \in \bbn
\end{align*}
as well as
\begin{align*}
L|_{U \cap \calm} &= \psi_* b|_{W \cap \caln} + \frac{1}{2} \sum_{j=1}^{\infty} \psi_{**}(\sigma^j|_{W \cap \caln},\sigma^j|_{W \cap \caln})
\\ &= \psi_* \varphi_* \ell + \frac{1}{2} \sum_{j=1}^{\infty} \big( \psi_{**} ( \varphi_* a^j, \varphi_* a^j ) + \psi_* \varphi_{**} (a^j,a^j) \big)
\\ &= \phi_{*} \ell + \frac{1}{2} \sum_{j=1}^{\infty} \phi_{**}(a^j,a^j).
\end{align*}
Therefore, by Proposition \ref{prop-inv-para-2} the submanifold $\calm$ is locally invariant for the SPDE (\ref{SPDE}).
\end{proof}

For the next result, recall that the submanifold $\calm$ has one chart if $\caln$ has one chart; see Lemma \ref{lemma-M-N-one-chart}.

\begin{proposition}\label{prop-inv-embedded}
If the submanifold $\calm$ is locally invariant for the SPDE (\ref{SPDE}) and the submanifold $\caln$ has one chart with a global parametrization $\varphi : V \to \caln$, then for continuous mappings $\ell : V \to \bbr^m$ and $a : V \to \ell^2(\bbr^m)$ the following statements are equivalent:
\begin{enumerate}
\item[(i)] $\ell : V \to \bbr^m$ and $a : V \to \ell^2(\bbr^m)$ are the unique solutions of the equations (\ref{A-local}) and (\ref{L-local}).

\item[(ii)] $\ell : V \to \bbr^m$ and $a : V \to \ell^2(\bbr^m)$ are the unique solutions of the equations
\begin{align}\label{eqn-push-1b}
\sigma^j &= \varphi_* a^j, \quad j \in \bbn,
\\ \label{eqn-push-2b} b &= \varphi_* \ell + \frac{1}{2} \sum_{j=1}^{\infty} \varphi_{**}(a^j,a^j),
\end{align}
where the continuous mappings $b : \caln \to \bbr^d$ and $\sigma : \caln \to \ell^2(\bbr^d)$ are the unique solutions of the equations (\ref{eqn-push-1}) and (\ref{eqn-push-2}).
\end{enumerate}
If any of the previous two conditions is fulfilled and the open set $V$ is globally invariant for the $\bbr^m$-valued SDE
\begin{align*}
\left\{
\begin{array}{rcl}
d \Xi_t & = & \ell(\Xi_t) dt + a(\Xi_t) dW_t
\\ \Xi_0 & = & \xi_0,
\end{array}
\right.
\end{align*}
then the submanifold $\calm$ is globally invariant for the SPDE (\ref{SPDE}), and the submanifold $\caln$ is globally invariant for the SPDE (\ref{SDE-embedded}).
\end{proposition}

\begin{proof}
By Lemma \ref{lemma-M-N-one-chart} the submanifold $\calm$ has one chart with global parametrization $\phi := \psi \circ \varphi : V \to \calm$.

\noindent(i) $\Rightarrow$ (ii): Taking into account Lemma \ref{lemma-push-chain-rule}, by (\ref{eqn-push-1}) and (\ref{A-local}) we obtain
\begin{align*}
\sigma^j = \psi_*^{-1} \psi_* \sigma^j = \psi_*^{-1} A^j|_{\calm} = \psi_*^{-1} \phi_* a^j = \psi_*^{-1} \psi_* \varphi_* a^j = \varphi_* a^j, \quad j \in \bbn,
\end{align*}
and by (\ref{eqn-push-2}) and (\ref{L-local}) we obtain
\begin{align*}
&b - \frac{1}{2} \sum_{j=1}^{\infty} \varphi_{**}(a^j,a^j) = \psi_*^{-1} \psi_* \bigg( b - \frac{1}{2} \sum_{j=1}^{\infty} \varphi_{**}(a^j,a^j) \bigg)
\\ &= \psi_*^{-1} \bigg( L|_{\calm} - \frac{1}{2} \sum_{j=1}^{\infty} \big( \psi_{**}(\varphi_* a^j, \varphi_* a^j) + \psi_* \varphi_{**}(a^j,a^j) \big) \bigg)
\\ &= \psi_*^{-1} \bigg( L|_{\calm} - \frac{1}{2} \sum_{j=1}^{\infty} \phi_{**}(a^j,a^j) \bigg) = \psi_*^{-1} \phi_* \ell = \psi_*^{-1} \psi_* \varphi_* \ell = \varphi_* \ell.
\end{align*}

\noindent(ii) $\Rightarrow$ (i): Taking into account Lemma \ref{lemma-push-chain-rule}, by (\ref{eqn-push-1}) and (\ref{eqn-push-1b}) we obtain
\begin{align*}
A^j|_{\calm} = \psi_* \sigma^j = \psi_* \varphi_* a^j = \phi_* a^j, \quad j \in \bbn,
\end{align*}
and by  (\ref{eqn-push-2}) and (\ref{eqn-push-2b}) we obtain
\begin{align*}
L|_{\calm} &= \psi_* b + \frac{1}{2} \sum_{j=1}^{\infty} \psi_{**}(\sigma^j,\sigma^j) = \psi_* b + \frac{1}{2} \sum_{j=1}^{\infty} \psi_{**}(\varphi_* a^j,\varphi_* a^j)
\\ &= \psi_* b + \frac{1}{2} \sum_{j=1}^{\infty} \big( \phi_{**}(a^j,a^j) - \psi_* \varphi_{**}(a^j,a^j) \big)
\\ &= \psi_* \bigg( b - \frac{1}{2} \sum_{j=1}^{\infty} \varphi_{**}(a^j,a^j) \bigg) + \frac{1}{2} \sum_{j=1}^{\infty} \phi_{**}(a^j,a^j)
\\ &= \psi_* \varphi_* \ell + \frac{1}{2} \sum_{j=1}^{\infty} \phi_{**}(a^j,a^j) = \phi_* \ell + \frac{1}{2} \sum_{j=1}^{\infty} \phi_{**}(a^j,a^j).
\end{align*}
The additional statement is a consequence of Proposition \ref{prop-SPDE}.
\end{proof}

Now, we approach the proofs of Theorem \ref{thm-SPDE} and Proposition \ref{prop-SPDE}. For this purpose, we prepare some auxiliary results. For normed spaces $X_1,\ldots,X_n$ and $Y$ we denote by $L(X_1,\ldots,X_n;Y)$ the space of all continuous $n$-multilinear maps $T : X_1 \times \ldots \times X_n \to Y$.

\begin{lemma}\label{lemma-evaluation}
Let $X,Y,Z$ be Banach spaces. Then the following statements are true:
\begin{enumerate}
\item The mapping
\begin{align*}
\Phi : L(X,Y) \times \ell^2(X) \to \ell^2(Y), \quad (T,x) \mapsto ( T x^j )_{j \in \bbn}
\end{align*}
belongs to $L(L(X,Y), \ell^2(X); \ell^2(Y))$, and we have $\|  \Phi \| \leq 1$.

\item The mapping
\begin{align*}
\Psi : L(X,Y;Z) \times \ell^2(X) \times \ell^2(Y) \to \ell^1(Z), \quad (T,x,y) \mapsto \big( T(x^j,y^j) \big)_{j \in \bbn}
\end{align*}
belongs to $L(L(X,Y;Z), \ell^2(X), \ell^2(Y); \ell^1(Z))$, and we have $\| \Psi \| \leq 1$.

\item The mapping 
\begin{align*}
z' : \ell^1(Z) \to Z, \quad z \mapsto \sum_{j=1}^{\infty} z^j
\end{align*}
belongs to $L(\ell^1(Z),Z)$, and we have $\| z' \| \leq 1$.
\end{enumerate}
\end{lemma}

\begin{proof}
For all $(T,x) \in L(X,Y) \times \ell^2(X)$ we have
\begin{align*}
\| \Phi(T,x) \|_{\ell^2(Y)}^2 = \sum_{j=1}^{\infty} \| T x^j \|^2 \leq \sum_{j=1}^{\infty} \| T \|^2 \| x^j \|^2 = \| T \|^2 \| x \|_{\ell^2(X)}^2.
\end{align*}
Furthermore, by the Cauchy-Schwarz inequality, for all $(T,x,y) \in L(X,Y;Z) \times \ell^2(X) \times \ell^2(Y)$ we have
\begin{align*}
\| \Psi(T,x,y) \|_{\ell^1(X)} = \sum_{j=1}^{\infty} \| T(x^j,y^j) \| \leq \sum_{j=1}^{\infty} \| T \| \| x^j \| \| y^j \| \leq \| T \| \| x \|_{\ell^2(X)} \| y \|_{\ell^2(X)}
\end{align*}
Moreover, for all $z \in \ell^1(Z)$ we have
\begin{align*}
\| z'(z) \| = \bigg\| \sum_{j=1}^{\infty} z^j \bigg\| \leq \sum_{j=1}^{\infty} \| z^j \| = \| z \|_{\ell^1(Z)},
\end{align*}
completing the proof.
\end{proof}

Recall the notation introduced in Definition \ref{def-push}. More generally, for a mapping $a : V \to \ell^2(\bbr^m)$ we can define $\phi_* : U \cap \calm \to \ell^2(H)$, and for a mapping $A : U \cap \calm \to \ell^2(H)$ with $A^j \in \Gamma(T \calm_U)$ we can define $\phi_*^{-1} : V \to \ell^2(\bbr^m)$.

\begin{lemma}\label{lemma-push}
Let $\phi : V \to U \cap \calm$ be a local parametrization of $\calm$, and set $\calm_U := U \cap \calm$. Then following statements are true:
\begin{enumerate}
\item If $a : V \to \ell^2(\bbr^m)$ is continuous, then $\phi_* a : (\calm_U,\| \cdot \|_H) \to \ell^2(H)$ is continuous.

\item If $a,b : V \to \ell^2(\bbr^m)$ are continuous, then the mapping
\begin{align*}
\sum_{j=1}^{\infty} \phi_{**}(a^j, b^j) : (\calm_U,\| \cdot \|_H) \to H
\end{align*}
is well-defined an continuous.

\item If $A : (\calm_U,\| \cdot \|_G) \to \ell^2(H)$ is continuous with $A^j \in \Gamma(T \calm_U)$ for each $j \in \bbn$, then $\phi_*^{-1} A : V \to \ell^2(\bbr^m)$ is continuous.

\item If $A \in \Gamma(T \calm_U)$ is a vector field, then we have $\phi_* \phi_*^{-1} A = A$.
\end{enumerate}
\end{lemma}

\begin{proof}
The first three statements follow Lemma \ref{lemma-evaluation}, and the last statement is easily checked. For the proof of the third statement we also take into account \cite[Prop. 6.1.1]{fillnm}, and that by Proposition \ref{prop-manifold-scaled} we have $\phi \in C(V;G)$.
\end{proof}

Now, let $\phi : V \to U \cap \calm$ be a local parametrization of $\calm$, and set $\calm_U := U \cap \calm$. We assume there exists $\psi \in C^1(H;\bbr^m)$ such that $\phi^{-1} = \psi|_{\calm_U}$, and we have 
\begin{align}\label{inverse-coordinate}
D \psi(y)|_{T_y \calm} = D \phi(x)^{-1} \quad \text{for all $y \in \calm_U$,}
\end{align}
where $x := \psi(y) \in V$. For a mapping $A : U \cap \calm \to \ell^2(H)$ we define $\psi_* A : V \to \ell^2(\bbr^m)$ as $\psi_* A := (\psi_* A^j)_{j \in \bbn}$ with
\begin{align*}
(\psi_* A^j)(x) := D \psi(y) A^j(y), \quad x \in V,
\end{align*}
where $y := \phi(x) \in U \cap \calm$.

\begin{lemma}\label{lemma-push-cont}
If $A : (\calm_U,\| \cdot \|_G) \to \ell^2(H)$ is continuous, then $\psi_* A : V \to \ell^2(\bbr^m)$ is continuous.
\end{lemma}

\begin{proof}
By Proposition \ref{prop-manifold-scaled} we have $\phi \in C(V;G)$. Therefore, the assertion is a consequence of Lemma \ref{lemma-evaluation}.
\end{proof}

\begin{proposition}\label{prop-tangential}
For a mapping $A : \calm_U \to H$ the following statements are equivalent:
\begin{enumerate}
\item[(i)] We have $A \in \Gamma(T \calm_U)$.

\item[(ii)] We have $A = \phi_* \psi_* A$.
\end{enumerate}
If any of the previous two conditions is fulfilled, then we have $\psi_* A = \phi_*^{-1} A$.
\end{proposition}

\begin{proof}
(i) $\Rightarrow$ (ii): Noting (\ref{inverse-coordinate}), we see that $\psi_* A = \phi_*^{-1} A$. Therefore, using Lemma \ref{lemma-push} we obtain
\begin{align*}
\phi_* \psi_* A = \phi_* \phi_*^{-1} A = A.
\end{align*}

\noindent (ii) $\Rightarrow$ (i): Taking into account the definition of $\phi_*$, we obtain
\begin{align*}
A = \phi_* \psi_* A \in \Gamma(T \calm_U),
\end{align*}
completing the proof.
\end{proof}

\begin{proposition}\label{prop-second-coordinate-zero}
Suppose that $\psi \in L(H,\bbr^m)$. Then for all mappings $A,B : U \cap \calm \to H$ we have
\begin{align*}
\psi_* \phi_{**} ( \psi_* A, \psi_* B ) = 0.
\end{align*}
\end{proposition}

\begin{proof}
Note that $\psi \circ \phi = \Id_V$ and that $D^2 \psi = 0$, because $\psi$ is linear. Therefore, by the second order chain rule (Lemma \ref{lemma-second-order-chain}), for each $x \in V$ we have
\begin{align*}
D \psi(y) \circ D^2 \phi(x) = D^2( \psi \circ \phi )(x) = 0,
\end{align*}
where $y := \phi(x) \in U \cap \calm$. We set $a := \psi_* A$ and $b := \psi_* B$ as well as $C := \phi_{**}(a,b)$. Let $x \in V$ be arbitrary, and set $y := \phi(x) \in U \cap \calm$. Then we obtain
\begin{align*}
(\psi_* \phi_{**} ( \psi_* A, \psi_* B ))(x) = (\psi_* C)(x) = D \psi(y) C(y) = D \psi(y) D^2 \phi(x)(a(x),b(x)) = 0,
\end{align*}
completing the proof.
\end{proof}

\begin{lemma}\label{lemma-compare-1}
Let $E \subset G$ be a subset, let $K : (E,\| \cdot \|_G) \to (H,\| \cdot \|_H)$ be a continuous function, let $y_0 \in G$ be arbitrary, let $\tau > 0$ be a positive constant, and let $Y : [0,\tau] \to (E,\| \cdot \|_G)$ be a continuous mapping with $Y_0 = y_0$ such that
\begin{align*}
\int_0^t K(Y_s) ds = 0 \quad \text{for all $t \in [0,\tau]$.}
\end{align*}
Then we have $K(y_0) = 0$. 
\end{lemma}

\begin{proof}
Let $y \in H'$ be an arbitrary continuous linear functional. By assumption we have
\begin{align*}
\int_{[0,t]} y'(K(Y_s)) ds = 0 \quad \text{for all $t \in [0,\tau]$.}
\end{align*}
By a monotone class argument, we even have
\begin{align*}
\int_B y'(K(Y_s)) ds = 0 \quad \text{for each $B \in \calb([0,\tau])$,}
\end{align*}
and therefore
\begin{align*}
y'(K(Y_s)) = 0 \quad \text{for $\lambda$-almost all $s \in [0,\tau]$.}
\end{align*}
By the continuity of the mapping $y' \circ K \circ Y : [0,\tau] \to \bbr$ we deduce that
\begin{align*}
y'(K(Y_s)) = 0 \quad \text{for all $s \in [0,\tau]$,}
\end{align*}
and hence, in particular $y'(K(y_0)) = 0$. Since the functional $y' \in H'$ was arbitrary, we arrive at $K(y_0) = 0$.
\end{proof}

\begin{lemma}\label{lemma-compare-2}
Let $E \subset G$ be a subset, let $L : (E,\| \cdot \|_G) \to (H,\| \cdot \|_H)$ and $A : (E,\| \cdot \|_G) \to (\ell^2(H),\| \cdot \|_{\ell^2(H)})$ be continuous mappings, let $y_0 \in G$ be arbitrary, let $Y$ be an $E$-valued process with $Y_0 = y_0$ such that the sample paths are continuous with respect to $\| \cdot \|_G$, and let $\tau > 0$ be a positive stopping time such that $\bbp$-almost surely
\begin{align*}
\int_0^{t \wedge \tau} L(Y_s) ds + \int_0^{t \wedge \tau} A(Y_s) dW_s = 0 \quad \text{for all $t \in \bbr_+$.}
\end{align*}
Then we have $L(y_0) = 0$ and $A(y_0) = 0$.
\end{lemma}

\begin{proof}
We define the $H$-valued processes $B$ and $M$ as
\begin{align*}
B_t &:= \int_0^{t \wedge \tau} L(Y_s) ds, \quad t \in \bbr_+,
\\ M_t &:= \int_0^{t \wedge \tau} A(Y_s) dW_s, \quad t \in \bbr_+.
\end{align*}
Then we have $B + M = 0$ up to an evanescent set. Let $\zeta \in H$ be arbitrary. In the terminology of \cite[Def. I.4.21.b]{Jacod-Shiryaev} the process $\langle \zeta,B \rangle_{H} + \langle \zeta,M \rangle_{H}$ is a special semimartingale with predictable finite variation part $\langle \zeta,B \rangle_{H}$ and local martingale part $\langle \zeta,M \rangle_{H}$. Since $\langle \zeta,B \rangle_{H} + \langle \zeta,M \rangle_{H} = 0$ and the decomposition of a special semimartingale is unique (see \cite[Cor. I.3.16]{Jacod-Shiryaev}), we deduce that $\langle \zeta,B \rangle_{H} = 0$ and $\langle \zeta,M \rangle_{H} = 0$ up to an evanescent set. Since $\zeta \in H$ was arbitrary, by separability of $H$ we infer that $B = 0$ and $M = 0$ up to an evanescent set. Using Lemma \ref{lemma-compare-1} we deduce that $L(y_0) = 0$. Furthermore, we also have $\langle M \rangle = 0$ up to an evanescent set, where $\langle M \rangle$ denotes the quadratic variation according to \cite[Def. 2.9]{Atma-book}. Therefore, by Remark \ref{rem-SPDEs} and \cite[Thm. 2.3]{Atma-book} we obtain $\bbp$-almost surely
\begin{align*}
\int_0^{t \wedge \tau} \| A(Y_s) \|_{\ell^2(H)}^2 ds = 0, \quad t \in \bbr_+.
\end{align*}
By Lemma \ref{lemma-compare-1} it follows that
\begin{align*}
\| A(y_0) \|_{\ell^2(H)} = 0,
\end{align*}
and hence $A(y_0) = 0$, completing the proof.
\end{proof}

Now, we consider the $\bbr^m$-valued SDE
\begin{align}\label{SDE}
\left\{
\begin{array}{rcl}
dX_t & = & \ell(X_t) dt + a(X_t) dW_t
\\ X_0 & = & x_0
\end{array}
\right.
\end{align}
with continuous mappings $\ell : V \to \bbr^m$ and $a : V \to \ell^2(\bbr^m)$. 

\begin{lemma}\label{lemma-V-SDE}
Every open subset $V \subset \bbr^m$ is a $C^{\infty}$-submanifold of $\bbr^m$, which is locally invariant for the SDE (\ref{SDE}).
\end{lemma}

\begin{proof}
It is obvious that $V$ is a $C^{\infty}$-submanifold of $\bbr^m$. Let $x_0 \in V$ be arbitrary. Since $V$ is open, there exists a compact, convex neighborhood $K \subset V$ of $x_0$. Let $P_K : \bbr^m \to K$ be the orthogonal projection on $K$. We consider the SDE
\begin{align}\label{SDE-weak-sol}
\left\{
\begin{array}{rcl}
d\bar{X}_t & = & \bar{\ell}(\bar{X}_t) dt + \bar{a}(\bar{X}_t) dW_t
\\ \bar{X}_0 & = & x_0,
\end{array}
\right.
\end{align}
where the coefficients are given by
\begin{align*}
\bar{\ell} &:= \ell \circ P_K : \bbr^m \to \bbr^m,
\\ \bar{a} &:= a \circ P_K : \bbr^m \to \ell^2(\bbr^m).
\end{align*}
Note that $\bar{\ell}$ and $\bar{a}$ are continuous and bounded. Hence, by Remark \ref{rem-global-weak-solution} there exists a global weak solution $(\bbb,W,\bar{X})$ to the SDE (\ref{SDE-weak-sol}) with $\bar{X}_0 = x_0$. Now, we define the positive stopping time $\tau > 0$ as
\begin{align*}
\tau := \inf \{ t \in \bbr_+ : \bar{X}_t \notin K \}.
\end{align*}
Setting $X := \bar{X}^{\tau}$, we have $X^{\tau} \in K \subset V$, and, since $\ell|_K = \bar{\ell}|_K$ and $a|_K = \bar{a}|_K$, the triplet $(\bbb,W,X)$ is a local weak solution to the SDE (\ref{SDE}) with $X_0 = x_0$ and lifetime $\tau$.
\end{proof}

Now, we are ready to provide the proof of Theorem \ref{thm-SPDE}.

\begin{proof}[Proof of Theorem \ref{thm-SPDE}]
(i) $\Rightarrow$ (ii): Let $y_0 \in \calm$ be arbitrary. According to Proposition \ref{prop-linear-inverse} there exist a local parametrization $\phi : V \to U \cap \calm$ around $y_0$ and a bounded linear operator $\psi \in L(H,\bbr^m)$ such that $\phi^{-1} = \psi|_{U \cap \calm}$ and we have
\begin{align*}
D \psi(y)|_{T_y \calm} = D \phi(x)^{-1} \quad \text{for all $y \in U \cap \calm$,}
\end{align*}
where $x := \psi(y) \in V$. By Proposition \ref{prop-manifold-scaled} we have $\phi \in C(V;G) \cap C^2(V;H)$. Now, let $y \in U \cap \calm$ be arbitrary, and set $x := \psi(y) \in V$. Since the submanifold $\calm$ is locally invariant for the SPDE (\ref{SPDE}), there exist a positive stopping time $\tau > 0$ and a local martingale solution $Y$ to (\ref{SPDE}) with $Y_0 = y$ and lifetime $\tau$ such that $Y^{\tau} \in \calm$ up to an evanescent set. Since $U$ is an open subset of $H$ and the sample paths of $Y$ are continuous with respect to $\| \cdot \|_H$, we may assume that $Y^{\tau} \in U \cap \calm$ up to an evanescent set. Now, we define the continuous $\bbr^m$-valued process $X := \psi(Y)$. Then we have $X^{\tau} \in V$, and since $\psi$ is linear, the process $X$ is a local weak solution to the SDE
\begin{align*}
\left\{
\begin{array}{rcl}
dX_t & = & (\psi_* L)(X_t) dt + (\psi_* A)(X_t) d W_t
\\ X_0 & = & x
\end{array}
\right.
\end{align*}
with lifetime $\tau$. The sample paths of $Y^{\tau} = \phi(X^{\tau})$ are continuous with respect to $\| \cdot \|_G$, because $\phi \in C(V;G)$. Since also $\phi \in C^2(V;H)$, by It\^{o}'s formula (see \cite[Thm. 2.3.1]{fillnm}) we obtain that the process $Y$ is a local martingale solution to the SPDE
\begin{align*}
\left\{
\begin{array}{rcl}
dY_t & = & \big( (\phi_* \psi_* L)(Y_t) + \frac{1}{2} \sum_{j=1}^{\infty} \phi_{**}(\psi_* A^j, \psi_* A^j)(Y_t) \big) dt + (\phi_* \psi_* A)(Y_t) dW_t
\\ Y_0 & = & y
\end{array}
\right.
\end{align*}
with lifetime $\tau$. On the other hand, the process $Y$ is a local martingale solution to the original SPDE (\ref{SPDE}) with $Y_0 = y$ and lifetime $\tau$. We set $\calm_U := U \cap \calm$. By Lemmas \ref{lemma-push} and \ref{lemma-push-cont}, the mappings
\begin{align*}
&\phi_* \psi_* A|_{\calm_U} : (\calm_U,\| \cdot \|_G) \to (\ell^2(H),\| \cdot \|_{\ell^2(H)}),
\\ &\phi_* \psi_* L|_{\calm_U} + \frac{1}{2} \sum_{j=1}^{\infty} \phi_{**}(\psi_* A^j|_{\calm_U}, \psi_* A^j|_{\calm_U}) : (\calm_U,\| \cdot \|_G) \to (H,\| \cdot \|_H)
\end{align*}
are continuous. Therefore, and since the sample paths of $Y^{\tau}$ are continuous with respect to $\| \cdot \|_G$, we may apply Lemma \ref{lemma-compare-2}, which gives us
\begin{align*}
A^j|_{\calm_U} &= \phi_* \psi_* A^j|_{\calm_U}, \quad j \in \bbn,
\\ L|_{\calm_U} &= \phi_* \psi_* L|_{\calm_U} + \frac{1}{2} \sum_{j=1}^{\infty} \phi_{**}( \psi_* A^j|_{\calm_U}, \psi_* A^j|_{\calm_U} ).
\end{align*}
Therefore, by Proposition \ref{prop-tangential} we deduce that 
\begin{align*}
A^j|_{\calm_U} \in \Gamma(T \calm_U), \quad j \in \bbn.
\end{align*}
Furthermore, using Proposition \ref{prop-second-coordinate-zero} we obtain
\begin{align*}
&L|_{\calm_U} - \frac{1}{2} \sum_{j=1}^{\infty} \phi_{**}( \psi_* A^j|_{\calm_U}, \psi_* A^j|_{\calm_U} )
\\ &= \phi_* \psi_* \bigg( L|_{\calm_U} - \frac{1}{2} \sum_{j=1}^{\infty} \phi_{**}( \psi_* A^j|_{\calm_U}, \psi_* A^j|_{\calm_U} ) \bigg).
\end{align*}
Therefore, by Proposition \ref{prop-tangential} we deduce that 
\begin{align*}
L|_{\calm_U} - \frac{1}{2} \sum_{j=1}^{\infty} \phi_{**}( \phi_*^{-1} A^j|_{\calm_U}, \phi_*^{-1} A^j|_{\calm_U} ) \in \Gamma(T \calm_U).
\end{align*}
Since the point $y_0 \in \calm$ chosen at the beginning of this proof was arbitrary, we deduce (\ref{tang-A}) and (\ref{tang-L}).

\noindent(ii) $\Rightarrow$ (iii): Let $y_0 \in \calm$ be arbitrary, and let $\phi : V \to U \cap \calm$ be an arbitrary local parametrization around $y_0$. By Proposition \ref{prop-manifold-scaled} we have $\phi \in C(V;G) \cap C^2(V;H)$. We set $x_0 := \phi^{-1}(y_0) \in V$. By Lemma \ref{lemma-push} the mappings
\begin{align*}
a &:= \phi_*^{-1} A|_{\calm_U} : V \to \ell^2(\bbr^m)
\\ \ell &:= \phi_*^{-1} \bigg( L|_{\calm_U} - \frac{1}{2} \sum_{j=1}^{\infty} \phi_{**} ( a^j, a^j ) \bigg) : V \to \bbr^m
\end{align*}
are continuous. Therefore, by Lemma \ref{lemma-V-SDE} the open set $V$ is locally invariant for the SDE (\ref{SDE}). Hence, there exist a stopping time $\tau > 0$ and a local weak solution $X$ to (\ref{SDE}) with $X_0 = x_0$ and lifetime $\tau$ such that $X^{\tau} \in V$ up to an evanescent set. We define the $\calm$-valued process $Y := \phi(X)$. Then we have $Y^{\tau} \in U \cap \calm$. Furthermore, since $\phi \in C(V;G)$, the sample paths of $Y^{\tau}$ are continuous with respect to $\| \cdot \|_G$. Taking into account Lemma \ref{lemma-push}, the mapping
\begin{align}\label{A-id-proof-Ito}
A|_{\calm_U} = \phi_* \phi_*^{-1} A|_{\calm_U} = \phi_* a : (\calm_U,\| \cdot \|_H) \to (\ell^2(H),\| \cdot \|_{\ell^2(H)})
\end{align}
is continuous. Furthermore, taking into account Lemma \ref{lemma-push} we have
\begin{align*}
L|_{\calm_U} - \frac{1}{2} \sum_{j=1}^{\infty} \phi_{**}( a^j,a^j ) = \phi_* \phi_*^{-1} \bigg( L|_{\calm_U} - \frac{1}{2} \sum_{j=1}^{\infty} \phi_{**}( a^j,a^j ) \bigg) = \phi_* \ell,
\end{align*}
and hence, by Lemma \ref{lemma-push} the mapping
\begin{align}\label{L-id-proof-Ito}
L|_{\calm_U} = \phi_* \ell + \frac{1}{2} \sum_{j=1}^{\infty} \phi_{**}( a^j, a^j ) : (\calm_U,\| \cdot \|_H) \to (H,\| \cdot \|_H)
\end{align}
is continuous. Moreover, by It\^{o}'s formula (see \cite[Thm. 2.3.1]{fillnm}) and relations (\ref{A-id-proof-Ito}), (\ref{L-id-proof-Ito}) we obtain that $Y^{\tau}$ is a local martingale solution to the SPDE
\begin{align*}
dY_t &= \bigg( (\phi_* \ell)(Y_t) + \frac{1}{2} \sum_{j=1}^{\infty} \phi_{**}(a^j,a^j)(Y_t) \bigg) dt + (\phi_* a)(Y_t) dW_t
\\ &= L(Y_t) dt + A(Y_t) dW_t,
\end{align*}
which is just the original SPDE (\ref{SPDE}), with $Y_0 = y_0$ and lifetime $\tau$. This proves that $\calm$ is locally invariant for the SPDE (\ref{SPDE}).

\noindent(iii) $\Rightarrow$ (i): This implication is obvious.
\end{proof}

\begin{proof}[Proof of Proposition \ref{prop-SPDE}]
This follows from inspecting the proof of the implication (ii) $\Rightarrow$ (iii) from Theorem \ref{thm-SPDE}.
\end{proof}

Now, let $H_0$ be another separable Hilbert space such that $(G,H_0,H)$ are continuously embedded, and suppose that $\calm$ is a $(G,H_0,H)$-submanifold of class $C^2$.

\begin{theorem}\label{thm-main-2}
Suppose that for each $j \in \bbn$ we have $A^j \in C(G;H_0)$ with an extension $A^j \in C^1(H_0;H)$, and that for each $y \in \calm$ the series $\sum_{j=1}^{\infty} D A^j(y) A^j(y)$ converges in $H$. Then the following statements are equivalent:
\begin{enumerate}
\item[(i)] The submanifold $\calm$ is locally invariant for the SPDE (\ref{SPDE}).

\item[(ii)] We have
\begin{align}\label{tang-A-2}
&A^j|_{\calm} \in \Gamma(T \calm), \quad j \in \bbn,
\\ \label{tang-L-2} &L|_{\calm} - \frac{1}{2} \sum_{j=1}^{\infty} D A^j \cdot A^j|_{\calm} \in \Gamma(T \calm).
\end{align}

\item[(iii)] The mappings (\ref{map-1-main}) and (\ref{map-2-main}) are continuous, and for each $y_0 \in \calm$ there exists a local martingale solution $Y$ to the SPDE (\ref{SPDE}) with $Y_0 = y_0$ and lifetime $\tau$ such that $Y^{\tau} \in \calm$ up to an evanescent set and the sample paths of $Y^{\tau}$ are continuous with respect to $\| \cdot \|_G$.
\end{enumerate}
\end{theorem}

\begin{proof}
By the decomposition (\ref{zerl-general}) from Proposition \ref{prop-decomp-Damir-scale} we have
\begin{align*}
[A^j|_{\calm},A^j|_{\calm}]_{\calm} = [D A^j \cdot A^j]_{\Gamma(T \calm)}, \quad j \in \bbn.
\end{align*}
Hence, the result is a consequence of Theorem \ref{thm-SPDE}.
\end{proof}

In the next result we present sufficient conditions for local invariance under the assumption that the volatilities $A^j$, $j \in \bbn$ have a quasi-linear structure. Recall that for any $z \in \calm$ the space $\Gamma_z(T \calm)$ denotes the space of all local vector fields on $\calm$ around $z$; see Definition \ref{def-local-vector-field}.

\begin{theorem}\label{thm-main-3}
We suppose that for each $j \in \bbn$ there exists a continuous mapping $\bar{A}^j : G \times G \to H_0$ such that
\begin{align*}
A^j(y) = \bar{A}^j(y,y), \quad y \in G
\end{align*}
having a continuous extension $\bar{A}^j : H_0 \times G \to H$ such that $\bar{A}_z^j := \bar{A}^j(\cdot,z)$ belongs to $L(H_0,H)$ for each $z \in G$. Furthermore, we assume that for each $y \in \calm$ the series $\sum_{j=1}^{\infty} \bar{A}^j(A^j(y),y)$ converges in $H$, and that
\begin{align}\label{A-vf}
&\bar{A}_z^j|_{\calm} \in \Gamma_z(T \calm), \quad z \in \calm, \quad j \in \bbn,
\\ \label{L-vf-A2} &L|_{\calm} - \frac{1}{2} \sum_{j=1}^{\infty} \bar{A}^j(A^j(\cdot),\cdot)|_{\calm} \in \Gamma(T \calm).
\end{align}
Then mappings (\ref{map-1-main}) and (\ref{map-2-main}) are continuous, and for each $y_0 \in \calm$ there exists a local martingale solution $Y$ to the SPDE (\ref{SPDE}) with $Y_0 = y_0$ and lifetime $\tau$ such that $Y^{\tau} \in \calm$ up to an evanescent set and the sample paths of $Y^{\tau}$ are continuous with respect to $\| \cdot \|_G$. In particular, the submanifold $\calm$ is locally invariant for the SPDE (\ref{SPDE}).
\end{theorem}

\begin{proof}
Note that condition (\ref{A-vf}) implies (\ref{tang-A}). Furthermore, using the decomposition (\ref{zerl-1-linear}) from Proposition \ref{prop-decomp-Damir-scale} we obtain
\begin{align*}
[ A^j|_{\calm}, A^j|_{\calm} ]_{\calm} = [ \bar{A}^j(A^j(\cdot),\cdot) ]_{\Gamma(T \calm)}, \quad j \in \bbn,
\end{align*}
and hence, condition (\ref{L-vf-A2}) is equivalent to (\ref{tang-L}). Consequently, applying Theorem \ref{thm-SPDE} completes the proof.
\end{proof}

\begin{remark}\label{rem-difference}
Suppose that conditions (\ref{A-vf}) and (\ref{L-vf-A2}) from Theorem \ref{thm-main-3} are fulfilled such that $\bar{A}^j$ even has an extension $\bar{A}^j \in C^1(H_0 \times H_0;H)$ for each $j \in \bbn$. Then the submanifold $\calm$ is locally invariant for the SPDE (\ref{SPDE}), and the mapping $A^j \in C(G;H_0)$ has an extension $A^j \in C^1(H_0;H)$ for each $j \in \bbn$. If for each $y \in \calm$ the series $\sum_{j=1}^{\infty} D A^j(y) A^j(y)$ converges in $H$, then by Theorem \ref{thm-main-2} the invariance condition (\ref{tang-L-2}) is satisfied as well. The vector fields in (\ref{tang-L-2}) and (\ref{L-vf-A2}) do not, in general, coincide. Using Proposition \ref{prop-decomp-Damir-scale}, we can determine their difference by using local coordinates. Namely, if $\phi : V \to U \cap \calm$ is a local parametrization, then by the decomposition (\ref{zerl-2-linear}) we have
\begin{align*}
\bigg( L - \frac{1}{2} \sum_{j=1}^{\infty} \bar{A}^j(A^j(\cdot),\cdot) \bigg)\bigg|_{\calm_U} - \bigg( L - \frac{1}{2} \sum_{j=1}^{\infty} D A^j \cdot A^j \bigg)\bigg|_{\calm_U} = \frac{1}{2} \sum_{j=1}^{\infty} \phi_* ( D_2 \bar{a}^j \cdot a^j ),
\end{align*}
where the notation is analogous to that in Proposition \ref{prop-decomp-Damir-scale}. 
\end{remark}

We conclude this section by indicating a result analogous to Theorem \ref{thm-SPDE} for deterministic PDEs of the kind
\begin{align}\label{PDE-Y}
\left\{
\begin{array}{rcl}
dY_t & = & K(Y_t) dt
\\ Y_0 & = & y_0
\end{array}
\right.
\end{align}
with a continuous mapping $K : G \to H$. Here $G$ and $H$ may be Banach spaces, and $\calm$ only needs to be a $(G,H)$-submanifold of class $C^1$. The proof of following result is similar to that of Theorem \ref{thm-SPDE}; indeed the arguments are even simpler.

\begin{theorem}\label{thm-PDE}
The following statements are equivalent:
\begin{enumerate}
\item[(i)] The submanifold $\calm$ is locally invariant for the PDE (\ref{PDE-Y}).

\item[(ii)] We have $K|_{\calm} \in \Gamma(T \calm)$.

\item[(iii)] The mapping $K|_{\calm} : (\calm,\| \cdot \|_H) \to (H,\| \cdot \|_H)$ is continuous, and for each $y_0 \in \calm$ there exists a local solution $Y : [0,T] \to G$ to the PDE (\ref{PDE-Y}) with $Y_0 = y_0$ for some deterministic time $T > 0$ such that $Y \in \calm$ and $Y$ is continuous with respect to $\| \cdot \|_G$.
\end{enumerate}
\end{theorem}

\section{Quasi-semilinear stochastic partial differential equations}\label{sec-quasi-semilinear}

In this section we investigate invariance of submanifolds for quasi-semilinear SPDEs. It is organized as follows: In Section \ref{sub-sec-qsl} we treat the general situation, and in Section \ref{sub-sec-semilinear} we draw consequences for semilinear SPDEs.

\subsection{The general situation}\label{sub-sec-qsl}

Let $(G,H)$ be separable Hilbert spaces with continuous embedding, and let $L : G \to H$ and $A : G \to \ell^2(H)$ be continuous mappings. Throughout this section, we assume that the following assumption is satisfied.

\begin{assumption}[Quasi-semilinearity]\label{ass-quasi-semi-lin}
We suppose that the following conditions are fulfilled:
\begin{enumerate}
\item $G$ is a dense subspace of $H$.

\item There exist a continuous mapping $\bar{L} : G \times H \to H$ and a continuous mapping $\alpha : H \to H$ such that
\begin{align*}
L(y) = \bar{L}(y,y) + \alpha(y), \quad y \in G,
\end{align*}
and for each $z \in H$ the mapping
\begin{align*}
\bar{L}_z := \bar{L}(\cdot,z) : G \to H
\end{align*}
extends to a closed operator $\bar{L}_z : H \supset D(\bar{L}_z) \to H$.

\item There exist a continuous mapping $\bar{A} : G \times H \to \ell^2(H)$ and a continuous mapping $\sigma : H \to \ell^2(H)$ such that
\begin{align*}
A(y) = \bar{A}(y,y) + \sigma(y), \quad y \in G,
\end{align*}
and for each $z \in H$ and each $j \in \bbn$ the mapping
\begin{align*}
\bar{A}_z^j := \bar{A}^j(\cdot,z) : G \to H
\end{align*}
extends to a closed operator $\bar{A}_z^j : H \supset D(\bar{A}_z^j) \to H$.

\item For each $z \in H$ we have
\begin{align}\label{G-intersections}
G = D(\bar{L}_z) \cap \bigg( \bigcap_{j=1}^{\infty} D(\bar{A}_z^j) \bigg).
\end{align}

\item There is a dense subspace $H_0 \subset H$ such that for each $z \in H$ we have
\begin{align*}
H_0 \subset D(\bar{L}_z^*) \cap \bigg( \bigcap_{j=1}^{\infty} D(\bar{A}_z^{j,*}) \bigg),
\end{align*}
and for each $\zeta \in H_0$ we have $\bar{A}_z^{*} \zeta := (\bar{A}_z^{j,*} \zeta)_{j \in \bbn} \in \ell^2(H)$, and the mappings 
\begin{align}\label{map-cont-1}
&H \to H, \quad z \mapsto \bar{L}_z^* \zeta,
\\ \label{map-cont-2} &H \to \ell^2(H), \quad z \mapsto \bar{A}_z^{*} \zeta
\end{align}
are continuous.
\end{enumerate}
\end{assumption}

In view of condition (5), recall that for a densely defined operator $A : H \supset D(A) \rightarrow H$ the adjoint operator $A^* : H \supset D(A^*) \rightarrow H$ is defined on the subspace
\begin{align}\label{def-domain-adjoint}
D(A^*) := \{ z \in H : \xi \mapsto \langle A \xi, z \rangle_H \text{ is continuous on } D(A) \},
\end{align}
and that it is characterized by the property
\begin{align}\label{adjoint-char}
\la Ay,z \ra_H = \la y,A^* z \ra_H \quad \text{for all $y \in D(A)$ and $z \in D(A^*)$.}
\end{align}

\begin{proposition}\cite[Thm. 13.12]{Rudin}\label{prop-A-adj-dicht}
Let $A : H \supset D(A) \rightarrow H$ be densely defined and closed. Then $A^*$ is densely defined and we have $A = A^{**}$.
\end{proposition}

If Assumption \ref{ass-quasi-semi-lin} is fulfilled, then we also call the SPDE (\ref{SPDE}) a quasi-semilinear SPDE. Here are two examples where Assumption \ref{ass-quasi-semi-lin} is satisfied.

\begin{example}
Let $T = (T(t))_{t \in \bbr^d}$ be a multi-parameter $C_0$-group on a separable Hilbert space $H$, and denote by $B = (B_1,\ldots,B_d)$ its generator. We set $G := D(B)$ and assume that the coefficients $L : G \to H$ and $A : G \to \ell^2(H)$ are given by
\begin{align*}
L(y) &= \sum_{i=1}^d \lambda_i(y) B_i y + \alpha(y),
\\ A^j(y) &= \sum_{i=1}^d \kappa_{i}^j(y) B_i y + \sigma^j(y), \quad j \in \bbn
\end{align*}
with continuous mappings $\lambda_i : H \to \bbr$, $i=1,\ldots,d$, $\alpha : H \to H$, $\kappa_{i} : H \to \ell^2(\bbr)$, $i=1,\ldots,d$, and $\sigma : H \to \ell^2(H)$. By Lemma \ref{lemma-domain-multi-dense} and Proposition \ref{prop-adjoint-group} we see that Assumption \ref{ass-quasi-semi-lin} is fulfilled with $H_0 := D(B^*)$.
\end{example}

\begin{example}
Let $H := \cals_p(\bbr^d)$ for some $p \in \bbr$. We set $G := \cals_{p+\frac{1}{2}}(\bbr^d)$ and assume that the coefficients $L : G \to H$ and $A : G \to \ell^2(H)$ are given by
\begin{align*}
L(y) &= \sum_{i=1}^d \lambda_i(y) \partial_i y + \alpha(y),
\\ A^j(y) &= \sum_{i=1}^d \kappa_{i}^j(y) \partial_i y + \sigma^j(y), \quad j \in \bbn
\end{align*}
with continuous mappings $\lambda_i : H \to \bbr$, $i=1,\ldots,d$, $\alpha : H \to H$, $\kappa_{i} : H \to \ell^2(\bbr)$, $i=1,\ldots,d$, and $\sigma : H \to \ell^2(H)$. By Lemma \ref{lemma-diff-closed} we see that Assumption \ref{ass-quasi-semi-lin} is fulfilled with $H_0 := \cals(\bbr^d)$.
\end{example}

\begin{definition}\label{def-weak-q-semi-lin}
Let $y_0 \in H$ be arbitrary. A triplet $(\bbb,W,Y)$ is called a \emph{local analytically weak martingale solution} to the SPDE (\ref{SPDE}) with $Y_0 = y_0$ if the following conditions are fulfilled:
\begin{enumerate}
\item $\bbb = (\Omega,\calf,(\calf_t)_{t \in \bbr_+},\bbp)$ is a stochastic basis; that is, a filtered probability space satisfying the usual conditions.

\item $W$ is a standard $\bbr^{\infty}$-Wiener process on the stochastic basis $\bbb$.

\item $Y$ is an $H$-valued adapted, continuous process such that, for some strictly positive stopping time $\tau > 0$, for each $\zeta \in H_0$ we have $\bbp$-almost surely
\begin{equation}\label{cond-weak-1}
\begin{aligned}
&\int_0^{t \wedge \tau} \bigg( \big| \langle \bar{L}_{Y_s}^* \zeta,Y_s \rangle_H + \langle \zeta,\alpha(Y_s) \rangle_H \big|
\\ &\qquad\quad + \big\| \langle \bar{A}_{Y_s}^* \zeta,Y_s \rangle_H + \langle \zeta,\sigma(Y_s) \rangle_H \big\|_{\ell^2(H)}^2 \bigg) ds < \infty, \quad t \in \bbr_+
\end{aligned}
\end{equation}
and $\bbp$-almost surely
\begin{equation}\label{cond-weak-2}
\begin{aligned}
\langle \zeta,Y_{t \wedge \tau} \rangle_H &= \langle \zeta,y_0 \rangle_H + \int_0^{t \wedge \tau} \big( \langle \bar{L}_{Y_s}^* \zeta,Y_s \rangle_H + \langle \zeta,\alpha(Y_s) \rangle_H \big) ds
\\ &\quad + \int_0^{t \wedge \tau} \big( \langle \bar{A}_{Y_s}^* \zeta,Y_s \rangle_H + \langle \zeta,\sigma(Y_s) \rangle_H \big) dW_s, \quad t \in \bbr_+,
\end{aligned}
\end{equation}
where for each $y \in H$ we agree on the notation
\begin{align*}
\langle \bar{A}_{y}^* \zeta,y \rangle_H &:= \big( \langle \bar{A}_{y}^{j,*} \zeta,y \rangle_H \big)_{j \in \bbn} \in \ell^2(H),
\\ \langle \zeta,\sigma(y) \rangle_H &:= \big( \langle \zeta,\sigma^j(y) \rangle_H \big)_{j \in \bbn} \in \ell^2(H).
\end{align*}
The stopping time $\tau$ is also called the \emph{lifetime} of $Y$.
\end{enumerate}
If we can choose $\tau = \infty$, then $(\bbb,W,Y)$ is also called a \emph{global analytically weak martingale solution} (or simply an \emph{analytically weak martingale solution}) to the SPDE (\ref{SPDE}) with $Y_0 = y_0$.
\end{definition}

\begin{remark}
Note that the integrands in (\ref{cond-weak-1}) and (\ref{cond-weak-2}) are continuous and adapted by virtue of the continuity of the mappings (\ref{map-cont-1}) and (\ref{map-cont-2}).
\end{remark}

\begin{remark}
If there is no ambiguity, we will simply call $Y$ a local analytically weak martingale solution or a global analytically weak martingale solution to the SPDE (\ref{SPDE}) with $Y_0 = y_0$.
\end{remark}

Let $\calm$ be a finite dimensional $C^2$-submanifold of $H$.

\begin{definition}
The submanifold $\calm$ is called \emph{weakly locally invariant} for the SPDE (\ref{SPDE}) if for each $y_0 \in \calm$ there exists a local analytically weak martingale solution $Y$ to the SPDE (\ref{SPDE}) with $Y_0 = y_0$ and lifetime $\tau > 0$ such that $Y^{\tau} \in \calm$ up to an evanescent set.
\end{definition}

\begin{definition}
The submanifold $\calm$ is called \emph{weakly globally invariant} (or simply \emph{weakly invariant}) for the SPDE (\ref{SPDE}) if for each $y_0 \in \calm$ there exists a global analytically weak martingale solution $Y$ to the SPDE (\ref{SPDE}) with $Y_0 = y_0$ such that $Y \in \calm$ up to an evanescent set.
\end{definition}

\begin{remark}\label{rem-weak-invariance}
If $\calm$ is locally invariant (or globally invariant) for the SPDE (\ref{SPDE}), then $\calm$ is also weakly locally invariant (or weakly globally invariant) for the SPDE (\ref{SPDE}).
\end{remark}

\begin{theorem}\label{thm-quasi-semilinear}
Suppose that Assumption \ref{ass-quasi-semi-lin} is fulfilled. Then the following statements are equivalent:
\begin{enumerate}
\item[(i)] The submanifold $\calm$ is weakly locally invariant for the SPDE (\ref{SPDE}).

\item[(ii)] We have $\calm \subset G$, the submanifold $\calm$ is locally invariant for the SPDE (\ref{SPDE}), and the mappings
\begin{align}\label{map-1-qsl}
&L|_{\calm} : (\calm,\| \cdot \|_H) \to (H,\| \cdot \|_H),
\\ \label{map-2-qsl} &A|_{\calm} : (\calm,\| \cdot \|_H) \to (\ell^2(H),\| \cdot \|_{\ell^2(H)})
\end{align}
are continuous.
\end{enumerate}
\end{theorem}

\begin{proof}
(ii) $\Rightarrow$ (i): See Remark \ref{rem-weak-invariance}.

\noindent(i) $\Rightarrow$ (ii): Let $y_0 \in \calm$ be arbitrary. Since $H_0$ is dense in $H$, by Proposition \ref{prop-linear-inverse} there exist a local parametrization $\phi : V \to U \cap \calm$ around $y_0$ and a bounded linear operator $\psi \in L(H,\bbr^m)$ of the form $\psi = \la \zeta,\cdot \ra_H$ with $\zeta_1,\ldots,\zeta_m \in H_0$ such that we have $\phi^{-1} = \psi|_{U \cap \calm}$. Now, let $y \in U \cap \calm$ be arbitrary, and set $x := \psi(y) \in V$. Since the submanifold $\calm$ is weakly locally invariant for the SPDE (\ref{SPDE}), there exist a positive stopping time $\tau > 0$ and a local analytically weak martingale solution $Y$ to (\ref{SPDE}) with $Y_0 = y$ and lifetime $\tau$ such that $Y^{\tau} \in \calm$ up to an evanescent set. Since $U$ is an open subset of $H$ and the sample paths of $Y$ are continuous, we may assume that $Y^{\tau} \in U \cap \calm$ up to an evanescent set. Now, we define the continuous $\bbr^m$-valued process $X := \psi(Y)$. Then we have $X^{\tau} \in V$, and since $\zeta_1,\ldots,\zeta_m \in H_0$, the process $X$ is a local strong solution to the SDE
\begin{align*}
\left\{
\begin{array}{rcl}
dX_t & = & L_{\zeta}(X_t) dt + A_{\zeta}(X_t) d W_t
\\ X_0 & = & x
\end{array}
\right.
\end{align*}
with lifetime $\tau$, where $L_{\zeta} : V \to \bbr^m$ and $A_{\zeta} : V \to \ell^2(\bbr^m)$ are given by
\begin{align}\label{push-generalized-1}
L_{\zeta}(z) &:= \la \bar{L}_{\phi(z)}^* \zeta, \phi(z) \ra_H + \la \zeta,\alpha(\phi(z)) \ra_H,
\\ \label{push-generalized-2} A_{\zeta}^j(z) &:= \la \bar{A}_{\phi(z)}^{j,*} \zeta, \phi(z) \ra_H + \la \zeta,\sigma^j(\phi(z)) \ra_H, \quad j \in \bbn.
\end{align}
Note that the mappings $L_{\zeta}$ and $A_{\zeta}$ are continuous by virtue of the continuity of the mappings (\ref{map-cont-1}) and (\ref{map-cont-2}). Since $\phi \in C^2(V;H)$, by It\^{o}'s formula (see \cite[Thm. 2.3.1]{fillnm}) we obtain that the process $Y$ is a local solution to the SPDE
\begin{align}\label{SPDE-in-proof}
\left\{
\begin{array}{rcl}
dY_t & = & \big( (\phi_* L_{\zeta})(Y_t) + \frac{1}{2} \sum_{j=1}^{\infty} \phi_{**}(A_{\zeta}^j, A_{\zeta}^j)(Y_t) \big) dt + (\phi_* A_{\zeta})(Y_t) dW_t
\\ Y_0 & = & y
\end{array}
\right.
\end{align}
with lifetime $\tau$, where we recall the notation from Definition \ref{def-push}. Let $\xi \in H_0$ be arbitrary. Then we have
\begin{align*}
\la \xi,Y_{t \wedge \tau} \ra_H &= \la \xi,y \ra_H + \int_0^{t \wedge \tau} \bigg\la \xi, (\phi_* L_{\zeta})(Y_s) + \frac{1}{2} \sum_{j=1}^{\infty} \phi_{**}(A_{\zeta}^j, A_{\zeta}^j)(Y_s) \big) \bigg\ra_H ds
\\ &\quad + \int_0^{t \wedge \tau} \la \xi, (\phi_* A_{\zeta})(Y_s) \ra_H dW_s, \quad t \in \bbr_+
\end{align*}
On the other hand, the process $Y$ is a local analytically weak martingale solution to the original SPDE (\ref{SPDE}) with $Y_0 = y$ and lifetime $\tau$. Therefore, we have
\begin{align*}
\langle \xi,Y_{t \wedge \tau} \rangle_H &= \langle \xi,y \rangle_H + \int_0^{t \wedge \tau} \big( \langle \bar{L}_{Y_s}^* \xi,Y_s \rangle_H + \langle \xi,\alpha(Y_s) \rangle_H \big) ds
\\ &\quad + \int_0^{t \wedge \tau} \big( \langle \bar{A}_{Y_s}^* \xi,Y_s \rangle_H + \langle \xi,\sigma(Y_s) \rangle_H \big) dW_s, \quad t \in \bbr_+.
\end{align*}
Thus, taking into account Lemma \ref{lemma-push} and the continuity of the mappings (\ref{map-cont-1}) and (\ref{map-cont-2}), we have
\begin{align}\label{L-1-semi-q}
\la \bar{L}_y^* \xi,y \ra_H &= \bigg\la \xi, (\phi_* L_{\zeta})(y) + \frac{1}{2} \sum_{j=1}^{\infty} \phi_{**}(A_{\zeta}^j, A_{\zeta}^j)(y) - \alpha(y) \bigg\ra_H,
\\ \label{A-1-semi-q} \la (\bar{A}_y^j)^* \xi,y \ra_H &= \la \xi, (\phi_* A_{\zeta}^j)(y) - \sigma^j(y) \ra_H, \quad j \in \bbn.
\end{align}
Taking into account Proposition \ref{prop-A-adj-dicht} and (\ref{def-domain-adjoint}), we have
\begin{align*}
D(\bar{L}_y) = D(\bar{L}_y^{**}) = \{ z \in H : \xi \mapsto \langle \bar{L}_y^* \xi, z \rangle_H \text{ is continuous on } D(\bar{L}_y^*) \}
\end{align*}
as well as
\begin{align*}
D(\bar{A}_y^j) = D((\bar{A}_y^j)^{**}) = \{ z \in H : \xi \mapsto \langle (\bar{A}_y^j)^* \xi, z \rangle_H \text{ is continuous on } D((\bar{A}_y^j)^*) \}
\end{align*}
for all $j \in \bbn$. This proves $y \in D(\bar{L}_y)$ and $y \in D(\bar{A}_y^j)$ for all $j \in \bbn$. Taking into account (\ref{G-intersections}), we deduce that $y \in G$. Consequently, we have $\calm \subset G$. By (\ref{push-generalized-1}) and (\ref{push-generalized-2}) we obtain
\begin{align*}
L_{\zeta}(x) &= \la \zeta, \bar{L}_{\phi(x)} \phi(x) \ra_H + \la \zeta,\alpha(\phi(x)) \ra_H = \la \zeta, L(\phi(x)) \ra_H,
\\ A_{\zeta}^j(x) &= \la \zeta, \bar{A}_{\phi(x)}^j \phi(x) \ra_H + \la \zeta,\sigma^j(\phi(x)) \ra_H = \la \zeta, A^j(\phi(x)) \ra_H, \quad j \in \bbn
\end{align*}
for each $x \in V$. Furthermore, from (\ref{L-1-semi-q}) and (\ref{A-1-semi-q}) we obtain
\begin{align*}
\la \xi,\bar{L}_y y \ra_H &= \bigg\la \xi, (\phi_* L_{\zeta})(y) + \frac{1}{2} \sum_{j=1}^{\infty} \phi_{**}(A_{\zeta}^j, A_{\zeta}^j)(y) - \alpha(y) \bigg\ra_H,
\\ \la \xi,\bar{A}_y^j y \ra_H &= \la \xi, (\phi_* A_{\zeta}^j)(y) - \sigma^j(y) \ra_H, \quad j \in \bbn
\end{align*}
for all $\xi \in H_0$ and all $y \in U \cap \calm$. Since $H_0$ is dense in $H$, we obtain
\begin{align*}
L(y) &= \bar{L}(y,y) + \alpha(y) = (\phi_* L_{\zeta})(y) + \frac{1}{2} \sum_{j=1}^{\infty} \phi_{**}(A_{\zeta}^j, A_{\zeta}^j)(y),
\\ A^j(y) &= \bar{A}^j(y,y) + \sigma^j(y) = (\phi_* A_{\zeta}^j)(y), \quad j \in \bbn
\end{align*}
for all $y \in U \cap \calm$. Since $y_0 \in \calm$ at the beginning of the proof was chosen arbitrary, by Lemma \ref{lemma-push} we deduce that the mappings (\ref{map-1-qsl}) and (\ref{map-2-qsl}) are continuous. Furthermore, by taking into account (\ref{SPDE-in-proof}), we see that $Y$ is local strong solution to the SPDE (\ref{SPDE}) with $Y_0 = y_0$, proving that $\calm$ is locally invariant for the SPDE (\ref{SPDE}).
\end{proof}

\subsection{Semilinear stochastic partial differential equations}\label{sub-sec-semilinear}

In this section we present consequences of our previous findings for semilinear SPDEs of the form
\begin{align}\label{SPDE-semigroup}
\left\{
\begin{array}{rcl}
dY_t & = & ( B Y_t + \alpha(Y_t) ) dt + \sigma(Y_t) dW_t
\\ Y_0 & = & y_0.
\end{array}
\right.
\end{align}
Such equations have been studied, for example, in \cite{Da_Prato, Atma-book, Liu-Roeckner, Prevot-Roeckner}. Here the state space $H$ is a separable Hilbert space, and $B : H \supset D(B) \to H$ is a densely defined, closed operator. Moreover $\alpha : H \to H$ and $\sigma : H \to \ell^2(H)$ are continuous mappings. We endow $G := D(B)$ with the graph norm
\begin{align}\label{graph-norm}
\| y \|_G := \sqrt{\| y \|_H^2 + \| B y \|_H^2 }, \quad y \in G.
\end{align}
By Proposition \ref{prop-domain-scale}, the pair $(G,H)$ consists of separable Hilbert spaces with continuous embedding.

\begin{remark}
Note that the semilinear SPDE (\ref{SPDE-semigroup}) is of the type (\ref{SPDE}) with $L = B + \alpha$ and $A = \sigma$. Furthermore, note that Assumption \ref{ass-quasi-semi-lin} is fulfilled with
\begin{align*}
\bar{L}(y,z) &= B(y) \quad \text{for all $y \in G$ and $z \in H$,}
\\ \bar{A} &= 0,
\end{align*}
and $H_0 = D(B^*)$. The concept of a local martingale solution (or a global martingale solution) from Definition \ref{def-martingale-solution} is just the concept of a local strong solution (or a global strong solution) for the semilinear SPDE (\ref{SPDE-semigroup}) in the sense of martingale solutions. Accordingly, the concept of a local analytically weak martingale solution (or a global analytically weak martingale solution) from Definition \ref{def-weak-q-semi-lin} is just the concept of a local weak solution (or a global weak solution) for the semilinear SPDE (\ref{SPDE-semigroup}) in the sense of martingale solutions.
\end{remark}

\begin{remark}
If $B$ generates a $C_0$-semigroup on $H$, then we can also consider mild solutions. However, this is not required for our upcoming results.
\end{remark}

Let $\calm$ be a finite dimensional $C^2$-submanifold of $H$. Invariant manifolds of weak solutions to semilinear SPDEs have been studied, for example, in \cite{Filipovic-inv, Nakayama}; see also \cite{FTT-manifolds} for the case of jump-diffusions and submanifolds with boundary.

\begin{lemma}\label{lemma-T-topology}
The following statements are equivalent:
\begin{enumerate}
\item[(i)] $\calm$ is a finite dimensional $(G,H)$-submanifold of class $C^2$

\item[(ii)] $\calm \subset G$ and the restriction $B|_{\calm} : (\calm,\| \cdot \|_H) \to (H,\| \cdot \|_H)$ is continuous.
\end{enumerate}
\end{lemma}

\begin{proof}
This is an immediate consequence of Proposition \ref{prop-group-top-manifold}.
\end{proof}

\begin{proposition}\label{prop-semi-lin-pre}
For a finite dimensional $C^2$-submanifold $\calm$ of $H$ the following statements are equivalent:
\begin{enumerate}
\item[(i)] The submanifold $\calm$ is weakly locally invariant for the semilinear SPDE (\ref{SPDE-semigroup}).

\item[(ii)] $\calm$ is a $(G,H)$-submanifold of class $C^2$, which is locally invariant for the semilinear SPDE (\ref{SPDE-semigroup}).
\end{enumerate}
\end{proposition}

\begin{proof}
(i) $\Rightarrow$ (ii): By Theorem \ref{thm-quasi-semilinear} we have $\calm \subset G$, the submanifold $\calm$ is locally invariant for the semilinear SPDE (\ref{SPDE-semigroup}), and the restriction $B|_{\calm} : (\calm,\| \cdot \|_H) \to (H,\| \cdot \|_H)$ is continuous. Moreover, by Lemma \ref{lemma-T-topology} the submanifold $\calm$ is a $(G,H)$-submanifold of class $C^2$.

\noindent(ii) $\Rightarrow$ (i): This implication is obvious.
\end{proof}

\begin{theorem}\label{thm-semi-lin}
Let $\calm$ be a finite dimensional $C^2$-submanifold of $H$. Then the following statements are equivalent:
\begin{enumerate}
\item[(i)] The submanifold $\calm$ is weakly locally invariant for the semilinear SPDE (\ref{SPDE-semigroup}).

\item[(ii)] $\calm$ is a $(G,H)$-submanifold of class $C^2$, and we have
\begin{align}\label{tang-semi-A}
&\sigma^j|_{\calm} \in \Gamma(T \calm), \quad j \in \bbn,
\\ \label{tang-semi-L} &[ (B + \alpha)|_{\calm} ]_{\Gamma(T \calm)} - \frac{1}{2} \sum_{j=1}^{\infty} [\sigma^j|_{\calm},\sigma^j|_{\calm}]_{\calm} = [0]_{\Gamma(T \calm)}.
\end{align}
\item[(iii)] $\calm$ is a $(G,H)$-submanifold of class $C^2$, the mapping $B|_{\calm} : (\calm,\| \cdot \|_H) \to (H,\| \cdot \|_H)$ is continuous, and for each $y_0 \in \calm$ there exists a local martingale solution $Y$ to the SPDE (\ref{SPDE}) with $Y_0 = y_0$ and lifetime $\tau$ such that $Y^{\tau} \in \calm$ up to an evanescent set and the sample paths of $Y^{\tau}$ are continuous with respect to the graph norm $\| \cdot \|_G$.
\end{enumerate}
\end{theorem}

\begin{proof}
This is a consequence of Proposition \ref{prop-semi-lin-pre} and Theorem \ref{thm-SPDE}.
\end{proof}

\begin{remark}\label{rem-semi-lin}
If we even have $\sigma^j \in C^1(H)$ for all $j \in \bbn$, and for each $y \in \calm$ the series $\sum_{j=1}^{\infty} D \sigma^j(y) \sigma^j(y)$ converges in $H$, then conditions (i)--(iii) are equivalent to the following:
\begin{enumerate}
\item[(iv)] $\calm$ is a $(G,H)$-submanifold of class $C^2$, and we have (\ref{tang-semi-A}) as well as
\begin{align*}
B|_{\calm} + \alpha|_{\calm} - \frac{1}{2} \sum_{j=1}^{\infty} D \sigma^j \cdot \sigma^j|_{\calm} \in \Gamma(T \calm).
\end{align*}
\end{enumerate}
This is a consequence of the decomposition (\ref{zerl-general}) from Proposition \ref{prop-decomp-Damir-scale}.
\end{remark}

\begin{remark}\label{rem-regularity}
Let $k \in \bbn$ and $l \in \bbn_0$ be arbitrary, let $\calm$ be a $C^k$-submanifold of $H$ and assume that $\sigma^j \in C^l(H)$ for all $j \in \bbn$. Then $k$ is the degree of smoothness of the submanifold, and $l$ is the degree of smoothness of the volatilities. In the literature, the following situations have been considered:
\begin{enumerate}
\item In \cite{Filipovic-inv} it is assumed that $k=2$ and $l=1$.

\item In \cite{Nakayama} (which uses the support theorem from \cite{Nakayama-Support}) it is assumed that $k=1$ and $l=1$.

\item Here, in Theorem \ref{thm-semi-lin} we assume that $k=2$ and $l=0$.
\end{enumerate}
Summing up these degrees of smoothness, we see that in our result we have also achieved $k+l = 2$.
\end{remark}

\section{Invariant manifolds generated by orbit maps}\label{sec-quasi-linear}

In this section we investigate invariance of submanifolds generated by orbit maps. It is organized as follows: In Section \ref{sec-ql-group} we investigate the structure of the coefficients of the SPDE in case of invariance of such a submanifold, and in Section \ref{sec-structure} we treat the structure of invariant submanifolds for SPDEs with such coefficients. In Section \ref{sec-examples-HS} we apply our findings to SPDEs in Hermite Sobolev spaces.

\subsection{Coefficients given by generators of group actions}\label{sec-ql-group}

Let $(G,H_0,H)$ be separable Hilbert spaces with continuous embeddings. We consider the SPDE (\ref{SPDE}) with continuous mappings $L : G \to H$ and $A : G \to \ell^2(H)$. Let $d \in \bbn$ be a positive integer, and let $T = (T(t))_{t \in \bbr^d}$ be a multi-parameter $C_0$-group on $H$ such that $T|_G$ is a multi-parameter $C_0$-group on $G$, and $T|_{H_0}$ is a multi-parameter $C_0$-group on $H_0$. We denote by $B = (B_1,\ldots,B_d)$ the generator of $T$; see Appendix \ref{app-groups} for further details. We assume that $H_0 \subset D(B)$ and $G \subset D(B^2)$. Furthermore, we assume that $B_i|_{H_0} \in L(H_0,H)$ and $B_i|_G \in L(G,H_0)$ for each $i=1,\ldots,d$. Let $y_0 \in G$ be arbitrary, and denote by $\psi \in C^2(\bbr^d;H)$ the orbit map given by $\psi(t) := T(t) y_0$ for each $t \in \bbr^d$. Let $\caln$ be an $m$-dimensional $C^2$-submanifold of $\bbr^d$ for some $m \leq d$, and let $\calm$ be an $m$-dimensional $(G,H_0,H)$-submanifold of class $C^2$, which is induced by $(\psi,\caln)$; see Definition \ref{def-manifold-embedded}. Recall that this requires that $\psi|_{\caln} : \caln \to \psi(\caln)$ is a homeomorphism, and that $\psi$ is a $C^2$-immersion on $\caln$. 

\begin{remark}\label{rem-matrix-mult}
For a multi-dimensional sequence $\sigma = (\sigma_1,\ldots,\sigma_d) \in \ell^2(\bbr^d) \cong \ell^2(\bbr)^{\times d}$ we denote by $\sigma \sigma^{\top} \in \bbr^{d \times d}$ the matrix with elements $(\sigma \sigma^{\top})_{ik} := \la \sigma_i, \sigma_k \ra_{\ell^2(\bbr)}$ for all $i,k=1,\ldots,d$. If there is an index $r \in \bbn$ such that $\sigma^j = 0$ for all $j > r$, then we may regard the sequence $\sigma$ as a matrix $\sigma \in \bbr^{d \times r}$, and $\sigma \sigma^{\top}$ is just the usual matrix multiplication with the transpose matrix.
\end{remark}

\begin{theorem}\label{thm-SPDE-group}
The following statements are equivalent:
\begin{enumerate}
\item[(i)] The submanifold $\calm$ is locally invariant for the SPDE (\ref{SPDE}).

\item[(ii)] The submanifold $\caln$ is locally invariant for the $\bbr^d$-valued SDE
\begin{align}\label{SDE-bar}
\left\{
\begin{array}{rcl}
dX_t & = & \bar{b}(X_t) dt + \bar{\sigma}(X_t) dW_t
\\ X_0 & = & x_0,
\end{array}
\right.
\end{align}
where the continuous mappings $\bar{\sigma} : \caln \to \ell^2(\bbr^d)$ and $\bar{b} : \caln \to \bbr^d$ are the unique solutions of the equations
\begin{align}\label{L-Rajeev}
L|_{\calm} &= \frac{1}{2} \sum_{i,j=1}^d (\bar{\sigma} \bar{\sigma}^{\top})_{ij} \circ \psi^{-1}|_{\calm} \, B_{ij}|_{\calm} + \sum_{i=1}^d \bar{b}_i \circ \psi^{-1}|_{\calm} \, B_i|_{\calm},
\\ \label{A-Rajeev} A^j|_{\calm} &= \sum_{i=1}^d \bar{\sigma}_{i}^j \circ \psi^{-1}|_{\calm} \, B_i|_{\calm}, \quad j \in \bbn.
\end{align}
\end{enumerate}
\end{theorem}

\begin{proof}
Let $y \in \calm$ be arbitrary, and set $x := \psi^{-1}(y) \in \caln$. By Proposition \ref{prop-degrees-smoothness} for $j \in \bbn$ we have
\begin{align*}
(\psi_* \bar{\sigma}^j)(y) = D \psi(x) \bar{\sigma}^j(x) = \sum_{i=1}^d B_i \psi(x) \bar{\sigma}_{i}^j(x) = \sum_{i=1}^d \bar{\sigma}_{i}^j(\psi^{-1}(y)) B_i y
\end{align*}
as well as
\begin{align*}
&(\psi_* \bar{b})(y) + \frac{1}{2} \sum_{j=1}^{\infty} \psi_{**}(\bar{\sigma}^j,\bar{\sigma}^j)(y) = D \psi(x) \bar{b}(x) + \frac{1}{2} \sum_{j=1}^{\infty} D^2 \psi(x) ( \bar{\sigma}^j(x), \bar{\sigma}^j(x) )
\\ &= \sum_{i=1}^d B_i \psi(x) \bar{b}_i(x) + \frac{1}{2} \sum_{j=1}^{\infty} \sum_{i,k=1}^d B_{ik} \psi(x) \bar{\sigma}_{i}^j(x) \bar{\sigma}_{k}^j(x)
\\ &= \sum_{i=1}^d \bar{b}_i(\psi^{-1}(y)) B_i y + \frac{1}{2} \sum_{i,k=1}^d \bar{\sigma}(\psi^{-1}(y)) \bar{\sigma}(\psi^{-1}(y))^{\top} B_{ik} y.
\end{align*}
Therefore, applying Theorem \ref{thm-inv-embedded} concludes the proof.
\end{proof}

\begin{proposition}\label{prop-SPDE-group}
Suppose that the following conditions are fulfilled:
\begin{enumerate}
\item The submanifold $\calm$ is locally invariant for the SPDE (\ref{SPDE}).

\item The submanifold $\caln$ has one chart with a global parametrization $\varphi : V \to \caln$.

\item The open set $V$ is globally invariant for the $\bbr^m$-valued SDE
\begin{align}\label{global-1}
\left\{
\begin{array}{rcl}
d \Xi_t & = & \ell(\Xi_t) dt + a(\Xi_t) dW_t
\\ \Xi_0 & = & \xi_0,
\end{array}
\right.
\end{align}
whose coefficients $a : V \to \ell^2(\bbr^m)$ and $\ell : V \to \bbr^m$ are the unique solutions of the equations
\begin{align}\label{global-2}
\bar{\sigma}^j &= \varphi_* a^j, \quad j \in \bbn,
\\ \label{global-3} \bar{b} &= \varphi_* \ell + \frac{1}{2} \sum_{j=1}^{\infty} \varphi_{**}(a^j,a^j),
\end{align}
where the continuous mappings $\bar{\sigma} : \caln \to \ell^2(\bbr^d)$ and $\bar{b} : \caln \to \bbr^d$ are the unique solutions of the equations (\ref{L-Rajeev}) and (\ref{A-Rajeev})
\end{enumerate}
Then the submanifold $\calm$ is globally invariant for the SPDE (\ref{SPDE}), and the submanifold $\caln$ is globally invariant for the SDE (\ref{SDE-bar}).
\end{proposition}

\begin{proof}
This is a consequence of Proposition \ref{prop-inv-embedded}.
\end{proof}

\begin{remark}
Examples of submanifolds $\calm$ as in Theorem \ref{thm-SPDE-group} are obtained from Examples \ref{ex-manifolds-domain} with $k=2$ and choosing $G = D(B^2)$ as well as $H_0 = D(B)$. Moreover, regarding Proposition \ref{prop-SPDE-group}, recall that the submanifold $\calm$ has one chart if $\caln$ has one chart; see Lemma \ref{lemma-M-N-one-chart}.
\end{remark}

\subsection{The structure of invariant submanifolds}\label{sec-structure}

In the previous we have considered invariant submanifolds which are induced by $(\psi,\caln)$, and shown that the coefficients of the SPDE (\ref{SPDE}) must be of the form (\ref{L-Rajeev}) and (\ref{A-Rajeev}). In this section, we will show that for such coefficients an invariant submanifold must, subject to appropriate regularity conditions, necessarily be an induced submanifold.

Let $T = (T(t))_{t \in \bbr^d}$ be a multi-parameter $C_0$-group on $H$ as in Section \ref{sec-ql-group}. Furthermore, let $\calm$ be an $m$-dimensional $(G,H_0,H)$-submanifold of class $C^2$, which is locally invariant for the SPDE (\ref{SPDE}). Suppose that for each $j=1,\ldots,m$ we have $A^j \in C(G;H_0)$ with an extension $A^j \in C^1(H_0;H)$. Let $y_0 \in \calm$ be arbitrary. By Proposition \ref{prop-degrees} there exists a local parametrization $\phi : V \to U \cap \calm$ around $y_0$ such that
\begin{align*}
\phi \in C(V;G) \cap C^1(V;H_0) \cap C^2(V;H).
\end{align*}
We assume there exists a mapping $\Lambda : V \to \bbr^{m \times d}$ of class $C^1$ such that
\begin{align}\label{B-Lambda-A}
A(y) = \Lambda(x) B(y), \quad y \in U \cap \calm,
\end{align}
where $x := \phi^{-1}(y) \in V$, and where we use the notations $A = (A^1,\ldots,A^m)$ and $B = (B_1,\ldots,B_d)$. Then the volatilities $A^1,\ldots,A^m$ are locally of the form (\ref{A-Rajeev}). We assume that
\begin{align*}
\dim \lin \{ A^1 y, \ldots, A^m y \} = m \quad \text{for each $y \in U \cap \calm$.}
\end{align*}
By Theorem \ref{thm-SPDE} we have $A^1,\ldots,A^m \in \Gamma(T \calm)$, and hence
\begin{align}\label{tang-Bi}
T_y \calm = \lin \{ A^1 y, \ldots, A^m y \} \quad \text{for each $y \in U \cap \calm$.}
\end{align}

\begin{lemma}
There exists a mapping $\Gamma : V \to \bbr^{m \times m}$ of class $C^1$ such that
\begin{align}\label{nabla-phi}
\nabla \phi(x) = \Gamma(x) A \phi(x), \quad x \in V.
\end{align}
\end{lemma}

\begin{proof}
Let $x \in V$ be arbitrary, and set $y := \phi(x) \in U \cap \calm$. Noting (\ref{tang-Bi}), the two sets
\begin{align*}
\{ \partial_1 \phi(x), \ldots, \partial_m \phi(x) \} \quad \text{and} \quad \{ A^1 \phi(x), \ldots, A^m \phi(x) \}
\end{align*}
are bases of $T_y \calm$. Hence, there is a unique matrix $\Gamma(x) \in \bbr^{m \times m}$ such that $\nabla \phi(x) = \Gamma(x) A \phi(x)$. This gives us a mapping $\Gamma : V \to \bbr^{m \times m}$ satisfying (\ref{nabla-phi}). The mapping $\nabla \phi : V \to H$ is of class $C^1$ because $\phi \in C^2(V;H)$. Furthermore, the mapping $A \phi$ is of class $C^1$ because $\phi \in C^1(V;H_0)$ and $A^j \in C^1(H_0;H)$ for each $j=1,\ldots,m$. Consequently, the mapping $\Gamma$ is of class $C^1$, which concludes the proof.
\end{proof}

Now, we consider the product $\Phi := \Gamma \cdot \Lambda : V \to \bbr^{m \times d}$, which is again of class $C^1$. Furthermore, we set $x_0 := \phi^{-1}(y_0) \in V$. Recall that $\psi \in C^2(\bbr^d;H)$ denotes the orbit map given by $\psi(t) := T(t) y_0$ for each $t \in \bbr^d$.

\begin{theorem}\label{thm-structure-manifold}
Suppose that $\Phi$ has a primitive and satisfies $\rk \, \Phi(x_0) = m$. Then there exist an $m$-dimensional $C^2$-submanifold $\caln$ of $\bbr^d$ and an open neighborhood $U_0 \subset U$ of $y_0$ such that the submanifold $U_0 \cap \calm$ is induced by $(\psi,\caln)$.
\end{theorem}

\begin{proof}
We may assume that the open set $V$ is a connected neighborhood of $x_0$. By (\ref{B-Lambda-A}) and (\ref{nabla-phi}) the mapping $\phi \in C^2(V;H)$ is a $D(B)$-valued solution to the PDE
\begin{align*}
\left\{
\begin{array}{rcl}
\nabla \phi(x) & = & \Phi(x) B \phi(x), \quad x \in V,
\\ \phi(x_0) & = & y_0.
\end{array}
\right.
\end{align*}
By assumption the mapping $\Phi$ has a primitive $\varphi : V \to \bbr^d$. We may assume that $\varphi(x_0) = 0$. Thus, by Proposition \ref{prop-solution-of-PDE} we obtain $\phi = \psi \circ \varphi$. Since $\nabla \varphi = \Phi$ and $\rk \, \Phi(x_0) = m$, the mapping $\varphi$ is a $C^2$-immersion at $x_0$. Hence, by Lemma \ref{lemma-induced-0} there exists an open neighborhood $V_0 \subset V$ of zero such that $\varphi|_{V_0} : V_0 \to \varphi(V_0)$ is a homeomorphism and $\varphi|_{V_0}$ is a $C^2$-immersion. Moreover, by Lemma \ref{lemma-induced} the set $\caln := \varphi(V_0)$ is an $m$-dimensional $C^2$-submanifold of $\bbr^d$. Since $\phi : V \to U \cap \calm$ is a homeomorphism, there exists an open neighborhood $U_0 \subset U$ of $y_0$ such that $\phi(V_0) = U_0 \cap \calm$, and hence $U_0 \cap \calm = \psi(\caln)$. Note that $\psi|_{\caln} : \caln \to \psi(\caln)$ is a homeomorphism, because $\phi|_{V_0} : V_0 \to \psi(\caln)$ and $\varphi|_{V_0} : V_0 \to \caln$ are homeomorphisms. Furthermore, by the chain rule, for each $x \in \caln$ we have
\begin{align*}
D \psi(x)|_{T_x \caln} = D \phi(\xi) D \varphi(\xi)^{-1} \in L(T_x \caln,H),
\end{align*}
where $\xi := \varphi^{-1}(x) \in V_0$, showing that $\psi$ is a $C^2$-immersion on $\caln$.
\end{proof}

\begin{remark}
We may assume that the open set $V$ is a simply connected neighborhood of $x_0$. Then $\Phi$ has a primitive if and only if
\begin{align*}
\frac{\partial \Phi_{ik}}{\partial x_j} = \frac{\partial \Phi_{jk}}{\partial x_i} \quad \text{for all $i,j = 1,\ldots,m$ and $k=1,\ldots,d$.}
\end{align*}
\end{remark}

\subsection{Invariant submanifolds in Hermite Sobolev spaces}\label{sec-examples-HS}

In this section we will apply our findings from Section \ref{sec-ql-group} in order to construct examples of invariant submanifolds in Hermite Sobolev spaces; see Appendix \ref{app-HS-spaces} for further details about Hermite Sobolev spaces. Let $p \in \bbr$ be arbitrary and set $G := \cals_{p+1}(\bbr^d)$, $H_0 := \cals_{p+\frac{1}{2}}(\bbr^d)$ and $H := \cals_p(\bbr^d)$. Furthermore, let $\tau = (\tau_x)_{x \in \bbr^d}$ be the translation group. Let $b \in \cals_{-(p+1)}(\bbr^d;\bbr^d)$ and $\sigma \in \ell^2(\cals_{-(p+1)}(\bbr^d;\bbr^d))$ be given, where for any $q \in \bbr$ we agree on the notation
\begin{align*}
\cals_q(\bbr^d;\bbr^d) := \cals_q(\bbr^d)^{\times d},
\end{align*}
which, endowed with the norm
\begin{align*}
\| f \|_{q,d} := \bigg( \sum_{i=1}^d \| f_i \|_q^2 \bigg)^{1/2}, \quad f \in \cals_q(\bbr^d;\bbr^d),
\end{align*}
is also a separable Hilbert space. Furthermore, the norm on $\ell^2(\cals_{q}(\bbr^d;\bbr^d))$ will be denoted by $\| \cdot \|_{q,\ell^2}$. We define the coefficients $L : G \to H$ and $A^j : G \to H_0$ for $j \in \bbn$ of the SPDE (\ref{SPDE}) as
\begin{align}\label{def-L}
L(y) &:= \frac{1}{2} \sum_{i,j=1}^d ( \langle \sigma,y \rangle \langle \sigma,y \rangle^{\top} )_{ij} \partial_{ij}^2 y - \sum_{i=1}^d \langle b_i,y \rangle \partial_i y,
\\ \label{def-A} A^j(y) &:= - \sum_{i=1}^d \langle \sigma_{i}^j,y \rangle \partial_i y, \quad j \in \bbn,
\end{align}
where $\la \cdot,\cdot \ra$ denotes the dual pair on $\cals_{-(p+1)}(\bbr^d) \times \cals_{p+1}(\bbr^d)$; see Lemma \ref{lemma-duality-pairing} and also Remark \ref{rem-dual-pairing}. Furthermore $\langle \sigma,y \rangle \in \ell^2(\bbr^{d})$ is given by $\langle \sigma,y \rangle := (\langle \sigma^j,y \rangle)_{j \in \bbn}$, where for $c \in \cals_{-(p+1)}(\bbr^d;\bbr^d)$ we define $\la c,y \ra \in \bbr^d$ as $\la c,y \ra := ( \la c_i,y \ra )_{i=1,\ldots,d}$. Recalling the notation introduced in Remark \ref{rem-matrix-mult}, it is obvious that $L : G \to H$ is continuous. In order to analyze the mapping $A := (A^j)_{j \in \bbn}$, for each $j \in \bbn$ we define $\bar{A}^j : H_0 \times G \to H$ as
\begin{align}\label{A-bar-y-z}
\bar{A}^j(y,z) &:= - \sum_{i=1}^d \langle \sigma_{i}^j,z \rangle \partial_i y, \quad (y,z) \in H_0 \times G.
\end{align}
Note that
\begin{align}\label{A-bar-A}
A^j(y) = \bar{A}^j(y,y) \quad \text{for all $y \in G$ and $j \in \bbn$.} 
\end{align}
Moreover, by Lemma \ref{lemma-diff-continuous} for each $j \in \bbn$ the definition (\ref{A-bar-y-z}) provides bounded bilinear operators $\bar{A}^j \in L(G,G;H_0)$ and $\bar{A}^j \in L(H_0,G;H)$.

\begin{lemma}
The following statements are true:
\begin{enumerate}
\item The definition (\ref{A-bar-y-z}) provides a bounded bilinear operator 
\begin{align*}
\bar{A} \in L(G,G;\ell^2(H_0)).
\end{align*}

\item The definition (\ref{def-A}) provides a continuous mapping $A : G \to \ell^2(H_0)$.

\item The definition (\ref{A-bar-y-z}) provides a bounded bilinear operator
\begin{align*}
\bar{A} \in L(H_0,G;\ell^2(H)).
\end{align*}

\item The mapping $G \to H$, $y \mapsto \sum_{j=1}^{\infty} \bar{A}^j(A^j(y),y)$ is well-defined and continuous.
\end{enumerate}
\end{lemma}

\begin{proof}
Using Lemma \ref{lemma-diff-continuous}, for all $y,z \in G$ we have
\begin{align*}
\| \bar{A}(y,z) \|_{\ell^2(H_0)}^2 &\leq d \sum_{j=1}^{\infty} \sum_{i=1}^d |\langle \sigma_{i}^j,z \rangle|^2 \| \partial_i y \|_{H_0}^2 \leq d \sum_{j=1}^{\infty} \sum_{i=1}^d \| \sigma_{i}^j \|_{-(p+1)}^2 \| z \|_G^2 \| y \|_{G}^2
\\ &\leq Cd \| \sigma \|_{-(p+1),\ell^2}^2 \| y \|_G^2 \| z \|_G^2
\end{align*}
with a universal constant $C > 0$, proving the first statement. Since we have (\ref{A-bar-A}), the second statement follows as well. Similarly, using Lemma \ref{lemma-diff-continuous}, for all $y \in H_0$ and $z \in G$ we have
\begin{align*}
\| \bar{A}(y,z) \|_{\ell^2(H)}^2 &\leq d \sum_{i=1}^d |\langle \sigma_{i}^j,z \rangle|^2 \| \partial_i y \|_{H}^2 \leq d \sum_{i=1}^d \| \sigma_{i}^j \|_{-(p+1)}^2 \| z \|_G^2 \| y \|_{H_0}^2
\\ &\leq Cd \| \sigma \|_{-(p+1),\ell^2}^2 \| y \|_{H_0}^2 \| z \|_G^2
\end{align*}
with a universal constant $C > 0$, proving the third statement. Now, for $j \in \bbn$ we define $\bar{B}^j : \ell^2(H_0) \times G \to H$ as
\begin{align*}
\bar{B}^j(y,z) := \bar{A}^j(y^j,z).
\end{align*}
Then $\bar{B} := (\bar{B}^j)_{j \in \bbn}$ provides a bounded bilinear operator
\begin{align*}
\bar{B} \in L(\ell^2(H_0),G;\ell^1(H)).
\end{align*}
Indeed, using Lemma \ref{lemma-diff-continuous} and the Cauchy-Schwarz inequality, for all $y \in \ell^2(H_0)$ and $z \in G$ we obtain
\begin{align*}
\| \bar{B}(y,z) \|_{\ell^1(H)} &= \sum_{j=1}^{\infty} \| \bar{A}^j(y^j,z) \|_H \leq C \sum_{j=1}^{\infty} \| \sigma^j \|_{-(p+1),d} \| y^j \|_{H_0} \| z \|_G
\\ &\leq C \| \sigma \|_{-(p+1),\ell^2} \| y \|_{\ell^2(H_0)} \| z \|_G
\end{align*}
with a universal constant $C > 0$. By Lemma \ref{lemma-evaluation} the mapping $\ell^1(H) \to H$, $z \mapsto \sum_{j=1}^{\infty} z^j$ belongs to $L(\ell^1(H),H)$. Therefore, and since $A : G \to \ell^2(H_0)$ is continuous, the proof is completed.
\end{proof}

\begin{remark}\label{rem-non-smoothness}
Note that the mapping $A : G \to \ell^2(H_0)$ generally does not satisfy the smoothness assumption imposed in Theorem \ref{thm-main-2}, where it is required that for every $j \in \bbn$ the mapping $A^j \in C(G;H_0)$ admits an extension $A^j \in C^1(H_0;H)$. Indeed, for this we would need that for all $i=1,\ldots,d$ and all $j \in \bbn$ the continuous linear functional $\la \sigma_i^j,\cdot \ra : G \to \bbr$ admits a continuous extension $\la \sigma_i^j,\cdot \ra : H_0 \to \bbr$, and this is only true if we make the stronger assumption $\sigma \in \ell^2(\cals_{-(p+\frac{1}{2})}(\bbr^d;\bbr^d))$.
\end{remark}

Let $\Phi \in G$ be arbitrary, and denote by $\psi \in C^2(\bbr^d;H)$ the orbit map given by $\psi(x) = \tau_x \Phi$ for each $x \in \bbr^d$. Due to our results from Section \ref{sec-manifolds-HS} we are in the mathematical setting of Section \ref{sec-ql-group}. In particular, by Proposition \ref{prop-Sp-in-domain} we have $H_0 \subset D(-\partial)$ and $G \subset D((-\partial)^2)$. Let $\caln$ be an $m$-dimensional $C^2$-submanifold of $\bbr^d$, and let $\calm$ be an $m$-dimensional $(G,H_0,H)$-submanifold of class $C^2$, which is induced by $(\psi,\caln)$. Recall that this requires that $\psi|_{\caln} : \caln \to \psi(\caln)$ is a homeomorphism, and that $\psi$ is a $C^2$-immersion on $\caln$.

\begin{theorem}\label{thm-M-N-HS-spaces}
The following statements are equivalent:
\begin{enumerate}
\item[(i)] The submanifold $\calm$ is locally invariant for the SPDE (\ref{SPDE}).

\item[(ii)] The submanifold $\caln$ is locally invariant for the $\bbr^d$-valued SDE
\begin{align}\label{SDE-bar-2}
\left\{
\begin{array}{rcl}
dX_t & = & \bar{b}(X_t) dt + \bar{\sigma}(X_t) dW_t
\\ X_0 & = & x_0,
\end{array}
\right.
\end{align}
where the continuous mappings $\bar{\sigma} : \bbr^d \to \ell^2(\bbr^d)$ and $\bar{b} : \bbr^d \to \bbr^d$ are defined as
\begin{align}\label{bar-sigma}
\bar{\sigma}^j &:= \la \sigma^j,\psi(\cdot) \ra, \quad j \in \bbn,
\\ \label{bar-b} \bar{b} &:= \la b,\psi(\cdot) \ra.
\end{align}
\end{enumerate}
\end{theorem}

\begin{proof}
Noting the definitions (\ref{def-A}) and (\ref{def-L}), this is a consequence of Theorem \ref{thm-SPDE-group}.
\end{proof}

\begin{proposition}\label{prop-M-N-HS-spaces}
Suppose that the following conditions are fulfilled:
\begin{enumerate}
\item The submanifold $\calm$ is locally invariant for the SPDE (\ref{SPDE}).

\item The submanifold $\caln$ has one chart with a global parametrization $\varphi : V \to \caln$.

\item The open set $V$ is globally invariant for the $\bbr^m$-valued SDE (\ref{global-1}), whose coefficients $a : V \to \ell^2(\bbr^m)$ and $\ell : V \to \bbr^m$ are the unique solutions of the equations
\begin{align*}
\bar{\sigma}^j|_{\caln} &= \varphi_* a^j, \quad j \in \bbn,
\\ \bar{b}|_{\caln} &= \varphi_* \ell + \frac{1}{2} \sum_{j=1}^{\infty} \varphi_{**}(a^j,a^j),
\end{align*}
where the continuous mappings $\bar{\sigma} : \bbr^d \to \ell^2(\bbr^d)$ and $\bar{b} : \bbr^d \to \bbr^d$ are given by (\ref{bar-sigma}) and (\ref{bar-b})
\end{enumerate}
Then the submanifold $\calm$ is globally invariant for the SPDE (\ref{SPDE}), and the submanifold $\caln$ is globally invariant for the SDE (\ref{SDE-bar-2}).
\end{proposition}

\begin{proof}
This is a consequence of Proposition \ref{prop-SPDE-group}.
\end{proof}

Now, we will construct some examples of induced submanifolds which are invariant for the SPDE (\ref{SPDE}) with coefficients given by (\ref{def-L}) and  (\ref{def-A}). Recall that $b \in \cals_{-(p+1)}(\bbr^d;\bbr^d)$ and $\sigma \in \ell^2(\cals_{-(p+1)}(\bbr^d;\bbr^d))$, and that $\la \cdot,\cdot \ra$ denotes the dual pair on $\cals_{-(p+1)}(\bbr^d) \times \cals_{p+1}(\bbr^d)$. In each of the upcoming examples, we will impose conditions on the choice of $p$. Also recall that $\bbr^m \times \{ 0 \} \subset \bbr^d$ denotes the subspace $\bbr^m \times \{ 0 \} = \lin \{ e_1,\ldots,e_m \}$, where $e_1,\ldots,e_m \in \bbr^d$ are the first $m$ unit vectors. The following examples of invariant submanifolds are consequences of Theorem \ref{thm-M-N-HS-spaces}, Proposition \ref{prop-M-N-HS-spaces} and Examples \ref{ex-manifolds-HS-spaces} with $k=2$.

\begin{example}[Distributions given by measures]\label{example-measures}
We choose $p \in \bbr$ such that $p+1 < - \frac{d}{4}$, and let $\Phi = \mu \in G$ be a finite signed measure on $(\bbr^d,\calb(\bbr^d))$ with compact support such that $\mu(\bbr^d) \neq 0$. Furthermore, setting $\caln := \bbr^m \times \{ 0 \}$ we assume that for all $x \in \caln$ we have
\begin{align}\label{linear-1a}
&\la b,\tau_x \mu \ra \in \caln,
\\ \label{linear-1b} &\la \sigma^j, \tau_x \mu \ra \in \caln, \quad j \in \bbn.
\end{align}
Then the set
\begin{align*}
\calm := \psi(\caln) = \{ \tau_x \mu : x \in \caln \}
\end{align*}
is an $m$-dimensional $(G,H_0,H)$-submanifold of class $C^2$ with one chart, which is globally invariant for the SPDE (\ref{SPDE}). The global invariance follows from Remark \ref{rem-global-weak-solution}, because the coefficients $a : \bbr^m \to \ell^2(\bbr^m)$ and $\ell : \bbr^m \to \bbr^m$ are bounded by virtue of Lemma \ref{lemma-distr-int-ex}.
\end{example}

\begin{example}[Distributions given by polynomials]
We choose $p \in \bbr$ such that $p+1 < - \frac{d}{4} - \frac{n}{2}$ for some $n \in \bbn$ such that $m \leq n \leq d$, and let $\Phi = f \in G$ be the polynomial $f : \bbr^d \to \bbr$ given by $f(x) = x_1 \cdot \ldots \cdot x_n$. Furthermore, setting $\caln := \bbr^m \times \{ 0 \}$ we assume that for all $x \in \caln$ we have
\begin{align}\label{linear-2a}
&\la b,\tau_x f \ra \in \caln,
\\ \label{linear-2b} &\la \sigma^j, \tau_x f \ra \in \caln, \quad j \in \bbn.
\end{align}
Then the set
\begin{align*}
\calm := \psi(\caln) = \{ \tau_x f : x \in \caln \}
\end{align*}
is an $m$-dimensional $(G,H_0,H)$-submanifold of class $C^2$ with one chart, which is locally invariant for the SPDE (\ref{SPDE}). If $m = n = 1$, which means that $\caln = \bbr \times \{ 0 \}$ and $f(x) = x_1$, then $\calm$ is even globally invariant for the SPDE (\ref{SPDE}). Taking into account Remark \ref{rem-global-weak-solution}, this follows from Lemma \ref{lemma-pol-lin-growth}, which ensures that the coefficients $a : \bbr^m \to \ell^2(\bbr^m)$ and $\ell : \bbr^m \to \bbr^m$ satisfy the linear growth condition.
\end{example}

For the next example, recall the Sobolev embedding theorem for Hermite Sobolev spaces (see Theorem \ref{thm-Ck-versions}).

\begin{example}[Distributions given by $C^1$-functions]
We choose $p \in \bbr$ such that $p+1 > \frac{d}{4} + \frac{1}{2}$, and let $\Phi = \varphi \in G$ be arbitrary. Setting $\caln := \bbr^m \times \{ 0 \}$, we assume there are $z_1,\ldots,z_m \in \bbr^d$ such that the matrix $(\partial_i \varphi(z_j))_{i,j=1,\ldots,m} \in \bbr^{m \times m}$ is invertible, and we assume that for all $x \in \caln$ we have
\begin{align}\label{linear-3a}
&\la b,\tau_x \varphi \ra \in \caln,
\\ \label{linear-3b} &\la \sigma^j, \tau_x \varphi \ra \in \caln, \quad j \in \bbn.
\end{align}
Then the set
\begin{align*}
\calm := \psi(\caln) = \{ \tau_x \varphi : x \in \caln \}
\end{align*}
is an $m$-dimensional $(G,H_0,H)$-submanifold of class $C^2$ with one chart, which is locally invariant for the SPDE (\ref{SPDE}). Note that the invertibility of the matrix $(\partial_i \varphi(z_j))_{i,j=1,\ldots,m}$ is required in order to ensure that $\psi$ is an immersion on $\caln$; see Proposition \ref{prop-psi-function}. If $b_i \in L^2(\bbr^d)$ for $i=1,\ldots,d$ and $\sigma_{i}^j \in L^2(\bbr^d)$ for $i=1,\ldots,d$ and $j \in \bbn$, then $\calm$ is even globally invariant for the SPDE (\ref{SPDE}). This follows from Remark \ref{rem-global-weak-solution}, because, recalling that $L^2(\bbr^d) = \cals_0(\bbr^d)$, by Lemma \ref{lemma-L2-bounded} the coefficients $a : \bbr^m \to \ell^2(\bbr^m)$ and $\ell : \bbr^m \to \bbr^m$ are bounded.
\end{example}

\begin{remark}
Note that in each of the previous examples we have considered the submanifold $\caln := \bbr^m \times \{ 0 \}$, which ensures that in any case the assumptions from Examples \ref{ex-manifolds-HS-spaces} concerning $\caln$ are fulfilled. Since the submanifold $\caln$ is a linear space, in any case the respective conditions (\ref{linear-1a})--(\ref{linear-1b}), (\ref{linear-2a})--(\ref{linear-2b}) or (\ref{linear-3a})--(\ref{linear-3b}) ensures that $\caln$ is locally invariant for the SDE (\ref{SDE-bar-2}); see Corollary \ref{cor-affine}. Of course, we can also consider other choices of the submanifold $\caln$ such that the assumptions from Examples \ref{ex-manifolds-HS-spaces} are fulfilled. In particular, noting Theorem \ref{thm-SPDE}, in the situation of Example \ref{example-measures} we can choose any $m$-dimensional $C^2$-submanifold $\caln$ of $\bbr^d$ such that
\begin{align*}
&\bar{\sigma}^j|_{\caln} \in \Gamma(T \caln), \quad j \in \bbn,
\\ &[\bar{b}|_{\caln}]_{\Gamma(T \caln)} - \frac{1}{2} \sum_{j=1}^{\infty} [ \bar{\sigma}^j|_{\caln}, \bar{\sigma}^j|_{\caln} ]_{\caln} = [0]_{\Gamma(T \caln)},
\end{align*}
where the continuous mappings $\bar{\sigma} : \bbr^d \to \ell^2(\bbr^d)$ and $\bar{b} : \bbr^d \to \bbr^d$ are defined as
\begin{align*}
\bar{\sigma}^j(x) &:= \la \sigma^j,\tau_x \mu \ra, \quad j \in \bbn,
\\ \bar{b}(x) &:= \la b,\tau_x \mu \ra
\end{align*}
for each $x \in \bbr^d$.
\end{remark}

\begin{remark}\label{rem-delta-0}
Consider the particular situation $m=d$, $\caln = \bbr^d$ and $\Phi = \delta_0$, which is covered by Example \ref{example-measures}. Then, by Lemma \ref{lemma-delta-distr} the invariant submanifold is given by
\begin{align*}
\calm = \{ \delta_x : x \in \bbr^d \},
\end{align*}
and the coefficients of the SDE (\ref{SDE-bar-2}) are simply given by $\bar{b} = b$ and $\bar{\sigma} = \sigma$.
\end{remark}

\begin{remark}
Note that the findings of this section are in accordance with \cite[Lemma 3.6]{Rajeev}, where it was shown that solutions to the SPDE (\ref{SPDE}) with coefficients (\ref{def-L}) and (\ref{def-A}) can be realized locally as $Y_t = \tau_{X_t} \Phi$ with an $\bbr^d$-valued It\^{o} process $X$.
\end{remark}

\section{Interplay between SPDEs and finite dimensional SDEs}\label{sec-interplay}

In this section we illustrate how our findings from the previous Section \ref{sec-examples-HS} can be used in order to study stochastic invariance for finite dimensional diffusions. Consider the $\bbr^d$-valued SDE
\begin{align}\label{SDE-b-sigma}
\left\{
\begin{array}{rcl}
dX_t & = & b(X_t)dt + \sigma(X_t) dW_t
\\ X_0 & = & x_0
\end{array}
\right.
\end{align}
with measurable mappings $b : \bbr^d \to \bbr^d$ and $\sigma : \bbr^d \to \ell^2(\bbr^d)$. We assume that for some $q > \frac{d}{4}$ we have $b \in \cals_q(\bbr^d;\bbr^d)$ and $\sigma \in \ell^2(\cals_q(\bbr^d;\bbr^d))$. 

\begin{remark}
Note that sufficient conditions for the assumption that the components of $b$ and $\sigma$ belong to $\cals_q(\bbr^d)$ are provided by Proposition \ref{prop-Sobolev-embedding} and Corollary \ref{cor-Sobolev-embedding}.
\end{remark}

\begin{lemma}
The mappings $b : \bbr^d \to \bbr^d$ are $\sigma : \bbr^d \to \ell^2(\bbr^d)$ are continuous and bounded.
\end{lemma}

\begin{proof}
By the Sobolev embedding theorem for Hermite Sobolev spaces (Theorem \ref{thm-Ck-versions}) the mapping $b : \bbr^d \to \bbr^d$ is continuous and bounded, and for each $j \in \bbn$ the mapping $\sigma^j : \bbr^d \to \bbr^d$ is continuous and bounded. Let $x \in \bbr^d$ and $j \in \bbn$ be arbitrary. By Theorem \ref{thm-Ck-versions} we have
\begin{align*}
\| \sigma^j(x) \| \leq C \| \sigma^j \|_{q,d}
\end{align*}
with a universal constant $C > 0$. Therefore, by Lebesgue's dominated convergence theorem the claim follows.
\end{proof}

Consequently, by Remark \ref{rem-global-weak-solution} for each $x_0 \in \bbr^d$ there exists a global weak solution $X$ to the SDE (\ref{SDE-b-sigma}) with $X_0 = x_0$. Let $\caln$ be an $m$-dimensional $C^2$-submanifold of $\bbr^d$ for some $m \leq d$. Taking into account Remark \ref{rem-delta-0}, our idea is to link invariance of the submanifold $\caln$ for the SDE (\ref{SDE-b-sigma}) with invariance of the submanifold $\calm$ for the SPDE (\ref{SPDE}) in Hermite Sobolev spaces, where $\calm$ is defined in (\ref{M-intro}) below. For this purpose, we set $p := -(q+1)$. Then we have $q = -(p+1)$ as well as $p+1 < -\frac{d}{4}$, and hence, we can consider the SPDE (\ref{SPDE}) with coefficients (\ref{def-L}) and (\ref{def-A}) in the framework of the previous Section \ref{sec-examples-HS} with $G = \cals_{-q}(\bbr^d)$, $H_0 = \cals_{-(q + \frac{1}{2})}(\bbr^d)$, $H = \cals_{-(q+1)}(\bbr^d)$ and $\Phi = \delta_0$. As pointed out in Remark \ref{rem-delta-0}, then the coefficients of the SDE (\ref{SDE-bar-2}) are simply given by $\bar{b} = b$ and $\bar{\sigma} = \sigma$, and hence, the SDE (\ref{SDE-b-sigma}) from this section coincides with the SDE (\ref{SDE-bar-2}). By Lemma \ref{lemma-delta-distr}, the orbit map $\psi \in C^2(\bbr^d;H)$ is given by $\psi(x) = \delta_x$ for each $x \in \bbr^d$. Therefore, by Examples \ref{ex-manifolds-HS-spaces} with $k=2$ the set
\begin{align}\label{M-intro}
\calm := \psi(\caln) = \{ \delta_x : x \in \caln \}
\end{align}
is a $d$-dimensional $(G,H_0,H)$-submanifold of class $C^2$, which is induced by $(\psi,\caln)$. The following result shows how local invariance of the submanifold $\caln$ for the SDE (\ref{SDE-b-sigma}) is connected with local invariance of the submanifold $\calm$ for the SPDE (\ref{SPDE}).

\begin{theorem}\label{thm-M-N}
The following statements are equivalent:
\begin{enumerate}
\item[(i)] The submanifold $\calm$ is locally invariant for the SPDE (\ref{SPDE}).

\item[(ii)] The submanifold $\caln$ is locally invariant for the SDE (\ref{SDE-b-sigma}).
\end{enumerate}
\end{theorem}

\begin{proof}
Taking into account Remark \ref{rem-delta-0}, this is a consequence of Theorem \ref{thm-M-N-HS-spaces}.
\end{proof}

\begin{proposition}\label{prop-M-N}
Suppose that the submanifold $\caln$ is locally invariant for the SDE (\ref{SDE-b-sigma}). Then the following statements are true:
\begin{enumerate}
\item If the submanifold $\caln$ has one chart with a global parametrization $\varphi : V \to \caln$, and the open set $V$ is globally invariant for the $\bbr^m$-valued SDE (\ref{global-1}), whose coefficients $a : V \to \ell^2(\bbr^m)$ and $\ell : V \to \bbr^m$ are the unique solutions of the equations
\begin{align*}
\sigma^j|_{\caln} &= \varphi_* a^j, \quad j \in \bbn,
\\ b|_{\caln} &= \varphi_* \ell + \frac{1}{2} \sum_{j=1}^{\infty} \varphi_{**}(a^j,a^j),
\end{align*}
then the submanifold $\calm$ is globally invariant for the SPDE (\ref{SPDE}), and the submanifold $\caln$ is globally invariant for the SDE (\ref{SDE-b-sigma}).

\item If the submanifold $\caln$ is closed as a subset of $\bbr^d$, then it is globally invariant for the SDE (\ref{SDE-b-sigma}).
\end{enumerate}
\end{proposition}

\begin{proof}
The first statement is a consequence of Proposition \ref{prop-M-N-HS-spaces}. In the situation of the second statement, let $x_0 \in \caln$ be arbitrary, and let $X$ be a global weak solution to the SDE (\ref{SDE-b-sigma}) with $X_0 = x_0$. We define the stopping time
\begin{align*}
\tau := \inf \{ t \in \bbr_+ : X_t \notin \caln \},
\end{align*}
and, since $\caln$ is closed as a subset of $\bbr^d$, arguing by contradiction we can show that $\bbp(\tau = \infty) = 1$; see, for example, the proof of \cite[Thm. 2.8]{FTT-manifolds}.
\end{proof}

Consequently, when we are interested in proving local invariance of the submanifold $\caln$ for the SDE (\ref{SDE-b-sigma}), we can alternatively show local invariance of the submanifold $\calm$ for the SPDE (\ref{SPDE}), which turns out to be simpler in certain situations. We illustrate this procedure in the upcoming two subsections, which are organized as follows: In Section \ref{sub-sec-vector-fields} we treat the invariance of submanifolds for coefficients given by vector fields, and in Section \ref{sub-sec-zeros} we investigate the invariance of submanifolds given by the zeros of smooth functions.

\subsection{Coefficients given by vector fields}\label{sub-sec-vector-fields}

For the following results, consider the conditions
\begin{align}\label{conditions-1}
b|_{\caln} &\in \Gamma(T \caln), 
\\ \label{conditions-2} \sigma^j|_{\caln} &\in \Gamma(T \caln), \quad j \in \bbn.
\end{align}
We are interested in finding an additional condition ensuring that $\caln$ is locally invariant for the SDE (\ref{SDE-b-sigma}). In the general framework of Section \ref{sec-SPDE-general}, such a condition is provided by Proposition \ref{prop-L-tang}. In the present situation, we will establish another equivalent condition by using the connection to the SPDE (\ref{SPDE}). For the following auxiliary result, recall the definition (\ref{A-bar-y-z}) of $\bar{A}^j$ for $j \in \bbn$.

\begin{lemma}\label{lemma-inv-b}
The following statements are true:
\begin{enumerate}
\item For all $y,z \in \calm$ we have
\begin{align*}
\bar{A}^j(y,z) &= D \psi(x) \sigma^j(\xi), \quad j \in \bbn,
\end{align*}
where $x := \psi^{-1}(y) \in \caln$ and $\xi := \psi^{-1}(z) \in \caln$.

\item In particular, we have
\begin{align*}
A^j|_{\calm} = \psi_* \sigma^j|_{\caln}, \quad j \in \bbn.
\end{align*}

\item We have
\begin{align*}
L|_{\calm} - \frac{1}{2} \sum_{j=1}^{\infty} \bar{A}^j(A^j(\cdot),\cdot)|_{\calm} = \psi_* b|_{\caln}.
\end{align*}
\end{enumerate}
\end{lemma}

\begin{proof}
Recalling (\ref{def-L}) and (\ref{def-A}), these statements follow from Proposition \ref{prop-orbit-in-Sp}, where for the third statement we note that for each $y \in \calm$ we have
\begin{align*}
\sum_{j=1}^{\infty} \bar{A}^j(A^j(y),y) &= - \sum_{j=1}^{\infty} \sum_{i=1}^d \la \sigma_{i}^j, y \ra \partial_i A^j(y) = \sum_{j=1}^{\infty} \sum_{i=1}^d \la \sigma_{i}^j, y \ra \sum_{k=1}^d \la \sigma_{k}^j,y \ra \partial_{ik}^2 y
\\ &= \sum_{i,k=1}^d ( \langle \sigma,y \rangle \langle \sigma,y \rangle^{\top} )_{ik} \partial_{ik}^2 y.
\end{align*}
This completes the proof.
\end{proof}

Concerning the notation used in equations (\ref{quad-v-zero}) and (\ref{quad-A}) below, we refer to Definition \ref{def-A-Strat}.

\begin{theorem}\label{thm-inv-b}
Suppose that conditions (\ref{conditions-1}) and (\ref{conditions-2}) are fulfilled. Then the following statements are equivalent:
\begin{enumerate}
\item[(i)] $\caln$ is locally invariant for the SDE (\ref{SDE-b-sigma}).

\item[(ii)] We have
\begin{align}\label{quad-v-zero}
\sum_{j=1}^{\infty} [ \sigma^j|_{\caln}, \sigma^j|_{\caln} ]_{\caln} = [0]_{\Gamma(T \caln)}.
\end{align}
\item[(iii)] We have
\begin{align}\label{quad-A}
\sum_{j=1}^{\infty} \big( [A^j|_{\calm},A^j|_{\calm}]_{\calm} - [ \bar{A}^j(A^j(\cdot),\cdot)|_{\calm} ]_{\Gamma(T \calm)} \big) = [0]_{\Gamma(T \calm)}.
\end{align}
\end{enumerate}
\end{theorem}

\begin{proof}
(i) $\Leftrightarrow$ (ii): This equivalence is a consequence of Proposition \ref{prop-L-tang}.

\noindent(i) $\Leftrightarrow$ (iii): By Lemmas \ref{lemma-inv-b} and \ref{lemma-eq-loc-fields} we have
\begin{align*}
&A^j|_{\calm} \in \Gamma(T \calm), \quad j \in \bbn,
\\ &L|_{\calm} - \frac{1}{2} \sum_{j=1}^{\infty} \bar{A}^j(A^j(\cdot),\cdot)|_{\calm} \in \Gamma(T \calm).
\end{align*}
The latter relation shows that
\begin{align*}
&[ L|_{\calm} ]_{\Gamma(T \calm)} - \frac{1}{2} \sum_{j=1}^{\infty} [A^j|_{\calm}, A^j|_{\calm}]_{\calm}
\\ &= \frac{1}{2} \sum_{j=1}^{\infty} \big( [A^j|_{\calm},A^j|_{\calm}]_{\calm} - [ \bar{A}^j(A^j(\cdot),\cdot)|_{\calm} ]_{\Gamma(T \calm)} \big),
\end{align*}
and thus, the stated equivalence is a consequence of Theorem \ref{thm-M-N} and Theorem \ref{thm-SPDE}.
\end{proof}

If the submanifold $\caln$ is affine, then it is locally invariant for the SDE (\ref{SDE-b-sigma}) if and only if we have (\ref{conditions-1}) and (\ref{conditions-2}). This is a consequence of Corollary \ref{cor-affine}. More generally, we have the following result. Recall that $\Gamma^*(T \caln)$ denotes the space of all locally simultaneous vector fields on $\caln$; see Definition \ref{def-simultaneous}. 

\begin{proposition}\label{prop-inv-fin-dim}
Suppose that
\begin{align}\label{inv-cond-local}
b|_{\caln} &\in \Gamma(T \caln),
\\ \label{inv-cond-local-2} \sigma^j|_{\caln} &\in \Gamma^*(T \caln), \quad j \in \bbn.
\end{align}
Then the submanifold $\caln$ is locally invariant for the SDE (\ref{SDE-b-sigma}).
\end{proposition}

\begin{proof}
By Lemmas \ref{lemma-inv-b} and \ref{lemma-eq-loc-fields} we have
\begin{align}\label{A-in-proof}
&\bar{A}_z^j|_{\calm} \in \Gamma_z(T \calm), \quad z \in \calm, \quad j \in \bbn,
\\ \label{L-in-proof} &L|_{\calm} - \frac{1}{2} \sum_{j=1}^{\infty} \bar{A}^j(A^j(\cdot),\cdot)|_{\calm} \in \Gamma(T \calm),
\end{align}
where $\Gamma_z(T \calm)$ denotes the space of all local vector fields on $\calm$ around $z$; see Definition \ref{def-local-vector-field}. Using the decomposition (\ref{zerl-1-linear}) from Proposition \ref{prop-decomp-Damir-scale}, we obtain
\begin{align*}
[ A^j|_{\calm}, A^j|_{\calm} ]_{\calm} = [ \bar{A}^j(A^j(\cdot),\cdot) ]_{\Gamma(T \calm)}, \quad j \in \bbn.
\end{align*}
Therefore, condition (\ref{quad-A}) is fulfilled, and hence, by Theorem \ref{thm-inv-b} the submanifold $\caln$ is locally invariant for the SDE (\ref{SDE-b-sigma}).
\end{proof}

\begin{remark}
Once we have established (\ref{A-in-proof}) and (\ref{L-in-proof}), alternatively we can also use Theorem \ref{thm-main-3} and Theorem \ref{thm-M-N} in order to conclude the proof of Proposition \ref{prop-inv-fin-dim}.
\end{remark}

\begin{remark}\label{rem-Strat-1}
Consider the $\bbr^d$-valued Stratonovich SDE
\begin{align}\label{SDE-Strat}
\left\{
\begin{array}{rcl}
dX_t & = & c(X_t)dt + \sigma(X_t) \circ dW_t
\\ X_0 & = & x_0
\end{array}
\right.
\end{align}
with a continuous mapping $c : \bbr^d \to \bbr^d$. It is well known that the submanifold $\caln$ is locally invariant for the Stratonovich SDE (\ref{SDE-Strat}) if and only if
\begin{align*}
c|_{\caln} &\in \Gamma(T \caln),
\\ \sigma^j|_{\caln} &\in \Gamma(T \caln), \quad j \in \bbn,
\end{align*}
see, for example \cite[Cor. 1.ii]{Milian-manifold}. In Proposition \ref{prop-inv-fin-dim} we present similar conditions, namely (\ref{inv-cond-local}) and (\ref{inv-cond-local-2}), which are sufficient for local invariance of the submanifold $\caln$ for the It\^{o} SDE (\ref{SDE-b-sigma}).
\end{remark}

For the following results we will assume that even $\sigma \in \ell^2(\cals_{q + \frac{1}{2}}(\bbr^d;\bbr^d))$. Note that $p + \frac{1}{2} = -(q + \frac{1}{2})$ and $q + \frac{1}{2} = -(p + \frac{1}{2})$, which shows that
\begin{align*}
H_0 = \cals_{p+\frac{1}{2}}(\bbr^d) = \cals_{-(q+\frac{1}{2})}(\bbr^d).
\end{align*}
Furthermore, note that for each $j \in \bbn$ the definition (\ref{A-bar-y-z}), where $\la \cdot,\cdot \ra$ denotes the dual pair on $\cals_{q + \frac{1}{2}}(\bbr^d) \times \cals_{-(q + \frac{1}{2})}(\bbr^d)$, provides bounded bilinear operators $\bar{A}^j \in L(G,H_0;H_0)$ and $\bar{A}^j \in L(H_0,H_0;H)$.

\begin{lemma}\label{lemma-A-bar-q}
Suppose that $\sigma \in \ell^2(\cals_{q + \frac{1}{2}}(\bbr^d;\bbr^d))$. Then the following statements are true:
\begin{enumerate}
\item The definition (\ref{A-bar-y-z}) provides a bounded bilinear operator
\begin{align*}
\bar{A} \in L(G,H_0;\ell^2(H_0)).
\end{align*}

\item The definition (\ref{def-A}) provides a continuous mapping $A : H_0 \to \ell^2(H)$.

\item The definition (\ref{A-bar-y-z}) provides a bounded bilinear operator
\begin{align*}
\bar{A} \in L(H_0,H_0;\ell^2(H)).
\end{align*}

\item We have $A^j \in C^1(H_0;H)$ for each $j \in \bbn$, and the mapping $G \to H$, $y \mapsto \sum_{j=1}^{\infty} D A^j(y) A^j(y)$ is well-defined and continuous.
\end{enumerate}
\end{lemma}

\begin{proof}
Using Lemma \ref{lemma-diff-continuous}, for all $y,z \in G$ we have
\begin{align*}
\| \bar{A}(y,z) \|_{\ell^2(H_0)}^2 &\leq d \sum_{j=1}^{\infty} \sum_{i=1}^d |\langle \sigma_{i}^j,z \rangle|^2 \| \partial_i y \|_{H_0}^2 \leq d \sum_{j=1}^{\infty} \sum_{i=1}^d \| \sigma_{i}^j \|_{q + \frac{1}{2}}^2 \| z \|_{H_0}^2 \| y \|_{G}^2
\\ &\leq Cd \| \sigma \|_{q + \frac{1}{2},\ell^2}^2 \| y \|_{H_0}^2 \| z \|_G^2
\end{align*}
with a universal constant $C > 0$, proving the first statement. Hence, the second statement follows as well. Similarly, using Lemma \ref{lemma-diff-continuous}, for all $y,z \in H_0$ we have
\begin{align*}
\| \bar{A}(y,z) \|_{\ell^2(H)}^2 &\leq d \sum_{i=1}^d |\langle \sigma_{i}^j,z \rangle|^2 \| \partial_i y \|_{H}^2 \leq d \sum_{i=1}^d \| \sigma_{i}^j \|_{q + \frac{1}{2}}^2 \| z \|_{H_0}^2 \| y \|_{H_0}^2
\\ &\leq Cd \| \sigma \|_{q + \frac{1}{2},\ell^2}^2 \| y \|_{H_0}^2 \| z \|_{H_0}^2
\end{align*}
with a universal constant $C > 0$, proving the third statement. For each $j \in \bbn$ we have $A^j \in C^1(H_0;H)$, because $A^j(y) = \bar{A}^j(y,y)$ for each $y \in H_0$ and $\bar{A}^j \in L(H_0,H_0;H)$. Now, for $j \in \bbn$ we define $\Phi^j : \ell^2(L(H_0,H)) \times \ell^2(H_0) \to H$ as
\begin{align*}
\Phi^j(T,z) := T^j z^j.
\end{align*}
Then $\Phi := (\Phi^j)_{j \in \bbn}$ provides a bounded bilinear operator
\begin{align*}
\Phi \in L(\ell^2(L(H_0,H)),\ell^2(H_0);\ell^1(H)). 
\end{align*}
Indeed, by the Cauchy-Schwarz inequality, for all $T \in L(\ell^2(L(H_0,H))$ and $z \in \ell^2(H_0)$ we obtain
\begin{align*}
\| \Phi(T,z) \|_{\ell^1(H)} &= \sum_{j=1}^{\infty} \| T^j z^j \|_H \leq \sum_{j=1}^{\infty} \| T^j \|_{L(H_0,H)} \| z^j \|_{H_0}
\\ &\leq \| T \|_{\ell^2(L(H_0,H))} \| z \|_{\ell^2(H_0)}.
\end{align*}
Furthermore, note that
\begin{align*}
L(H_0,\ell^2(L(H_0,H))) \cong L(H_0,H_0;\ell^2(H)).
\end{align*}
Indeed, for $T \in L(H_0,\ell^2(L(H_0,H)))$ we assign $\phi_T \in L(H_0,H_0;\ell^2(H))$ as
\begin{align*}
\phi_T(y,z) = \big( (Ty)^j z \big)_{j \in \bbn},
\end{align*}
which provides an isometric isomorphism. By identification, for each $j \in \bbn$ we have $B^j := D A^j \in L(H_0,H_0;H)$, and by the Leibniz rule we obtain 
\begin{align*}
B^j(y,z) = \bar{A}^j(y,z) + \bar{A}^j(z,y), \quad y,z \in H_0.
\end{align*}
The series $B := (B^j)_{j \in \bbn}$ provides a bounded bilinear operator
\begin{align*}
B \in L(H_0,H_0;\ell^2(H)).
\end{align*}
Indeed, using Lemma \ref{lemma-diff-continuous}, for all $y,z \in H_0$ we have
\begin{align*}
\| B(y,z) \|_{\ell^2(H)}^2 &= \sum_{j=1}^{\infty} \| B^j(y,z) \|_H^2 \leq 4Cd \sum_{j=1}^{\infty} \| \sigma^j \|_{q + \frac{1}{2},d}^2 \| y \|_{H_0}^2 \| z \|_{H_0}^2
\\ &= 4Cd \| \sigma \|_{q + \frac{1}{2},\ell^2}^2 \| y \|_{H_0}^2 \| z \|_{H_0}^2
\end{align*}
with a universal constant $C > 0$. By Lemma \ref{lemma-evaluation} the mapping $\ell^1(H) \to H$, $z \mapsto \sum_{j=1}^{\infty} z^j$ belongs to $L(\ell^1(H),H)$. Therefore, and since $A : G \to \ell^2(H_0)$ is continuous, the proof is completed.
\end{proof}

\begin{lemma}\label{lemma-inv-b-c}
Suppose that $\sigma \in \ell^2(\cals_{q + \frac{1}{2}}(\bbr^d;\bbr^d))$. Then we have $\sigma^j \in C_0^1(\bbr^d)$ for each $j \in \bbn$, and the mapping $\bbr^d \to \bbr^d$, $x \mapsto \sum_{j=1}^{\infty} D \sigma^j(x) \sigma^j(x)$ is well-defined and continuous.
\end{lemma}

\begin{proof}
Let $j \in \bbn$ be arbitrary. By the Sobolev embedding theorem for Hermite Sobolev spaces (Theorem \ref{thm-Ck-versions}) we have $\sigma^j \in C_0^1(\bbr^d)$, and for each $x \in \bbr^d$ we have
\begin{align*}
\| \sigma^j(x) \| + \| D \sigma^j(x) \| &\leq C \| \sigma^j \|_{\cals_{q + \frac{1}{2}}(\bbr^d;\bbr^d)}
\end{align*}
with a universal constant $C > 0$, and hence
\begin{align*}
\| D \sigma^j(x) \sigma^j(x) \| \leq \| D \sigma^j(x) \| \, \| \sigma^j(x) \| \leq C^2 \| \sigma^j \|_{\cals_{q + \frac{1}{2}}(\bbr^d;\bbr^d)}^2.
\end{align*}
Hence, the statement follows from Lebesgue's dominated convergence theorem.
\end{proof}

\begin{remark}
If $\sigma \in \ell^2(\cals_{q + \frac{1}{2}}(\bbr^d;\bbr^d))$, then by Lemma \ref{lemma-inv-b-c} the It\^{o} SDE (\ref{SDE-b-sigma}) can equivalently be expressed by the Stratonovich SDE (\ref{SDE-Strat}), where the continuous mapping $c : \bbr^d \to \bbr^d$ is given by
\begin{align}\label{c-Strat}
c = b - \frac{1}{2} \sum_{j=1}^{\infty} D \sigma^j \cdot \sigma^j.
\end{align}
\end{remark}

\begin{proposition}\label{prop-decomp-A-sigma}
Suppose that $\sigma \in \ell^2(\cals_{q + \frac{1}{2}}(\bbr^d;\bbr^d))$. Then we have the decomposition
\begin{align}\label{decomp-A-sigma}
\sum_{j=1}^{\infty} D A^j \cdot A^j|_{\calm} = \sum_{j=1}^{\infty} \bar{A}^j(A^j(\cdot),\cdot)|_{\calm} + \psi_* \bigg( \sum_{j=1}^{\infty} D \sigma^j \cdot \sigma^j|_{\caln} \bigg).
\end{align}
\end{proposition}

\begin{proof}
Let $j \in \bbn$ be arbitrary. By the Leibniz rule we have
\begin{align*}
D A^j(y) z = \bar{A}^j(y,z) + \bar{A}^j(z,y), \quad y,z \in G,
\end{align*}
and hence
\begin{align*}
D A^j(y) A^j(y) = \bar{A}^j(y,A^j(y)) + \bar{A}^j(A^j(y),y), \quad y \in G.
\end{align*}
Now, let $y \in \calm$ be arbitrary. Then we have $y = \delta_x$, where $x := \psi^{-1}(y) \in \caln$. Therefore, by duality we obtain
\begin{align*}
\bar{A}^j(y,A^j(y)) &= - \sum_{i=1}^d \la \sigma_{i}^j, A^j(y) \ra \partial_i y = \sum_{i=1}^d \sum_{k=1}^d \la \sigma_{i}^j, \la \sigma_{k}^j, y \ra \partial_k y \ra \partial_i y
\\ &= - \sum_{i=1}^d \sum_{k=1}^d \la \partial_k \sigma_{i}^j, y \ra \la \sigma_{k}^j, y \ra \partial_i y = - \sum_{i=1}^d \sum_{k=1}^d \partial_k \sigma_{i}^j(x) \sigma_{k}^j(x) \partial_i y
\\ &= - \sum_{i=1}^d \la e_i, D \sigma^j(x) \sigma^j(x) \ra \partial_i y.
\end{align*}
Therefore, for all $y \in \calm$ we obtain
\begin{align*}
\sum_{j=1}^{\infty} D A^j(y) A^j(y) = \sum_{j=1}^{\infty} \bar{A}^j(A^j(y),y) - \sum_{i=1}^d \bigg\la e_i, \sum_{j=1}^{\infty} D \sigma^j(x) \sigma^j(x) \bigg\ra \partial_i y.
\end{align*}
Consequently, using Proposition \ref{prop-orbit-in-Sp} completes the proof.
\end{proof}

\begin{proposition}\label{prop-connection-b-c}
Suppose that $\sigma \in \ell^2(\cals_{q + \frac{1}{2}}(\bbr^d;\bbr^d))$. If conditions (\ref{conditions-1}) and (\ref{conditions-2}) are fulfilled, then the following statements are equivalent:
\begin{enumerate}
\item[(i)] The submanifold $\caln$ is locally invariant for the SDE (\ref{SDE-b-sigma}).

\item[(ii)] We have $\sum_{j=1}^{\infty} D \sigma^j \cdot \sigma^j|_{\caln} \in \Gamma(T \caln)$.

\item[(iii)] We have $c|_{\caln} \in \Gamma(T \caln)$, where the continuous mapping $c : \bbr^d \to \bbr^d$ is given by (\ref{c-Strat}).
\end{enumerate}
If any of the previous conditions is fulfilled, then we have
\begin{align*}
L|_{\calm} - \frac{1}{2} \sum_{j=1}^{\infty} \bar{A}^j(A^j(\cdot),\cdot)|_{\calm} = \psi_* b|_{\caln} \in \Gamma(T \calm),
\\ L|_{\calm} - \frac{1}{2} \sum_{j=1}^{\infty} D A^j \cdot A^j|_{\calm} = \psi_* c|_{\caln} \in \Gamma(T \calm),
\end{align*}
and the difference is given by
\begin{align*}
&\bigg( L|_{\calm} - \frac{1}{2} \sum_{j=1}^{\infty} \bar{A}^j(A^j(\cdot),\cdot)|_{\calm} \bigg) - \bigg( L|_{\calm} - \frac{1}{2} \sum_{j=1}^{\infty} D A^j \cdot A^j|_{\calm} \bigg) 
\\ &= \frac{1}{2} \psi_* \bigg( \sum_{j=1}^{\infty} D \sigma^j \cdot \sigma^j|_{\caln} \bigg) \in \Gamma(T \calm).
\end{align*}
\end{proposition}

\begin{proof}
Noting Lemma \ref{lemma-A-bar-q}, the equivalences (i) $\Leftrightarrow$ (ii) $\Leftrightarrow$ (iii) are a consequence of Theorem \ref{thm-main-2}. The additional statements follow from Lemma \ref{lemma-inv-b} and the decomposition (\ref{decomp-A-sigma}) from Proposition \ref{prop-decomp-A-sigma}.
\end{proof}

Consequently, we see the following connection between the coefficients of the SDE (\ref{SDE-b-sigma}) and the associated SPDE (\ref{SPDE}). The vector field in (\ref{L-vf-A2}) corresponds to the drift $b$, and the vector field in (\ref{tang-L-2}) corresponds to the Stratonovich corrected drift $c$. Furthermore, we have computed the difference between these two vector fields, which in the general situation has been determined in Remark \ref{rem-difference}.

\subsection{Submanifolds given by zeros of smooth functions}\label{sub-sec-zeros}

Now, let $\caln$ be a $(d-n)$-dimensional $C^2$-submanifold of $\bbr^d$ for some $n \in \bbn$ such that $n < d$. We assume there exist an open subset $O \subset \bbr^d$ and a mapping $f : \bbr^d \to \bbr^n$ such that
\begin{align}\label{N-zeros}
\caln = \{ x \in O : f(x) = 0 \}.
\end{align}
Concerning the components of $f$ we assume that $f_k \in \cals_{q+1}(\bbr^d)$ for all $k=1,\ldots,n$. Recalling that $q > \frac{d}{4}$, by the Sobolev embedding theorem for Hermite Sobolev spaces (Theorem \ref{thm-Ck-versions}) we have $f \in C^2(\bbr^d;\bbr^n)$. We also assume that $Df(x)\bbr^d = \bbr^n$ for all $x \in \caln$. Then, by Lemma \ref{lemma-tang-nullstellen} we have
\begin{align}\label{N-zeros-tang}
T_x \caln = \ker D f(x) \quad \text{for each $x \in \caln$.}
\end{align}
We define the operator $\call : C^2(\bbr^d) \to C(\bbr^d)$ as
\begin{align*}
(\call g)(x) := \sum_{i=1}^d b_i(x) \partial_i g(x) + \frac{1}{2} \sum_{i,j=1}^d ( \sigma(x) \sigma(x)^{\top} )_{ij} \partial_{ij}^2 g(x), \quad x \in \bbr^d,
\end{align*}
and for each $j \in \bbn$ we define the operator $\cala^j : C^2(\bbr^d) \to C(\bbr^d)$ as
\begin{align*}
(\cala^j g)(x) := \sum_{i=1}^d \sigma_{i}^j(x) \partial_i g(x), \quad x \in \bbr^d.
\end{align*}

\begin{theorem}\label{thm-nullstellen}
The following statements are equivalent:
\begin{enumerate}
\item[(i)] The submanifold $\caln$ is locally invariant for the SDE (\ref{SDE-b-sigma}).

\item[(ii)] For all $k = 1,\ldots,n$ we have
\begin{align}\label{cond-1}
\call f_k |_{\caln} &= 0,
\\ \label{cond-2} \cala^j f_k |_{\caln} &= 0, \quad j \in \bbn.
\end{align}
\end{enumerate}
\end{theorem}

Before we provide the proof of Theorem \ref{thm-nullstellen}, let us state some consequences. Note that we can decompose the operator $\call$ as $\call = \call_1 + \call_2$, where the first order operator $\call_1 : C^2(\bbr^d) \to C(\bbr^d)$ is given by
\begin{align*}
(\call_1 g)(x) := \sum_{i=1}^d b_i(x) \partial_i g(x), \quad x \in \bbr^d,
\end{align*}
and the second order operator $\call_2 : C^2(\bbr^d) \to C(\bbr^d)$ is given by
\begin{align*}
(\call_2 g)(x) := \frac{1}{2} \sum_{i,j=1}^d ( \sigma(x) \sigma(x)^{\top} )_{ij} \partial_{ij}^2 g(x), \quad x \in \bbr^d.
\end{align*}

\begin{proposition}
Conditions (\ref{conditions-1}) and (\ref{conditions-2}) are satisfied if and only if for all $k=1,\ldots,n$ we have
\begin{align*}
\call_1 f_k |_{\caln} &= 0,
\\ \cala^j f_k |_{\caln} &= 0, \quad j \in \bbn,
\end{align*}
and in this case, the following statements are equivalent:
\begin{enumerate}
\item[(i)] $\caln$ is locally invariant for the SDE (\ref{SDE-b-sigma}).

\item[(ii)] We have $\call_2 f_k |_{\caln} = 0$ for all $k=1,\ldots,n$.
\end{enumerate}
\end{proposition}

\begin{proof}
The first equivalence follows from Lemma \ref{lemma-tang-nullstellen}, and in this case, the equivalence (i) $\Leftrightarrow$ (ii) is a consequence of Theorem \ref{thm-nullstellen}.
\end{proof}

\begin{corollary}[Unit sphere]\label{cor-unit-sphere}
Let $d \geq 2$ be arbitrary, and consider the unit sphere $\bbs^{d-1} = \{ x \in \bbr^d : \| x \| = 1 \}$. Then the following statements are equivalent:
\begin{enumerate}
\item[(i)] $\bbs^{d-1}$ is globally invariant for the SDE (\ref{SDE-b-sigma}).

\item[(ii)] $\bbs^{d-1}$ is locally invariant for the SDE (\ref{SDE-b-sigma}).

\item[(iii)] For each $x \in \bbs^{d-1}$ we have
\begin{align}\label{sphere-b}
&\la x,b(x) \ra + \frac{1}{2} \tr \big( \sigma(x) \sigma(x)^{\top} \big) = 0,
\\ \label{sphere-sigma} &\la x,\sigma^j(x) \ra = 0, \quad j \in \bbn.
\end{align}
\end{enumerate}
\end{corollary}

\begin{proof}
(i) $\Leftrightarrow$ (ii): Since $\bbs^{d-1}$ is a closed subset of $\bbr^d$, this equivalence follows from Proposition \ref{prop-M-N}.

\noindent(ii) $\Leftrightarrow$ (iii): By Lemma \ref{lemma-fct-comp-supp} there exists a function $f \in \cals_q(\bbr^d)$ such that
\begin{align*}
f(x) = \| x \|^2 - 1, \quad x \in O,
\end{align*}
where $O \subset \bbr^d$ denotes the open set $O = \{ x \in \bbr^d : \| x \| < 2 \}$. Furthermore, the unit sphere $\bbs^{d-1}$ is a $(d-1)$-dimensional submanifold having the representation
\begin{align*}
\bbs^{d-1} = \{ x \in O : f(x) = 0 \}.
\end{align*}
For each $x \in O$ we obtain the partial derivatives
\begin{align*}
\partial_i f(x) &= 2 x_i, \quad i=1,\ldots,d,
\\ \partial_{ij}^2 f(x) &= 2 \delta_{ij}, \quad i,j=1,\ldots,d,
\end{align*}
which in particular shows that $Df(x)\bbr^d = \bbr$ for all $x \in \bbs^{d-1}$. Therefore, applying Theorem \ref{thm-nullstellen} completes the proof.
\end{proof}

\begin{example}[Stroock’s representation of spherical Brownian motion]\label{example-Stroock}
Let $\bbs^{d-1}$ be the unit sphere in $\bbr^d$, and consider the $\bbr^d$-valued Stratonovich SDE
\begin{align}
\left\{
\begin{array}{rcl}
dX_t & = & ( \Id - X_t X_t^{\top} ) \circ dW_t
\\ X_0 & = & x_0
\end{array}
\right.
\end{align}
with an $\bbr^d$-valued Wiener process $W$; see \cite[Example 3.3.2]{Hsu}. With our notation, the volatilities $\sigma^1,\ldots,\sigma^d : \bbr^d \to \bbr^d$ are given by
\begin{align*}
\sigma^j(x) = ( \delta_{ij} - x_i x_j )_{i=1,\ldots,d} = e_i - x_j x, \quad j=1,\ldots,d.
\end{align*}
Let us compute the corresponding It\^{o} dynamics. For this purpose, let $x \in \bbr^d$ be arbitrary. Then we have
\begin{align*}
\partial_i \sigma^j(x) = - \delta_{ij} x - x_j e_i, \quad i,j=1,\ldots,d,
\end{align*}
and hence, for each $j=1,\ldots,d$ we obtain
\begin{align*}
D \sigma^j(x) \sigma^j(x) &= \sum_{i=1}^d \sigma_{ij}(x) \partial_i \sigma^j(x) = -\sum_{i=1}^d ( \delta_{ij} - x_i x_j ) (\delta_{ij} x + x_j e_i)
\\ &= -\sum_{i=1}^d ( \delta_{ij} x + \delta_{ij} x_j e_i - x_i x_j \delta_{ij} x - x_i x_j^2 e_i )
\\ &= - x - x_j e_j + x_j^2 x + x_j^2 \sum_{i=1}^d x_i e_i = - x - x_j e_j + 2 x_j^2 x.
\end{align*}
Therefore, we have
\begin{align*}
\sum_{j=1}^d D \sigma^j(x) \sigma^j(x) = - d x - x + 2 \| x \|^2 x = - ( d + 1 - 2 \| x \|^2 ) x.
\end{align*}
In particular, for $x \in \bbs^{d-1}$ we obtain
\begin{align*}
\frac{1}{2} \sum_{j=1}^d D \sigma^j(x) \sigma^j(x) = -\frac{d-1}{2} x.
\end{align*}
Therefore, we may alternatively consider the $\bbr^d$-valued It\^{o} SDE
\begin{align}\label{SDE-unit-shpere}
\left\{
\begin{array}{rcl}
dX_t & = & -\frac{d-1}{2} X_t dt + ( \Id - X_t X_t^{\top} ) dW_t
\\ X_0 & = & x_0,
\end{array}
\right.
\end{align}
cf., for example, equation (2.1) in \cite{Mijatovic}. Using Corollary \ref{cor-unit-sphere}, we will show that the unit sphere $\bbs^{d-1}$ is globally invariant for the SDE (\ref{SDE-unit-shpere}). First, note that the SDE (\ref{SDE-unit-shpere}) is of the form (\ref{SDE-b-sigma}). Let $O \subset \bbr^d$ be the open set $O = \{ x \in \bbr^d : \| x \| < 2 \}$. By virtue of Lemma \ref{lemma-fct-comp-supp} there exist $b_i \in \cals_q(\bbr^d)$, $i=1,\ldots,d$ such that
\begin{align*}
b(x) = -\frac{d-1}{2} x, \quad x \in O,
\end{align*}
where $b = (b_i)_{i=1,\ldots,d}$, and there exist $\sigma_{ij} \in \cals_q(\bbr^d)$, $i,j = 1,\ldots,d$ such that
\begin{align*}
\sigma(x) = \Id - x x^{\top}, \quad x \in O,
\end{align*}
where $\sigma = (\sigma_{ij})_{i,j=1,\ldots,d}$. Hence, we may assume that the coefficients $b : \bbr^d \to \bbr^d$ and $\sigma : \bbr^d \to \bbr^{d \times d}$ of the SDE (\ref{SDE-b-sigma}) are given by these mappings with components from $\cals_q(\bbr^d)$. Now, let $x \in \bbs^{d-1}$ be arbitrary. Since the matrix $\sigma(x)$ is symmetric, taking into account the identification $\bbr^d \cong \bbr^{d \times 1}$ we have
\begin{align*}
\sigma(x)^{\top} x = \sigma(x) x = (\Id - x x^{\top})x = x - x x^{\top} x = x (1 - x^{\top} x) = x (1 - \| x \|^2) = 0.
\end{align*}
Furthermore, since $x^{\top} x = \| x \|^2 = 1$, we obtain
\begin{align*}
\sigma(x) \sigma(x)^{\top} &= \sigma(x)^2 = (\Id - x x^{\top})^2 = \Id - 2 x x^{\top} + x x^{\top} x x^{\top}
\\ &= \Id - x x^{\top} = \sigma(x).
\end{align*}
Therefore, we have
\begin{align*}
\tr \big( \sigma(x) \sigma(x)^{\top} \big) = \tr \big( \Id - x x^{\top} \big) = d - \| x \|^2 = d - 1,
\end{align*}
and hence
\begin{align*}
\la x,b(x) \ra + \frac{1}{2} \tr \big( \sigma(x) \sigma(x)^{\top} \big) = -\frac{d-1}{2} + \frac{d-1}{2} = 0.
\end{align*}
Consequently, by Corollary \ref{cor-unit-sphere} the unit sphere $\bbs^{d-1}$ is globally invariant for the SDE (\ref{SDE-unit-shpere}).
\end{example}

\begin{remark}
Suppose that the submanifold $\caln$ is globally invariant for the SDE (\ref{SDE-b-sigma}), and that its complement $\bbr^d \setminus \caln$ consists of two connected components $\caln_1$ and $\caln_2$. Then the two sets $\caln_1 \cup \caln$ and $\caln_2 \cup \caln$ are also globally invariant for the SDE (\ref{SDE-b-sigma}), and the submanifold $\caln$ is an absorbing set in the sense that for each $y_0 \in \bbr^d$ we have $Y \in \caln$ up to an evanescent set on $\IL \tau,\infty \IL$, where $Y$ denotes any weak solution to the SDE (\ref{SDE-b-sigma}) with $Y_0 = y_0$, and $\tau$ denotes the stopping time $\tau := \inf \{ t \in \bbr_+ : Y_t \in \caln \}$. Some examples for the submanifold $\caln$ are as follows:
\begin{itemize}
\item Let $\caln$ be a $(d-1)$-dimensional affine hyperplane. Then there are $\eta \in \bbr^d$ and $b \in \bbr$ such that
\begin{align*}
\caln = \{ x \in \bbr^d : \la x,\eta \ra = b \}.
\end{align*}
By Corollary \ref{cor-affine} and Proposition \ref{prop-M-N} the affine hyperplane $\caln$ is globally invariant for the SDE (\ref{SDE-b-sigma}) if and only if conditions (\ref{conditions-1}) and (\ref{conditions-2}) are fulfilled. Its complement $\bbr^d \setminus \caln$ consists of the two connected components
\begin{align*}
\caln_1 = \{ x \in \bbr^d : \la x,\eta \ra < b \} \quad \text{and} \quad \caln_2 = \{ x \in \bbr^d : \la x,\eta \ra > b \}.
\end{align*}
\item Let $\caln = \bbs^{d-1}$ be the unit sphere in $\bbr^d$. By Corollary \ref{cor-unit-sphere} the unit sphere $\caln$ is globally invariant for the SDE (\ref{SDE-b-sigma}) if and and only if conditions (\ref{sphere-b}) and (\ref{sphere-sigma}) are fulfilled for each $x \in \caln$. Its complement $\bbr^d \setminus \caln$ consists of the two connected components
\begin{align*}
\caln_1 = \{ x \in \bbr^d : \| x \| < 1 \} \quad \text{and} \quad \caln_2 = \{ x \in \bbr^d : \| x \| > 1 \}.
\end{align*}
\item More generally, let $\caln$ be a $(d-1)$-dimensional submanifold of $\bbr^d$ which is compact and connected. By the Jordan-Brouwer separation theorem its complement $\bbr^d \setminus \caln$ consists of two connected components $\caln_1$ and $\caln_2$.
\end{itemize}
\end{remark}

Now, we approach the proof of Theorem \ref{thm-nullstellen}. Recall that $\psi \in C^2(\bbr^d;H)$ denotes the orbit map $\psi = \xi_{\Phi}$ with $\Phi = \delta_0$. Thus, we have $\psi(x) = \delta_x$ for all $x \in \bbr^d$, and by Proposition \ref{prop-psi-measure} the mapping $\psi$ is a $C^2$-immersion, and $\psi : \bbr^d \to \psi(\bbr^d)$ is a homeomorphism. By Examples \ref{ex-manifolds-HS-spaces} the set
\begin{align}\label{K-psi-O}
\calk := \psi(O) = \{ \delta_x : x \in O \}
\end{align}
is a $d$-dimensional $(G,H_0,H)$-submanifold of class $C^2$ with one chart. Furthermore, by Examples \ref{ex-manifolds-HS-spaces} the set
\begin{align}\label{M-equal-psi-N}
\calm := \psi(\caln) = \{ \delta_x : x \in \caln \}
\end{align}
is a $(d-n)$-dimensional $(G,H_0,H)$-submanifold of class $C^2$, which is induced by $(\psi,\caln)$, and obviously we have $\calm \subset \calk$.

\begin{lemma}\label{lemma-K-inv}
The submanifold $\calk$ is locally invariant for the SPDE (\ref{SPDE}).
\end{lemma}

\begin{proof}
Since the open subset $O$ is locally invariant for the SDE (\ref{SDE-b-sigma}), this is an immediate consequence of Theorem \ref{thm-M-N}.
\end{proof}

For the next auxiliary result note that $f_k \in \cals_{-p}(\bbr^d)$ for all $k=1,\ldots,n$.

\begin{lemma}\label{lemma-tangent-L-M}
The following statements are true:
\begin{enumerate}
\item We have $\calm = \calk \cap \bigcap_{k=1}^n \ker( \la f_k, \cdot \ra )$.

\item For each $y \in \calm$ we have $T_y \calm = T_y \calk \cap \bigcap_{k=1}^n \ker( \la f_k, \cdot \ra )$.
\end{enumerate}
\end{lemma}

\begin{proof}
Let $y \in \calk$ be arbitrary. Setting $x := \psi^{-1}(y) \in O$, we have $y = \delta_x$. We have $y \in \calm$ if and only if $x \in \caln$, and by (\ref{N-zeros}) we have $x \in \caln$ if and only if $f_k(x) = 0$ for all $k=1,\ldots,n$. This is equivalent to $\la f_k,\delta_x \ra = 0$ for all $k=1,\ldots,n$, which is satisfied if and only if $y \in \bigcap_{k=1}^n \ker( \la f_k, \cdot \ra )$, proving the first statement.

For the proof of the second statement, let $y \in \calm$ be arbitrary. Setting $x := \psi^{-1}(y) \in \caln$, we have $y = \delta_x$. By Lemma \ref{lemma-tang-embedded} we have
\begin{align*}
T_y \calm = D \psi(x) T_x \caln \subset D \psi(x) \bbr^d = T_y \calk.
\end{align*}
Let $w \in T_y \calk$ be arbitrary. There is a unique vector $v \in \bbr^d$ such that $w = D \psi(x) v$. We have $w \in T_y \calm$ if and only if $v \in T_x \caln$. By Proposition \ref{prop-orbit-in-Sp} and by duality, for each $k=1,\ldots,n$ we have
\begin{align*}
\la f_k,w \ra &= \la f_k, D \psi(x) v \ra = - \bigg\la f_k, \sum_{i=1}^d v_i \partial_i y \bigg\ra = - \sum_{i=1}^d v_i \la f_k, \partial_i y \ra
\\ &= \sum_{i=1}^d v_i \la \partial_i f_k, y \ra = \sum_{i=1}^d v_i \partial_i f_k(x) = D f_k(x) v.
\end{align*}
Therefore, by (\ref{N-zeros-tang}) we have $v \in T_x \caln$ if and only if $w \in \bigcap_{k=1}^n \ker( \la f_k, \cdot \ra )$, completing the proof.
\end{proof}

Now, we are ready to provide the proof of Theorem \ref{thm-nullstellen}.

\begin{proof}[Proof of Theorem \ref{thm-nullstellen}]
(i) $\Rightarrow$ (ii): Let $x \in \caln$ be arbitrary. There exist a global weak solution $X$ to the SDE (\ref{SDE}) with $X_0 = x$ and a positive stopping time $\tau > 0$ such that $X^{\tau} \in \caln$ up to an evanescent set. Let $k = 1,\ldots,n$ be arbitrary. By It\^{o}'s formula (see \cite[Thm. 2.3.1]{fillnm}) we have $\bbp$-almost surely
\begin{align*}
f_k(X_{t \wedge \tau}) &= f_k(x) + \int_0^{t \wedge \tau} (\call f_k)(X_s) ds + \int_0^{t \wedge \tau} (\cala f_k)(X_s) dW_s, \quad t \in \bbr_+.
\end{align*}
where the continuous mapping $\cala f_k : \bbr^d \to \ell^2(\bbr^d)$ is given by $\cala f_k = (\cala^j f_k)_{j \in \bbn}$. Noting that $f_k(X^{\tau}) = 0$, we deduce (\ref{cond-1}) and (\ref{cond-2}).

\noindent(ii) $\Rightarrow$ (i): Our strategy is to prove that the submanifold $\calm$ defined in (\ref{M-equal-psi-N}) is locally invariant for the SPDE (\ref{SPDE}) with coefficients (\ref{def-L}) and (\ref{def-A}), and then to apply Theorem \ref{thm-M-N} in order to deduce that the submanifold $\caln$ is locally invariant for the SDE (\ref{SDE-b-sigma}). First, note that for all $y \in \calm$ and all $k=1,\ldots,n$ we have
\begin{align}\label{ker-A}
\la f_k,A^j(y) \ra &= 0, \quad j \in \bbn,
\\ \label{ker-L} \la f_k,L(y) \ra &= 0.
\end{align}
Indeed, let $y \in \calm$ be arbitrary. Setting $x := \psi^{-1}(y) \in \caln$, we have $y = \delta_x$. Thus, taking into account the definitions (\ref{def-L}) and (\ref{def-A}) of the coefficients, by duality for all $k=1,\ldots,n$ we obtain
\begin{align*}
\la f_k,A^j(y) \ra &= - \sum_{i=1}^d \la \sigma_{ij},y \ra \la f_k,\partial_i y \ra = \sum_{i=1}^d \la \sigma_{ij},y \ra \la \partial_i f_k, y \ra
\\ &= \sum_{i=1}^d \sigma_{ij}(x) \partial_i f_k(x) = \cala^j f_k(x) = 0, \quad j \in \bbn
\end{align*}
as well as
\begin{align*}
\la f_k,L(y) \ra &= \frac{1}{2} \sum_{i,j=1}^d ( \langle \sigma,y \rangle \langle \sigma,y \rangle^{\top} )_{ij} \la f_k, \partial_{ij}^2 y \ra - \sum_{i=1}^d \langle b_i,y \rangle \la f_k, \partial_i y \ra
\\ &= \frac{1}{2} \sum_{i,j=1}^d ( \sigma(x) \sigma(x)^{\top} )_{ij} \partial_{ij}^2 f_k(x) + \sum_{i=1}^d b_i(x) \partial_i f_k(x) = \call f_k(x) = 0.
\end{align*}
Now, let $y \in \calm$ be arbitrary. Setting $x := \psi^{-1}(y) \in \caln$, we have $y = \delta_x$. Let $\varphi : V \to W \cap \caln$ be a local parametrization around $x := \psi^{-1}(y) \in \caln$ with $W \subset O$. By Lemma \ref{lemma-induced} there exists an open neighborhood $U \subset H$ of $y$ such that $\phi := \psi \circ \varphi : V \to U \cap \calm$ is a local parametrization around $y$. Hence, the mapping $\psi|_{W \cap \caln} : W \cap \caln \to U \cap \calm$ is a homeomorphism, and noting (\ref{K-psi-O}) the mapping $\psi|_{O} : O \to \calk$ is a homeomorphism. Since $W \subset O$, it follows that the mapping $\psi|_W : W \to U \cap \calk$ is a local parametrization of $\calk$ around $y$. By Lemma \ref{lemma-K-inv} the submanifold $\calk$ is locally invariant for the SPDE (\ref{SPDE}). Therefore, by Proposition \ref{prop-inv-para-2} there are continuous mappings $\bar{b} : W \to \bbr^d$ and $\bar{\sigma} : W \to \ell^2(\bbr^d)$ which are the unique solutions of the equations
\begin{align*}
A^j|_{U \cap \calk} &= \psi_* \bar{\sigma}^j, \quad j \in \bbn,
\\ L|_{U \cap \calk} &= \psi_* \bar{b} + \frac{1}{2} \sum_{j=1}^{\infty} \psi_{**}(\bar{\sigma}^j,\bar{\sigma}^j).
\end{align*}
In particular, we have
\begin{align}\label{tang-U-K}
A^j|_{\calm} \in \Gamma(T \calk_U), \quad j \in \bbn,
\end{align}
where $\calk_U := U \cap \calk$. From these equations, it follows that
\begin{align}\label{restrict-1}
A^j|_{U \cap \calm} &= \psi_* \bar{\sigma}^j|_{W \cap \caln}, \quad j \in \bbn,
\\ \label{restrict-2} L|_{U \cap \calm} &= \psi_* \bar{b}|_{W \cap \caln} + \frac{1}{2} \sum_{j=1}^{\infty} \psi_{**}(\bar{\sigma}^j |_{W \cap \caln},\bar{\sigma}^j |_{W \cap \caln}).
\end{align}
Let $j \in \bbn$ be arbitrary. Noting (\ref{ker-A}) and (\ref{tang-U-K}), by Lemma \ref{lemma-tangent-L-M} we obtain
\begin{align*}
A^j|_{\calm} \in \Gamma(T \calm_U), 
\end{align*}
where $\calm_U := U \cap \calm$. Therefore, taking into account (\ref{restrict-1}), by Lemma \ref{lemma-tang-embedded} we deduce that
\begin{align*}
\bar{\sigma}^j|_{W \cap \caln} \in \Gamma(T \caln_W), 
\end{align*}
where $\caln_W := W \cap \caln$. Hence there is a continuous mappings $a : V \to \ell^2(\bbr^m)$ whose components are the unique solutions to the equations
\begin{align}\label{restrict-3}
\bar{\sigma}^j|_{W \cap \caln} = \varphi_* a^j, \quad j \in \bbn.
\end{align}
Taking into account Lemma \ref{lemma-push-chain-rule}, by (\ref{restrict-1}) and (\ref{restrict-3}) we obtain
\begin{align*}
A^j|_{U \cap \calm} = \psi_* \bar{\sigma}^j|_{W \cap \caln} = \psi_* \varphi_* a^j = \phi_* a^j, \quad j \in \bbn.
\end{align*}
Furthermore, taking into account Lemma \ref{lemma-push-chain-rule}, by (\ref{restrict-2}) and (\ref{restrict-3}) we have
\begin{equation}\label{restrict-4}
\begin{aligned}
L|_{U \cap \calm} &= \psi_* \bar{b}|_{W \cap \caln} + \frac{1}{2} \sum_{j=1}^{\infty} \psi_{**}(\varphi_* a^j,\varphi_* a^j)
\\ &= \psi_* \bar{b}|_{W \cap \caln} + \frac{1}{2} \sum_{j=1}^{\infty} \big( \phi_{**}(a^j,a^j) - \psi_* \varphi_{**}(a^j,a^j) \big)
\\ &= \psi_* \bigg( \bar{b}|_{W \cap \caln} - \frac{1}{2} \sum_{j=1}^{\infty} \varphi_{**}(a^j,a^j) \bigg) + \frac{1}{2} \sum_{j=1}^{\infty} \phi_{**}(a^j,a^j).
\end{aligned}
\end{equation}
Taking into account Lemma \ref{lemma-tangent-L-M}, we have $\phi(V) \subset \bigcap_{k=1}^n \ker(\la f_k,\cdot \ra)$, and hence
\begin{align*}
(\phi_{**}(a^j,a^j))(U \cap \calm) \subset \bigcap_{k=1}^n \ker(\la f_k,\cdot \ra) \quad \text{for all $j \in \bbn$.}
\end{align*}
Thus, noting (\ref{ker-L}), by Lemma \ref{lemma-tangent-L-M} we obtain
\begin{align*}
L|_{U \cap \calm} - \frac{1}{2} \sum_{j=1}^{\infty} \phi_{**}(a^j,a^j) \in \Gamma(T \calm_U).
\end{align*}
Therefore, by Lemma \ref{lemma-tang-embedded} we deduce that
\begin{align*}
\bar{b}|_{W \cap \caln} - \frac{1}{2} \sum_{j=1}^{\infty} \varphi_{**}(a^j,a^j) \in \Gamma (T \caln_W).
\end{align*}
Hence, there is a continuous mapping $\ell : V \to \bbr^m$ which is the unique solution to the equation
\begin{align}\label{restrict-5}
\bar{b}|_{W \cap \caln} - \frac{1}{2} \sum_{j=1}^{\infty} \varphi_{**}(a^j,a^j) = \varphi_* \ell.
\end{align}
Therefore, using Lemma \ref{lemma-push-chain-rule}, by (\ref{restrict-4}) and (\ref{restrict-5}) we obtain
\begin{align*}
L|_{U \cap \calm} - \frac{1}{2} \sum_{j=1}^{\infty} \phi_{**}(a^j,a^j) &= \psi_* \bigg( \bar{b}|_{W \cap \caln} - \frac{1}{2} \sum_{j=1}^{\infty} \varphi_{**}(a^j,a^j) \bigg) = \psi_* \varphi_* \ell = \phi_* \ell.
\end{align*}
Now, by Proposition \ref{prop-inv-para-2} we deduce that the submanifold $\calm$ is locally invariant for the SPDE (\ref{SPDE}). Consequently, by Theorem \ref{thm-M-N} it follows that the submanifold $\caln$ is locally invariant for the SDE (\ref{SDE-b-sigma}).
\end{proof}

\begin{appendix}

\section{Multi-parameter strongly continuous groups}\label{app-groups}

In this appendix we provide the required results about multi-parameter strongly continuous groups. For this purpose, we begin with reviewing one-parameter strongly continuous groups. Let $H$ be a separable Hilbert space. A family $T = (T(t))_{t \in \bbr}$ of continuous linear operators $T(t) \in L(H)$ is called a \emph{strongly continuous group} (or a \emph{$C_0$-group}) on $H$ if the following conditions are fulfilled:
\begin{enumerate}
\item $T(0) = \Id$.

\item We have $T(t+s) = T(t) T(s)$ for all $t,s \in \bbr$.

\item For each $x \in H$ the \emph{orbit map}
\begin{align*}
\xi_x : \bbr \to H, \quad \xi_x(t) := T(t) x
\end{align*}
is continuous.
\end{enumerate}
Let $T$ be a $C_0$-group on $H$. Recall that the generator $A : H \supset D(A) \to H$ is the operator
\begin{align*}
Ax := \dot{\xi}_x(0) = \lim_{h \to 0} \frac{T(h) x - x}{h}, \quad x \in D(A),
\end{align*}
where the domain is given by
\begin{align*}
D(A) := \{ x \in H : \xi_x \in C^1(\bbr;H) \}.
\end{align*}

\begin{remark}\label{rem-growth}
According to the generation theorem for groups (see \cite[p. 79]{Engel-Nagel}) an operator $(A,D(A))$ generates a $C_0$-group $T$ if and only if $(A,D(A))$ and $(-A,D(A))$ generate $C_0$-semigroups $T_+$ and $T_-$, and in this case there are $M \geq 1$ and $w \in \bbr$ such that we have the growth estimate
\begin{align*}
\| T(t) \| \leq M e^{w|t|} \quad \text{for all $t \in \bbr$.}
\end{align*}
\end{remark}

\begin{lemma}\cite[Lemma II.1.3]{Engel-Nagel}\label{lemma-T-A-comm}
Let $x \in D(A)$ be arbitrary. Then for each $t \in \bbr$ we have $T(t)x \in D(A)$ and
\begin{align*}
\frac{d}{dt} T(t) x = T(t) Ax = A T(t) x.
\end{align*}
\end{lemma}

Now, let $d \in \bbn$ be a positive integer.

\begin{definition}\label{def-commutative}
A family $T_1,\ldots,T_d$ of $C_0$-groups on $H$ is called \emph{commutative} if for every permutation $\pi : \{ 1,\ldots,d \} \to \{ 1,\ldots,d \}$ and every $t \in \bbr$ we have
\begin{align*}
T_1(t) \circ \ldots \circ T_d(t) = T_{\pi(1)}(t) \circ \ldots \circ T_{\pi(d)}(t). 
\end{align*}
\end{definition}

\begin{lemma}\label{lemma-commutative}
Let $T_1,\ldots,T_d$ be commutative $C_0$-groups. Then for every permutation $\pi : \{ 1,\ldots,d \} \to \{ 1,\ldots,d \}$ and all $t_1,\ldots,t_d \in \bbr$ we have
\begin{align*}
T_1(t_1) \circ \ldots \circ T_d(t_d) = T_{\pi(1)}(t_{\pi(1)}) \circ \ldots \circ T_{\pi(d)}(t_{\pi(d)}). 
\end{align*}
\end{lemma}

\begin{proof}
The proof is analogous to that of Statement I.5.15 on page 44 in \cite{Engel-Nagel}.
\end{proof}

\begin{definition}
A family $T = (T(t))_{t \in \bbr^d}$ of continuous linear operators $T(t) \in L(H)$ is called a \emph{multi-parameter strongly continuous group} (or a \emph{multi-parameter $C_0$-group}) on $H$ if the following conditions are fulfilled:
\begin{enumerate}
\item $T(0) = \Id$.

\item We have $T(t+s) = T(t) T(s)$ for all $t,s \in \bbr^d$.

\item For each $x \in H$ the \emph{orbit map}
\begin{align*}
\xi_x : \bbr^d \to H, \quad \xi_x(t) := T(t) x
\end{align*}
is continuous.
\end{enumerate}
\end{definition}

Let $T = (T(t))_{t \in \bbr^d}$ be a multi-parameter $C_0$-group on $H$. For each $i=1,\ldots,d$ we define the family $T_i = (T_i(t))_{t \in \bbr}$ of continuous linear operators $T_i(t) \in L(H)$ as
\begin{align*}
T_i(t) := T(t e_i), \quad t \in \bbr.
\end{align*}
Then $T_1,\ldots,T_d$ are commutative $C_0$-groups on $H$, and we have
\begin{align}\label{group-repr}
T(t) = T_1(t_1) \circ \ldots \circ T_d(t_d), \quad t \in \bbr^d.
\end{align}

\begin{remark}
As a consequence of Remark \ref{rem-growth}, there are constants $M \geq 1$ and $w \in \bbr$ such that
\begin{align*}
\| T(t) \| \leq M e^{w \| t \|} \quad \text{for all $t \in \bbr^d$.}
\end{align*}
\end{remark}

For the next result, let $A_i : H \supset D(A_i) \to H$, $i=1,\ldots,d$ be closed operators. Setting $A := (A_1,\ldots,A_d)$, we define the \emph{domain}
\begin{align*}
D(A) := \bigcap_{i=1}^d D(A_i).
\end{align*}
Inductively, for each $n \geq 2$ we define the \emph{higher-order domain}
\begin{align*}
D(A^n) := \{ x \in D(A^{n-1}) : A^{\alpha} x \in D(A) \text{ for all } \alpha \in \{ 1,\ldots,d \}^{n-1} \},
\end{align*}
where we use the notation
\begin{align*}
A^{\alpha} x = A_{\alpha_1} \circ \ldots \circ A_{\alpha_{n-1}} x.
\end{align*}
Furthermore, we agree on the notation $D(A^0) := H$. Then for each $n \in \bbn_0$ the space $D(A^n)$ equipped with the graph norm
\begin{align}\label{norm-D-An}
\| x \|_{D(A^n)} = \sqrt{ \sum_{m=0}^n \sum_{\alpha \in \{ 1,\ldots,d \}^m} \| A^{\alpha} x \|_H^2}, \quad x \in D(A^n)
\end{align}
is a separable Hilbert space. Indeed, the completeness is a consequence of the closed graph theorem, and the separability follows from considering the linear isometry
\begin{align*}
D(A^n) \to \bigoplus_{m=0}^n \bigoplus_{\alpha \in \{ 1,\ldots,d \}^m} H, \quad x \mapsto \big( A^{\alpha}x \big)_{m=0,\ldots,n \atop \alpha \in \{ 1,\ldots,d \}^m}.
\end{align*}
Furthermore, by the definition of the norm (\ref{norm-D-An}) the following result is obvious.

\begin{proposition}\label{prop-domain-scale}
For each $n \in \bbn$ the pair $(D(A^n),D(A^{n-1}))$ consists of separable Hilbert spaces with continuous embedding.
\end{proposition}

Now, we return to the multi-parameter group $T$. For each $i=1,\ldots,d$ let $(A_i,D(A_i))$ be the generator of $T_i$. We call $A = (A_1,\ldots,A_d)$ the \emph{generator} of $T$

\begin{lemma}\label{lemma-domain-multi-dense}
The subspace $D(A)$ is dense in $H$.
\end{lemma}

\begin{proof}
This is a consequence of \cite[II.2.7]{Engel-Nagel}.
\end{proof}

\begin{proposition}\label{prop-adjoint-group}
The adjoint group $T^* = (T(t)^*)_{t \in \bbr^d}$ is also a multi-parameter $C_0$-group on $H$, and for each $i = 1,\ldots,d$ the generator of $T_i^*$ is given by $A_i^*$.
\end{proposition}

\begin{proof}
This is a consequence of \cite[Cor. 1.10.6]{Pazy}.
\end{proof}

\begin{proposition}\label{prop-D-An-multi}
Let $n \in \bbn_0$ be arbitrary. Then the following statements are true:
\begin{enumerate}
\item We have
\begin{align*}
D(A^n) = \{ x \in H : \xi_x \in C^n(\bbr^d;H) \}.
\end{align*}
\item Let $x \in D(A^n)$ be arbitrary. Then for all $m \in \bbn_0$ with $m \leq n$ and all $\alpha \in \{ 1,\ldots,d \}^m$ we have
\begin{align*}
\frac{d^{\alpha}}{d t_{\alpha}} T(t)x = A^{\alpha} T(t) x = T(t) A^{\alpha} x, \quad t \in \bbr^d.
\end{align*}
\end{enumerate}
\end{proposition}

\begin{proof}
Taking into account Lemma \ref{lemma-T-A-comm} and the representation (\ref{group-repr}) of the group $T$, this follows by induction on $n$.
\end{proof}

\begin{proposition}\label{prop-degrees-smoothness}
Let $n \in \bbn_0$ and $x \in D(A^n)$ be arbitrary. Then the following statements are true:
\begin{enumerate}
\item We have
\begin{align*}
\xi_x \in \bigcap_{k=0}^n C^k(\bbr^d;D(A^{n-k})).
\end{align*}

\item In particular, we have $\xi_x \in C^n(\bbr^d;H)$, and for each $m \in \bbn_0$ with $m \leq n$ we have
\begin{align*}
D^m \xi_x(t) v = \sum_{\alpha \in \{ 1,\ldots,d \}^m} A^{\alpha} \xi_x(t) v_{\alpha}, \quad t \in \bbr^d \text{ and } v \in (\bbr^d)^m, 
\end{align*}
where we use the notation $v_{\alpha} := v_{\alpha_1} \cdot \ldots \cdot v_{\alpha_m}$.
\end{enumerate}
\end{proposition}

\begin{proof}
This is a consequence of Proposition \ref{prop-D-An-multi}.
\end{proof}

Now, let $V \subset \bbr^m$ be an open set for some $m \in \bbn$, and let $\Phi : V \to \bbr^{m \times d}$ be a continuous mapping. A mapping $\varphi : V \to \bbr^d$ of class $C^1$ is called a \emph{primitive} of $\Phi$ if $\nabla \varphi = \Phi$, where we use the notation
\begin{align*}
\nabla \varphi =
\left(
\begin{array}{ccc}
\partial_1 \varphi_1 & \cdots & \partial_1 \varphi_d
\\ \vdots & \ddots & \vdots
\\ \partial_m \varphi_1 & \cdots & \partial_m \varphi_d
\end{array}
\right).
\end{align*}
In this case, we have
\begin{align*}
\Phi_{ij} = \partial_i \varphi_j \quad \text{for all $i=1,\ldots,m$ and $j=1,\ldots,d$.}
\end{align*}

\begin{proposition}\label{prop-solution-of-PDE}
Let $V \subset \bbr^m$ be an open, connected set, let $t_0 \in V$ be arbitrary, and let $\Phi : V \to \bbr^{m \times d}$ be a continuous mapping having a primitive $\varphi : V \to \bbr^d$. Let $x_0 \in D(A)$ be arbitrary, and let $\phi \in C^1(V;H)$ be a $D(A)$-valued solution to the PDE
\begin{align}\label{PDE-para}
\left\{
\begin{array}{rcl}
\nabla \phi(t) & = & \Phi(t) A \phi(t), \quad t \in V,
\\ \phi(t_0) & = & x_0.
\end{array}
\right.
\end{align}
Then we have $\phi = \xi_{x_0} \circ \varphi$, where the primitive $\varphi : V \to \bbr^d$ is chosen such that $\varphi(t_0) = 0$.
\end{proposition}

\begin{proof}
Let $s \in V$ be arbitrary. The mapping $F : V \to H$ given by
\begin{align*}
F(t) := T( \varphi(s) - \varphi(t) ) \phi(t) = \xi_{\phi(t)}( \varphi(s) - \varphi(t) ), \quad t \in V
\end{align*}
is of class $C^1$. By the Leibniz rule, Proposition \ref{prop-D-An-multi} and the PDE (\ref{PDE-para}), for each $t \in V$ we obtain
\begin{align*}
\nabla F(t) &= \nabla T(\varphi(s) - \varphi(t)) \phi(t)
\\ &= - T(\varphi(s)-\varphi(t)) \Phi(t) A \phi(t) + T(\varphi(s)-\varphi(t)) \nabla \phi(t) = 0.
\end{align*}
Therefore, and since $V$ is connected, by \cite[Cor. 2.4.9]{Abraham} we deduce that $F$ is constant. In particular, we obtain
\begin{align*}
\phi(s) = T(0) \phi(s) = F(s) = F(t_0) = \xi_{\phi(t_0)}( \varphi(s) - \varphi(t_0) ) = \xi_{x_0}(\varphi(s)),
\end{align*}
completing the proof.
\end{proof}

\section{Hermite Sobolev spaces}\label{app-HS-spaces}

In this appendix we provide the required results about Hermite Sobolev spaces. References on this topic are, for example, \cite{Ito}, \cite{Kallianpur-Xiong} and \cite{Bhar-thesis}.

We fix a positive integer $d \in \bbn$. Let $\cals(\bbr^d)$ be the Schwartz space of rapidly decreasing functions. Note that $C_c^{\infty}(\bbr^d) \subset \cals(\bbr^d)$, where $C_c^{\infty}(\bbr^d)$ denotes the space of all $C^{\infty}$-functions $\varphi : \bbr^d \to \bbr$ with compact support.

\begin{lemma}\label{lemma-fct-comp-supp}
Let $f \in C^{\infty}(\bbr^d)$ be arbitrary. For each compact set $K \subset \bbr^d$ there exists a function $g \in C_c^{\infty}(\bbr^d)$ such that $f|_K = g|_K$.
\end{lemma}

\begin{proof}
There exists $N > 0$ such that $K \cap A = \emptyset$, where $A \subset \bbr^d$ denotes the closed set $A := \{ x \in \bbr^d : \| x \| \geq N \}$. By \cite[Thm. II.5.1]{Boothby} there exists a smooth function $\varphi \in C^{\infty}(\bbr^d;[0,1])$ such that $\varphi|_K \equiv 1$ and $\varphi|_A \equiv 0$. We define the product $g := \varphi \cdot f$, which concludes the proof.
\end{proof}

Let $\cals'(\bbr^d)$ be the dual space of the Schwartz space, the so-called space of tempered distributions. Then $(\cals'(\bbr^d),\cals(\bbr^d),\la \cdot,\cdot \ra)$ is a dual pair, where we use the notation
\begin{align}\label{duality-Schwartz}
\la \Phi,\varphi \ra := \Phi(\varphi) \quad \text{for all $\Phi \in \cals'(\bbr^d)$ and 
$\varphi \in \cals(\bbr^d)$.}
\end{align}
Assigning
\begin{align*}
\la \varphi,\cdot \ra_{L^2} \in \cals'(\bbr^d) \quad \text{for each $\varphi \in \cals(\bbr^d)$,}
\end{align*}
by identification we have $\cals(\bbr^d) \subset \cals'(\bbr^d)$, and the bilinear mapping $\la \cdot,\cdot \ra : \cals'(\bbr^d) \times \cals(\bbr^d) \to \bbr$ extends $\la \cdot,\cdot \ra_{L^2} : \cals(\bbr^d) \times \cals(\bbr^d) \to \bbr$. Let $p \in \bbr$ be arbitrary. On the Schwartz space $\cals(\bbr^d)$ we define the inner product $\la \cdot,\cdot \ra_p$ as
\begin{align}\label{norm-S-p}
\la \Phi,\varphi \ra_p := \sum_{k=0}^{\infty} \sum_{|n| = k} (2k + d)^{2p} \la \Phi,h_n \ra_{L^2} \la h_n,\varphi \ra_{L^2}, \quad \Phi,\varphi \in \cals(\bbr^d),
\end{align}
where $(h_n)_{n \in \bbn_0^d}$ are the Hermite functions. The \emph{Hermite Sobolev space} $\cals_p(\bbr^d)$ is defined as the completion of the pre-Hilbert space $(\cals(\bbr^d),\| \cdot \|_p)$. Then $\cals_p(\bbr^d)$ is a separable Hilbert space with orthonormal basis $(h_n^p)_{n \in \bbn_0^d}$, where $h_n^p = (2k+d)^{-p} h_n$ for each $k \in \bbn_0$ and each $n \in \bbn_0$ with $|n| = k$. Again by identification, we have
\begin{align*}
\cals(\bbr^d) \subset \cals_p(\bbr^d) \subset \cals'(\bbr^d).
\end{align*}
For each $p \in \bbr$ the Schwartz space $\cals(\bbr^d)$ is a dense subspace of $\cals_p(\bbr^d)$.

\begin{lemma}\label{lemma-scale-pq}
For all $p,q \in \bbr$ with $p \leq q$ we have $\cals_q(\bbr^d) \subset \cals_p(\bbr^d)$ and
\begin{align*}
\| \Phi \|_p \leq \| \Phi \|_q \quad \text{for each $\Phi \in \cals_q(\bbr^d)$.}
\end{align*}
In particular, the pair $(\cals_q(\bbr^d),\cals_p(\bbr^d))$ consists of separable Hilbert spaces with continuous embedding.
\end{lemma}

\begin{proof}
This is a consequence of (\ref{norm-S-p}).
\end{proof}

Lemma \ref{lemma-scale-pq} shows that the family $(\cals_p(\bbr^d))_{p \in \bbr}$ is a scale of separable Hilbert spaces. Moreover, we have the identities
\begin{align*}
\cals(\bbr^d) = \bigcap_{p \in \bbr} \cals_p(\bbr^d), \quad L^2(\bbr^d) = \cals_0(\bbr^d), \quad \cals'(\bbr^d) = \bigcup_{p \in \bbr} \cals_p(\bbr^d),
\end{align*}
and for every $p \geq 0$ we have the representation
\begin{align*}
\cals_p(\bbr^d) = \{ f \in L^2(\bbr^d) : \| f \|_p < \infty \}.
\end{align*}
The topology on $\cals(\bbr^d)$ is finer than that on $\cals_p(\bbr^d)$ for every $p \in \bbr$, and for any $p \in \bbr$ the topology on $\cals_p(\bbr^d)$ is finer than that on $\cals'(\bbr^d)$. Furthermore, for each $p \in \bbr$ the space $\cals_{-p}(\bbr^d)$ is dual to $\cals_p(\bbr^d)$. More precisely, we have the following result. 

\begin{lemma}\label{lemma-duality-pairing}
For each $p \in \bbr$ there is a unique continuous bilinear mapping ${}_{-p} \la \cdot,\cdot \ra_p : \cals_{-p}(\bbr^d) \times \cals_p(\bbr^d) \to \bbr$ extending $\la \cdot,\cdot \ra_{L^2} : \cals(\bbr^d) \times \cals(\bbr^d) \to \bbr$. Furthermore, the following statements are true:
\begin{enumerate}
\item For each $p \in \bbr$ the triplet $(\cals_{-p}(\bbr^d),\cals_p(\bbr^d),\la \cdot,\cdot \ra)$ is a dual pair.

\item For each $p \in \bbr$ we have
\begin{align*}
|{}_{-p} \la \Phi,\varphi \ra_p| \leq \| \Phi \|_{-p} \| \varphi \|_p \quad \text{for all $\Phi \in \cals_{-p}(\bbr^d)$ and $\varphi \in \cals_p(\bbr^d)$.}
\end{align*}
\item For each $p \in \bbr$ we have
\begin{align*}
{}_{-p} \la \Phi,\varphi \ra_p = {}_p\la \varphi,\Phi \ra_{-p} \quad \text{for all $\Phi \in \cals_{-p}(\bbr^d)$ and $\varphi \in \cals_p(\bbr^d)$.}
\end{align*}
\item For all $p,q \in \bbr$ with $p \leq q$ we have
\begin{align*}
{}_{-p} \la \Phi,\varphi \ra_p = {}_{-q}\la \Phi,\varphi \ra_q \quad \text{for all $\Phi \in \cals_{-p}(\bbr^d)$ and $\varphi \in \cals_q(\bbr^d)$.}
\end{align*}
\item For each $p \in \bbr$ we have
\begin{align*}
{}_{-p} \la \Phi,\varphi \ra_p = \la \Phi,\varphi \ra \quad \text{for all $\Phi \in \cals_{-p}(\bbr^d)$ and $\varphi \in \cals(\bbr^d)$,}
\end{align*}
where $\la \cdot,\cdot \ra$ denotes the dual pair from (\ref{duality-Schwartz}).
\end{enumerate}
\end{lemma}

\begin{proof}
Using the ONBs $(h_n^{-p})_{n \in \bbn_0^d}$ and $(h_n^p)_{n \in \bbn_0^d}$ of $\cals_{-p}(\bbr^d)$ and $\cals_p(\bbr^d)$, we obtain
\begin{align*}
|\la \Phi,\varphi \ra_{L^2}| \leq \| \Phi \|_{-p} \| \varphi \|_p \quad \text{for all $\Phi,\varphi \in \cals(\bbr^d)$,}
\end{align*}
which completes the proof. Note that for the last two statements we also use Lemma \ref{lemma-scale-pq} and the statements preceding this result.
\end{proof}

\begin{remark}\label{rem-dual-pairing}
In the sequel, we will simply write $\la \Phi,\varphi \ra$ whenever $\Phi \in \cals_{-p}(\bbr^d)$ and $\varphi \in \cals_p(\bbr^d)$ for some $p \in \bbr$, which is justified by the last two statements of Lemma \ref{lemma-duality-pairing}.
\end{remark}

Let $i \in \{ 1,\ldots,d \}$ be arbitrary. We extend the \emph{differential operator} $\partial_i : \cals(\bbr^d) \to \cals(\bbr^d)$ to an operator $\partial_i : \cals'(\bbr^d) \to \cals'(\bbr^d)$ by duality as
\begin{align*}
\la \partial_i \Phi, \varphi \ra := - \la \Phi, \partial_i \varphi \ra \quad \text{for all $\Phi \in \cals'(\bbr^d)$ and $\varphi \in \cals(\bbr^d)$.}
\end{align*}

\begin{lemma}\cite[Lemma 2.11.4]{Bhar-thesis}\label{lemma-diff-continuous}
For each $i=1,\ldots,d$ and each $p \in \bbr$ we have
\begin{align*}
\partial_i|_{\cals_{p+\frac{1}{2}}(\bbr^d)} \in L \big( \cals_{p+\frac{1}{2}}(\bbr^d),\cals_p(\bbr^d) \big).
\end{align*}
\end{lemma}

\begin{lemma}\label{lemma-diff-closed}
For each $i=1,\ldots,d$ and each $p \in \bbr$ the operator $\partial_i : \cals_p(\bbr^d) \supset \cals_{p+\frac{1}{2}}(\bbr^d) \to \cals_p(\bbr^d)$ is densely defined and closed, and we have $\cals(\bbr^d) \subset D(\partial_i^*)$.
\end{lemma}

\begin{proof}
This is a consequence of \cite[Lemma 2.11.5 and Lemma 3.2.1]{Bhar-thesis}.
\end{proof}

Let $i \in \{ 1,\ldots,d \}$ be arbitrary. We extend the \emph{multiplication operator} $M_i : \cals(\bbr^d) \to \cals(\bbr^d)$ given by
\begin{align*}
(M_i \varphi)(x) := x_i \varphi(x) \quad \text{for all $x \in \bbr^d$}
\end{align*}
to an operator $M_i : \cals'(\bbr^d) \to \cals'(\bbr^d)$ by duality as
\begin{align*}
\la M_i \Phi, \varphi \ra := \la \Phi, M_i \varphi \ra \quad \text{for all $\Phi \in \cals'(\bbr^d)$ and $\varphi \in \cals(\bbr^d)$.}
\end{align*}

\begin{lemma}\cite[Example 2.11.9]{Bhar-thesis}
For each $i=1,\ldots,d$ and each $p \in \bbr$ we have
\begin{align*}
M_i|_{\cals_{p+\frac{1}{2}}(\bbr^d)} \in L \big( \cals_{p+\frac{1}{2}}(\bbr^d),\cals_p(\bbr^d) \big).
\end{align*}
\end{lemma}

The \emph{Hermite operator} $\mathbf{H} : \cals'(\bbr^d) \to \cals'(\bbr^d)$ is defined as
\begin{align*}
\mathbf{H} := |x|^2 - \Delta := \sum_{i=1}^d (M_i^2 - \partial_i^2).
\end{align*}

\begin{lemma}\cite[Prop. 3.1]{Rajeev-Than-2003}\label{lemma-Hermite}
For each $p \in \bbr$ the Hermite operator
\begin{align*}
\mathbf{H}|_{\cals_{p+1}(\bbr^d)} \in L \big( \cals_{p+1}(\bbr^d),\cals_p(\bbr^d) \big) 
\end{align*}
is an isometric isomorphism.
\end{lemma}

Let $x \in \bbr^d$ be arbitrary. We extend the \emph{translation operator} $\tau_x : \cals(\bbr^d) \to \cals(\bbr^d)$ given by
\begin{align*}
(\tau_x \varphi)(y) := \varphi(y-x) \quad \text{for all $y \in \bbr^d$}
\end{align*}
to an operator $\tau_x : \cals'(\bbr^d) \to \cals'(\bbr^d)$ by duality as
\begin{align*}
\la \tau_x \Phi, \varphi \ra := \la \Phi, \tau_{-x} \varphi \ra \quad \text{for all $\Phi \in \cals'(\bbr^d)$ and $\varphi \in \cals(\bbr^d)$.}
\end{align*}

\begin{lemma}\label{lemma-tau-properties}
The following statements are true:
\begin{enumerate}
\item For each $p \in \bbr$ there exists a polynomial $P_k$ of degree $k = 2([|p|] + 1)$ such that for all $\Phi \in \cals_p(\bbr^d)$ and $x \in \bbr^d$ we have
\begin{align*}
\| \tau_x \Phi \|_p \leq P_k(|x|) \| \Phi \|_p.
\end{align*}
In particular, we have $\tau_x \Phi \in \cals_p(\bbr^d)$.

\item For each $p \in \bbr$ and every $\Phi \in \cals_p(\bbr^d)$ the map $\bbr^d \to \cals_p(\bbr^d)$, $x \mapsto \tau_x \Phi$ is continuous.

\item For each $x \in \bbr^d$ and each $i=1,\ldots,d$ we have $\tau_x \partial_i = \partial_i \tau_x$.
\end{enumerate}
\end{lemma}

\begin{proof}
(1) See \cite[Thm. 2.1]{Rajeev-Than-2003}. \\
\noindent (2) The continuity follows from the proof of \cite[Prop. 3.1]{Rajeev-Than-2008}. \\
\noindent (3) See \cite[Lemma 2.11.7.iii]{Bhar-thesis}.
\end{proof}

\begin{lemma}\label{lemma-tau-group}
For each $p \in \bbr$ the family $\tau = (\tau_x)_{x \in \bbr^d}$ is a multi-parameter $C_0$-group on $\cals_p(\bbr^d)$.
\end{lemma}

\begin{proof}
This is an immediate consequence of Lemma \ref{lemma-tau-properties}.
\end{proof}

\begin{lemma}\label{lemma-Taylor}
Let $p \in \bbr$, $\Phi \in \cals_{p+\frac{1}{2}}(\bbr^d)$ and $i=1,\ldots,d$ be arbitrary. Then there exists a continuous mapping $R : \bbr \times \bbr \to \cals_p(\bbr^d)$ with $R(x,0) = 0$ for all $x \in \bbr$ such that
\begin{align*}
\tau_{x+h}^i \Phi = \tau_x^i \Phi - h \partial_i \tau_x^i \Phi - h R(x,h), \quad x,h \in \bbr
\end{align*}
in the space $\cals_p(\bbr^d)$, where we use the notation $\tau_x^i := \tau_{x e_i}$ for each $x \in \bbr$.
\end{lemma}

\begin{proof}
let $\varphi \in \cals(\bbr^d)$ be arbitrary. By Taylor's formula, for all $x_0 \in \bbr^d$ we obtain
\begin{align*}
\varphi(x_0+(x+h) e_i) &= \varphi(x_0+x e_i) + h \partial_i \varphi(x_0+x e_i)
\\ &\quad + h \int_0^1 \big( \partial_i \varphi(x_0+(x+t h) e_i) - \partial_i \varphi(x_0+x e_i) \big) dt, \quad x,h \in \bbr.
\end{align*}
Therefore, we obtain the equation
\begin{align*}
\tau_{-(x+h)}^i \varphi = \tau_{-x}^i \varphi + h \tau_{-x}^i \partial_i \varphi + h \int_0^1 \big( \tau_{-(x+th)}^i \partial_i \varphi - \tau_{-x}^i \partial_i \varphi \big) dt, \quad x,h \in \bbr
\end{align*}
in the space $\cals_{-p}(\bbr^d)$, where the integral is an $\cals_{-p}(\bbr^d)$-valued Bochner integral, which is well-defined by virtue of Lemma \ref{lemma-tau-properties}. Now, applying $\la \Phi,\cdot \ra$ we obtain
\begin{align*}
\la \tau_{x+h}^i \Phi,\varphi \ra &= \la \tau_x^i \Phi,\varphi \ra - \la h \partial_i \tau_x^i \Phi, \varphi \ra
\\ &\quad - \bigg\la h \int_0^1 \big( \partial_i \tau_{x+th}^i \Phi - \partial_i \tau_x^i \Phi \big) dt, \varphi \bigg\ra, \quad x,h \in \bbr,
\end{align*}
where the integral is a $\cals_p(\bbr^d)$-valued Bochner integral, which is well-defined by virtue of Lemma \ref{lemma-tau-properties}. The mapping $R : \bbr \times \bbr \to \cals_p(\bbr^d)$ defined as
\begin{align*}
R(x,h) := \int_0^1 \big( \partial_i \tau_{x+th}^i \Phi - \partial_i \tau_x^i \Phi \big) dt, \quad x,h \in \bbr
\end{align*}
is continuous by Lemma \ref{lemma-tau-properties} and Lebesgue's dominated convergence theorem. Furthermore, we have $R(x,0) = 0$ for all $x \in \bbr$. Since $\varphi \in \cals(\bbr^d)$ was arbitrary, the claimed identity follows.
\end{proof}

For each $x \in \bbr^d$ we define the \emph{Dirac distribution} $\delta_x \in \cals'(\bbr^d)$ as
\begin{align*}
\la \delta_x, \varphi \ra := \varphi(x) \quad \text{for all $\varphi \in \cals(\bbr^d)$.}
\end{align*}

\begin{lemma}\label{lemma-delta-distr}
The following statements are true:
\begin{enumerate}
\item We have $\tau_x \delta_y = \delta_{x+y}$ for all $x,y \in \bbr^d$.

\item In particular, we have $\tau_x \delta_0 = \delta_x$ for all $x \in \bbr^d$.

\item For each $p < -\frac{d}{4}$ and every $x \in \bbr^d$ we have $\delta_x \in \cals_p(\bbr^d)$.

\item For each $p < -\frac{d}{4}$ the function $\bbr^d \to \cals_p(\bbr^d)$, $x \mapsto \delta_x$ is continuous and we have $\lim_{\| x \| \to \infty} \| \delta_x \|_p = 0$.
\end{enumerate}
\end{lemma}

\begin{proof}
This is a consequence of \cite[Lemma 2.11.15]{Bhar-thesis}, \cite[Thm. 4.1]{Rajeev-Than-2008} and Lemma \ref{lemma-tau-properties}.
\end{proof}

More generally, let $\mu$ be a finite signed measure on $(\bbr^d,\calb(\bbr^d))$. Then we define a tempered distribution, again denoted by $\mu$, by duality as
\begin{align*}
\langle \mu,\varphi \rangle := \int_{\bbr^d} \varphi(y) \mu(dy) \quad \text{for all $\varphi \in \cals(\bbr^d)$.}
\end{align*}
Note that $\mu \in \cals'(\bbr^d)$, because $|\langle \mu,\varphi \rangle| \leq \| \varphi \|_{\infty} |\mu|$ for all $\varphi \in \cals(\bbr^d)$.

\begin{lemma}\label{lemma-distr-int-ex}
Let $\mu$ be a finite signed measure on $(\bbr^d,\calb(\bbr^d))$. Then the following statements are true:
\begin{enumerate}
\item We have
\begin{align*}
\langle \tau_x \mu,\varphi \rangle = \int_{\bbr^d} \varphi(y+x) \mu(dy) \quad \text{for all $x \in \bbr^d$ and $\varphi \in \cals(\bbr^d)$.}
\end{align*}

\item For each $p < -\frac{d}{4}$ we have $\mu \in \cals_p(\bbr^d)$.

\item For each $p < -\frac{d}{4}$ the function $\bbr^d \to \cals_p(\bbr^d)$, $x \mapsto \tau_x \mu$ is continuous and bounded.
\end{enumerate}
\end{lemma}

\begin{proof}
Let $x \in \bbr^d$ and $\varphi \in \cals(\bbr^d)$ be arbitrary. By duality we have
\begin{align*}
\langle \tau_x \mu,\varphi \rangle = \langle \mu, \tau_{-x} \varphi \rangle = \int_{\bbr^d} \varphi(y+x) \mu(dy).
\end{align*}
Now, let $p < -\frac{d}{4}$ be arbitrary. By Lemma \ref{lemma-delta-distr} the function $\bbr^d \to \cals_p(\bbr^d)$, $x \mapsto \delta_x$ is continuous and bounded. Therefore, the Bochner integral
\begin{align*}
\Phi := \int_{\bbr^d} \delta_y \mu(dy) \in \cals_p(\bbr^d)
\end{align*}
is well-defined. For every $\varphi \in \cals(\bbr^d)$ we obtain
\begin{align*}
\langle \Phi,\varphi \rangle = \int_{\bbr^d} \langle \delta_y, \varphi \rangle \mu(dy) = \int_{\bbr^d} \varphi(y) \mu(dy).
\end{align*}
Therefore, we have $\Phi = \mu$, and hence $\mu \in \cals_p(\bbr^d)$. Noting that
\begin{align*}
\tau_x \mu = \int_{\bbr^d} \delta_{y+x} \mu(dy) \quad \text{for all $x \in \bbr^d$,}
\end{align*}
the continuity and the boundedness of $x \mapsto \tau_x \mu$ follow from Lemma \ref{lemma-delta-distr} and Lebesgue's dominated convergence theorem.
\end{proof}

Let $f : \bbr^d \to \bbr$ be a polynomial of several variables. Then we define a tempered distribution, again denoted by $f$, by duality as
\begin{align*}
\la f,\varphi \ra := \int_{\bbr^d} f(y) \varphi(y) dy \quad \text{for all $\varphi \in \cals(\bbr^d)$.}
\end{align*}
Note that $f \in \cals'(\bbr^d)$, because the polynomial $f$ is slowly increasing.

\begin{lemma}\cite[Example 2.11.18]{Bhar-thesis}\label{lemma-polynomial}
Let $f : \bbr^d \to \bbr$ be a polynomial of several variables with degree $n \in \bbn_0$. Then we have $f \in \cals_p(\bbr^d)$ for each $p < -\frac{d}{4} - \frac{n}{2}$.
\end{lemma}

\begin{lemma}\label{lemma-pol-lin-growth}
Let $p < -\frac{d}{4} - \frac{1}{2}$ be arbitrary, and let $f : \bbr^d \to \bbr$ be a polynomial of several variables with $\deg(f) \leq 1$. Then there is a constant $K > 0$ such that
\begin{align*}
\| \tau_x f \|_p \leq K ( 1 + \| x \| ), \quad x \in \bbr^d. 
\end{align*}
\end{lemma}

\begin{proof}
There are $c_0,c_1,\ldots,c_d \in \bbr$ such that
\begin{align*}
f(y) = c_0 + \sum_{i=1}^d c_i y_i, \quad y \in \bbr^d.
\end{align*}
For each $x \in \bbr^d$ we have
\begin{align*}
\tau_x f = f - \bigg( \sum_{i=1}^d c_i x_i \bigg) \bbI.
\end{align*}
Indeed, for all $\varphi \in \cals(\bbr^d)$ we have
\begin{align*}
\la \tau_x f,\varphi \ra &= \la f, \tau_{-x} \varphi \ra = \int_{\bbr^d} f(y) \varphi(y+x) dy = \int_{\bbr^d} f(y-x) \varphi(y) dy 
\\ &= \int_{\bbr^d} \bigg( c_0 + \sum_{i=1}^d c_i (y_i - x_i) \bigg) \varphi(y) dy = \int_{\bbr^d} \bigg( f(y) - \sum_{i=1}^d c_i x_i \bigg) \varphi(y) dy
\\ &= \bigg\la f - \bigg( \sum_{i=1}^d c_i x_i \bigg) \bbI, \varphi \bigg\ra.
\end{align*}
Therefore, for each $x \in \bbr^d$ we obtain
\begin{align*}
\| \tau_x f \|_p \leq \| f \|_p + \bigg( \sum_{i=1}^d |c_i| \, |x_i| \bigg) \| \bbI \|_p,
\end{align*}
proving the claimed linear growth condition.
\end{proof}

\begin{lemma}\label{lemma-L2-bounded}
Let $b,\varphi \in L^2(\bbr^d)$ be arbitrary. Then there is a constant $K > 0$ such that
\begin{align*}
| \la b, \tau_x \varphi \ra | \leq K, \quad x \in \bbr^d.
\end{align*}
\end{lemma}

\begin{proof}
Recalling that $L^2(\bbr^d) = \cals_0(\bbr^d)$, for each $x \in \bbr^d$ we have
\begin{align*}
| \la b, \tau_x \varphi \ra | \leq \| b \|_{L^2} \| \tau_x \varphi \|_{L^2} = \| b \|_{L^2} \| \varphi \|_{L^2},
\end{align*}
completing the proof.
\end{proof}

For the next result, we recall that $\cals_p(\bbr^d) \subset L^2(\bbr^d)$ for each $p \geq 0$. Also recall that $C_0(\bbr^d)$ denotes the space of all continuous functions $f : \bbr^d \to \bbr$ such that $\lim_{\| x \| \to \infty} f(x) = 0$. Equipped with the supremum norm, this space is a Banach space.

\begin{lemma}\label{lemma-Sp-continuous}
Let $p > \frac{d}{4}$ and $f \in \cals_p(\bbr^d)$ be arbitrary, and define the mapping
\begin{align*}
g : \bbr^d \to \bbr, \quad g(x) := \la \delta_x,f \ra.
\end{align*}
Then we have $g \in C_0(\bbr^d)$ and $f = g$ almost everywhere.
\end{lemma}

\begin{proof}
By Lemma \ref{lemma-delta-distr} we have $g \in C_0(\bbr^d)$. Furthermore, there exists a sequence $(\varphi_n)_{n \in \bbn} \subset \cals(\bbr^d)$ such that $\varphi_n \to f$ in $\cals_p(\bbr^d)$. Thus, we also have $\varphi_n \to f$ in $L^2(\bbr^d)$, and hence there is a subsequence $(n_k)_{k \in \bbn}$ such that $\varphi_{n_k} \to f$ almost everywhere. Therefore, for almost all $x \in \bbr^d$ we obtain
\begin{align*}
f(x) = \lim_{k \to \infty} \varphi_{n_k}(x) = \lim_{k \to \infty} \la \delta_x, \varphi_{n_k} \ra = \la \delta_x,f \ra = g(x),
\end{align*}
completing the proof.
\end{proof}

More generally, for each $k \in \bbn_0$ the space $C_0^k(\bbr^d)$ denotes the space of all $f \in C^k(\bbr^d)$ such that $D^{\alpha} f \in C_0(\bbr^d)$ for each $\alpha \in \bbn_0^d$ with $|\alpha| \leq k$. The space $C_0^k(\bbr^d)$ endowed with the norm
\begin{align*}
\| f \|_{C_0^k(\bbr^d)} = \sum_{|\alpha| \leq k} \| D^{\alpha} f \|_{\infty}
\end{align*}
is a Banach space.

\begin{lemma}\label{lemma-Sp-continuous-k}
Let $k \in \bbn_0$ and $p > \frac{d}{4} + \frac{k}{2}$ as well as $f \in \cals_p(\bbr^d)$ be arbitrary, and define the mapping
\begin{align*}
g : \bbr^d \to \bbr, \quad g(x) := \la \delta_x,f \ra.
\end{align*}
Then the following statements are true:
\begin{enumerate}
\item We have $f = g$ almost everywhere.

\item We have $g \in C_0^k(\bbr^d)$.

\item For each $\alpha \in \bbn_0^d$ with $|\alpha| \leq k$ we have $D^{\alpha} g(x) = \la \delta_x, \partial^{\alpha} f \ra$ for all $x \in \bbr^d$.
\end{enumerate}
\end{lemma}

\begin{proof}
By Lemma \ref{lemma-Sp-continuous} we have $f = g$ almost everywhere. We prove the remaining statements by induction on $k \in \bbn_0$. For $k=0$ these follow from Lemma \ref{lemma-Sp-continuous}. We proceed with the induction step $k-1 \to k$. Let $\beta \in \bbn_0^d$ with $|\beta| = k-1$ be arbitrary, and let $i=1,\ldots,d$ be arbitrary. We set $\alpha := \beta + e_i$. Let $x \in \bbr^d$ be arbitrary. By induction hypothesis we have
\begin{align*}
D^{\beta} g(x) = \la \delta_x, \partial^{\beta} f \ra.
\end{align*}
Hence, for each $h \in \bbr$ with $h \neq 0$ we have
\begin{align*}
\frac{D^{\beta} g(x + h e_i) - D^{\beta} g(x)}{h} = \bigg\la \frac{\delta_{x + h e_i} - \delta_x}{h}, \partial^{\beta} f \bigg\ra.
\end{align*}
Note that $\delta_0 \in \cals_{-p+\frac{k}{2}}(\bbr^d)$. Hence, by Lemma \ref{lemma-Taylor} (applied with $\Phi = \delta_0$ and $x = 0$) we have
\begin{align*}
\delta_{h e_i} - \delta_0 = - h \partial_i \delta_0 - h R(h), \quad h \in \bbr
\end{align*}
in the space $\cals_{-p+\frac{k}{2}-\frac{1}{2}}(\bbr^d)$, where $R : \bbr \to \cals_{-p+\frac{k}{2}-\frac{1}{2}}(\bbr^d)$ is a continuous function such that $R(0) = 0$. Therefore, applying the translation operator $\tau_x$ on both sides, we obtain
\begin{align*}
\delta_{x + h e_i} - \delta_x = - h \partial_i \delta_x - h \tau_x R(h), \quad h \in \bbr
\end{align*}
in the space $\cals_{-p+\frac{k}{2}-\frac{1}{2}}(\bbr^d)$. It follows that
\begin{align*}
\lim_{h \to 0} \frac{\delta_{x + h e_i} - \delta_x}{h} = -\partial_i \delta_x
\end{align*}
in the space $\cals_{-p+\frac{k}{2}-\frac{1}{2}}(\bbr^d)$. Therefore, we obtain
\begin{align*}
D^{\alpha} g(x) = \lim_{h \to 0} \frac{D^{\beta} g(x + h e_i) - D^{\beta} g(x)}{h} = - \la \partial_i \delta_x, \partial^{\beta} f \ra = \la \delta_x, \partial^{\alpha} f \ra,
\end{align*}
and by Lemma \ref{lemma-delta-distr} we deduce that $D^{\alpha} g \in C_0(\bbr^d)$.
\end{proof}

\begin{theorem}[Sobolev embedding theorem for Hermite Sobolev spaces]\label{thm-Ck-versions}
For each $k \in \bbn_0$ and $p > \frac{d}{4} + \frac{k}{2}$ the pair $(\cals_p(\bbr^d), C_0^k(\bbr^d))$ consists of continuously embedded Banach spaces.
\end{theorem}

\begin{remark}
Note that the inclusion $\cals_p(\bbr^d) \subset C_0^k(\bbr^d)$ in Theorem \ref{thm-Ck-versions} is meant in the sense that for each $f \in \cals_p(\bbr^d)$ there exists a version $g \in C_0^k(\bbr^d)$ such that $f = g$ almost everywhere.
\end{remark}

\begin{proof}[Proof of Theorem \ref{thm-Ck-versions}]
Let $f \in \cals_p(\bbr^d)$ be arbitrary, and define the mapping
\begin{align*}
g : \bbr^d \to \bbr, \quad g(x) := \la \delta_x,f \ra.
\end{align*}
By Lemma \ref{lemma-Sp-continuous-k} we have $f = g$ almost everywhere, $g \in C_0^k(\bbr^d)$ and
\begin{align*}
\| g \|_{C_0^k(\bbr^d)} &= \sum_{|\alpha| \leq k} \sup_{x \in \bbr^d} | D^{\alpha} g(x) | = \sum_{|\alpha| \leq k} \sup_{x \in \bbr^d} | \la \delta_x, \partial^{\alpha} f \ra |
\\ &\leq \sum_{|\alpha| \leq k} \sup_{x \in \bbr^d} \| \delta_x \|_{-p+\frac{|\alpha|}{2}} \| \partial^{\alpha} f \|_{p-\frac{|\alpha|}{2}} \leq C \| f \|_p
\end{align*}
with a finite constant $C > 0$, which does not depend on $f$. Note that for the last inequality we have used Lemmas \ref{lemma-diff-continuous} and \ref{lemma-delta-distr}.
\end{proof}

Recall that for $n \in \bbn_0$ the Sobolev space $W^n(\bbr^d)$ consists of all $f \in L^2(\bbr^d)$ such that the weak derivative $D^{\beta} f \in L^2(\bbr^d)$ exists for all $\beta \in \bbn_0^d$ with $|\beta| \leq n$.

\begin{proposition}\label{prop-Sobolev-embedding}
Let $m \in \bbn_0$ be arbitrary, and let $f \in W^{2m}(\bbr^d)$ be such that $x^{\alpha} D^{\beta} f \in L^2(\bbr^d)$ for all $\alpha,\beta \in \bbn_0^d$ with $|\alpha| + |\beta| \leq 2m$. Then we have $f \in \cals_m(\bbr^d)$.
\end{proposition}

\begin{proof}
For each $\beta \in \bbn_0^d$ with $|\beta| \leq 2m$ we have $D^{\beta} f = \partial^{\beta} f$. Indeed, using Lemma \ref{lemma-duality-pairing} for each $\varphi \in C_c^{\infty}(\bbr^d)$ we obtain
\begin{align*}
\la \partial^{\beta} f, \varphi \ra &= (-1)^{|\beta|} \la f, \partial^{\beta} \varphi \ra = (-1)^{|\beta|} \la f, D^{\beta} \varphi \ra = (-1)^{|\beta|} \la f, D^{\beta} \varphi \ra_{L^2} 
\\ &= \la D^{\beta} f, \varphi \ra_{L^2} = \la D^{\beta} f, \varphi \ra.
\end{align*}
Moreover, we have
\begin{align*}
\mathbf{H}^m f = \sum_{|\alpha|+|\beta| \leq 2m} C_{\alpha \beta} x^{\alpha} \partial^{\beta} f
\end{align*}
with suitable constants $C_{\alpha \beta} \in \bbr$ for all $\alpha,\beta \in \bbn_0^d$ with $|\alpha|+|\beta| \leq 2m$. By Lemma \ref{lemma-Hermite} we obtain
\begin{align*}
\| f \|_m = \| \mathbf{H}^m f \|_{L^2} \leq \sum_{|\alpha|+|\beta| \leq 2m} C_{\alpha \beta} \| x^{\alpha} \partial^{\beta} f \|_{L^2} < \infty,
\end{align*}
completing the proof.
\end{proof}

\begin{corollary}\label{cor-Sobolev-embedding}
Let $m \in \bbn_0$ be arbitrary. If $f \in W^{2m}(\bbr^d)$ has compact support, then we have $f \in \cals_m(\bbr^d)$.
\end{corollary}

\begin{proof}
This is an immediate consequence of Proposition \ref{prop-Sobolev-embedding}.
\end{proof}

\begin{remark}
Let $m,k \in \bbn_0$ be such that $m > \frac{d}{4} + \frac{k}{2}$. Moreover, let $f \in W^{2m}(\bbr^d)$ be arbitrary.
\begin{enumerate}
\item By the classical Sobolev embedding theorem we have $f \in C^k(\bbr^d)$. 

\item If moreover $x^{\alpha} D^{\beta} f \in L^2(\bbr^d)$ for all $\alpha,\beta \in \bbn_0^d$ with $|\alpha| + |\beta| \leq 2m$, then by Proposition \ref{prop-Sobolev-embedding} and the Sobolev embedding theorem for Hermite Sobolev spaces (Theorem \ref{thm-Ck-versions}) we even have $f \in C_0^k(\bbr^d)$.
\end{enumerate}
\end{remark}

\end{appendix}

\end{document}